\documentclass[sn-mathphys,Numbered]{sn-jnl}

\usepackage{graphicx}%
\usepackage{multirow}%
\usepackage{amsmath,amssymb,amsfonts}%
\usepackage{amsthm}%
\usepackage{mathrsfs}%
\usepackage[title]{appendix}%
\usepackage{xcolor}%
\usepackage{textcomp}%
\usepackage{manyfoot}%
\usepackage{booktabs}%
\usepackage{algorithm}%
\usepackage{algorithmicx}
\usepackage{algpseudocode}%
\usepackage{listings}%
\usepackage{caption,subcaption}
\usepackage{hyperref}
\usepackage{morefloats} 
\usepackage{siunitx}

\renewcommand{\v}[1]{\ensuremath{\mathbf{#1}}}

\newcommand{\mc}{\mathcal}

\newcommand{\be}{\begin{enumerate}}
\newcommand{\ee}{\end{enumerate}}

\newtheorem{lemma}{Lemma}

\renewcommand{\Re}{\mathbb{R}} 

\newcommand{\fhat}{\hat{f}} %

\newcommand{\fBefore}{f^{\textrm{before}}} %
\newcommand{\fAfter}{f^{\textrm{after}}} %

\title[Article Title]{Solution Polishing via Path Relinking for Continuous Black-Box Optimization}

\author*[1]{\fnm{Dimitri J.} \sur{Papageorgiou}}\email{dimitri.j.papageorgiou@exxonmobil.com}
\author[2]{\fnm{Jan} \sur{Kronqvist}}\email{jankr@kth.se}
\author[3]{\fnm{Asha} \sur{Ramanujam}}\email{aramanuj@purdue.edu}
\author[4]{\fnm{James} \sur{Kor}}\email{jkor@purdue.edu}
\author[1]{\fnm{Youngdae} \sur{Kim}}\email{youngdae.kim@exxonmobil.com}
\author[3]{\fnm{Can} \sur{Li}}\email{canli@purdue.edu}

\affil*[1]{\orgdiv{Energy Sciences}, \orgname{ExxonMobil Technology and Engineering Company}, \orgaddress{\street{1545 Route 22 East}, \city{Annandale}, \postcode{08801}, \state{NJ}, \country{USA}}}

\affil[2]{\orgdiv{Department of Mathematics}, \orgname{KTH Royal Institute of Technology}, \orgaddress{\street{Lindstedtsvagen 25}, \city{Stockholm}, \postcode{100 44}, \state{Stockholm}, \country{Sweden}}}

\affil[3]{\orgdiv{Charles D. Davidson School of Chemical Engineering}, \orgname{Purdue University}, \orgaddress{\street{480 W. Stadium Ave}, \city{West Lafayette}, \postcode{47907}, \state{IN}, \country{USA}}}

\affil[4]{\orgdiv{Department of Computer Science}, \orgname{Purdue University}, \orgaddress{\street{305 N University St}, \city{West Lafayette}, \postcode{47907}, \state{IN}, \country{USA}}}

\begin{document}

\abstract{
When faced with a limited budget of function evaluations, state-of-the-art black-box optimization (BBO) solvers struggle to obtain globally, or sometimes even locally, optimal solutions. In such cases, one may pursue solution polishing, i.e., a computational method to improve (or ``polish'') an incumbent solution, typically via some sort of evolutionary algorithm involving two or more solutions. 
While solution polishing in ``white-box'' optimization has existed for years, relatively little has been published regarding its application in costly-to-evaluate BBO.  To fill this void, we explore two novel methods for performing solution polishing along one-dimensional curves rather than along straight lines. We introduce a convex quadratic program that can generate promising curves through multiple elite solutions, i.e., via path relinking, or around a single elite solution.    
In comparing four solution polishing techniques for continuous BBO, we show that solution polishing along a curve is competitive with solution polishing using a state-of-the-art BBO solver.}


\maketitle

\section{Introduction} \label{sec:introduction}

This work focuses exclusively on computationally-expensive continuous black-box optimization (BBO) problems with variable bounds.  Mathematically, our interest lies in solution polishing applied to the problem
\begin{equation} \label{eq:generic_continuous_box_constrained_bbo_problem}
\min \Big\{ f(\v{x}) : \v{x} \in [\v{x}^L,\v{x}^U] \subset \Re^D \Big\}
\end{equation}
where $\v{x}^L$ and $\v{x}^U$ denote lower and upper bounds on the decision vector $\v{x} \in \Re^D$ and $f:\Re^D \mapsto \Re$ is a costly-to-evaluate black-box function, requiring hours to days of computing or large monetary sums to evaluate a single point $\v{x}$.  The term ``solution polishing'' refers to a computational method used to improve (or ``polish'') an incumbent solution, typically via some sort of restricted neighborhood search or evolutionary algorithm involving two or more solutions. 
Today, many commercial ``white-box'' mixed-integer linear and nonlinear optimization solvers (e.g., CPLEX, Gurobi, SCIP) include a solution polishing mechanism that a user can enable in hopes of finding a better solution \citep{rothberg2007evolutionary,gleixner2017scip}. In contrast, given the vast array of BBO solvers, solution polishing mechanisms for continuous BBO problems are not widely available/agreed upon.

Computationally-challenging black-box functions rear their ugly head in numerous scientific and engineering disciplines from molecular simulations to physical experiments requiring manual intervention. In many applications, the main limiting factor is the number (or ``budget'') of function evaluations that one can make. This is because time-consuming and/or monetarily-expensive simulations and experiments ultimately dictate the budget of function evaluations that a user is willing to allot.
Moreover, even with state-of-the-art methods and solvers, the function evaluation budget often hinders us from obtaining a globally, or even locally, optimal solution.

When restricted to a small number of function evaluations, the best solution obtained by a state-of-the-art BBO solver might not even be locally optimal. Thus, the idea of solution polishing becomes highly relevant and serves as the motivation for this work. With this backdrop, the goal of solution polishing can be stated as follows: \textit{Improve the quality of the current best-found solution by some additional calculations and a small number of additional function evaluations}.

In this study, we pursue this guiding question: Given a set of elite solutions, how effective are one-dimensional subspace searches ``involving'' (i.e., passing through) these elite solutions at improving/polishing the best incumbent solution?  We focus on two main settings in which we have (A) a set of elite solutions that we wish to polish; and (B) a single solution that we wish to polish.

The contributions of this paper are: 
\begin{enumerate}
\item We propose a convex quadratic optimization problem (QP) that allows us to explore a curve that passes through multiple elite solutions thereby extending the path relinking framework of ``straight linking'' to a more elaborate search along a curve.
\item Using the same QP, we show how to construct a ``propeller'' shape around a single elite solution to perform a thorough search along a curve passing through this elite solution.  
\item Computational experiments suggest that our proposed ``curve searches'' are competitive with using state-of-the-art DFO solvers for solution polishing.
\end{enumerate}

The remainder of this paper is organized as follows.
Section~\ref{sec:lit_review} briefly reviews the literature on solution polishing and path relinking. Furthermore, Section~\ref{sec:lit_review} describes an overview of the line search methodology that our proposed methods build upon.
In Section~\ref{sec:methods}, we introduce our convex QP for generating curves through elite solutions.
Section~\ref{sec:numerical_results} presents numerical results on a set of popular benchmark functions in 2, 4, 8, and 16 dimensions. 
We offer conclusions and future research directions in Section~\ref{sec:conclusions}.

\section{Literature review} \label{sec:lit_review}

Since the surveys by \cite{conn2009introduction}, \cite{larson2019derivative}, \cite{rios2013derivative}, and the references therein offer an extensive BBO introduction, we focus on discussing path relinking and efficient line search methods.

\subsection{Path relinking}

Path relinking (PR) is most often described as an intensification strategy to explore trajectories connecting elite solutions obtained by heuristic methods \citep{glover1998template}.  The term ``strategy'' refers to a high-level template or algorithmic framework as many implementation decisions are left to the user. \citet{glover1998template} initially introduced PR in tandem with scatter search, a population-based metaheuristic that has been successfully applied to challenging combinatorial optimization problems.  \citet{yin2010cyber} enumerate the similarities and differences between the PR approaches used in particle swarm optimization and scatter search.

While \citet{glover2000fundamentals} provide a detailed account of the evolutionary philosophy behind path relinking, we summarize it here. See also the survey by \cite{laguna202320}.
In contrast to traditional evolutionary algorithms that directly produce a child (i.e., a new solution) by combining two or more parents (i.e., original solutions) from a population, PR constructs paths between and beyond the selected solutions. One generates paths by attempting to introduce attributes from one or more elite guiding solutions into the current solution.  The hope is that by assimilating characteristics of elite solutions, an improving solution can be found. 

Figure~\ref{fig:Orthogonal_vs_straight_path_relinking} illustrates the difference between orthogonal and straight PR.  As shown in the left of the figure, ``orthogonal linking,'' which is most often used in combinatorial/discrete optimization settings, constructs trajectories from the initiating solution to the guiding solution in which ``attributes'' of the latter are gradually introduced into the former. While an attribute can be defined generally, in practice they are most often dimensions or coordinates of a solution.  A prototypical orthogonal linking algorithm goes as follows. First, one identifies the coordinates in which the two solutions differ. Second, one selects one of these coordinates at random and replaces the initial solution's coordinate value with that of the guiding solution. This process continues until the guiding solution is reached. Orthogonal linking generates trajectories along the edges of the hyper-rectangle defined by the two points (see Figure~\ref{fig:Orthogonal_vs_straight_path_relinking}). \citet{glover2000fundamentals} also permit exploring solutions beyond the two original solutions.

In contrast, ``straight linking'' constructs trajectories using linear combinations of initiating and guiding solutions. In its simplest form involving a single initiating solution and a single guiding solution, straight linking generates a sequence of solutions on, and potentially beyond, the line segment defined by the initiating and guiding solutions. See the right side of Figure~\ref{fig:Orthogonal_vs_straight_path_relinking}. 
\citet{duarte2011path} describe a more general straight linking strategy involving multiple guiding solutions.  For their computational experiments, they use two guiding solutions. 

\begin{figure}[h!] 
\centering
\includegraphics[width=9cm]{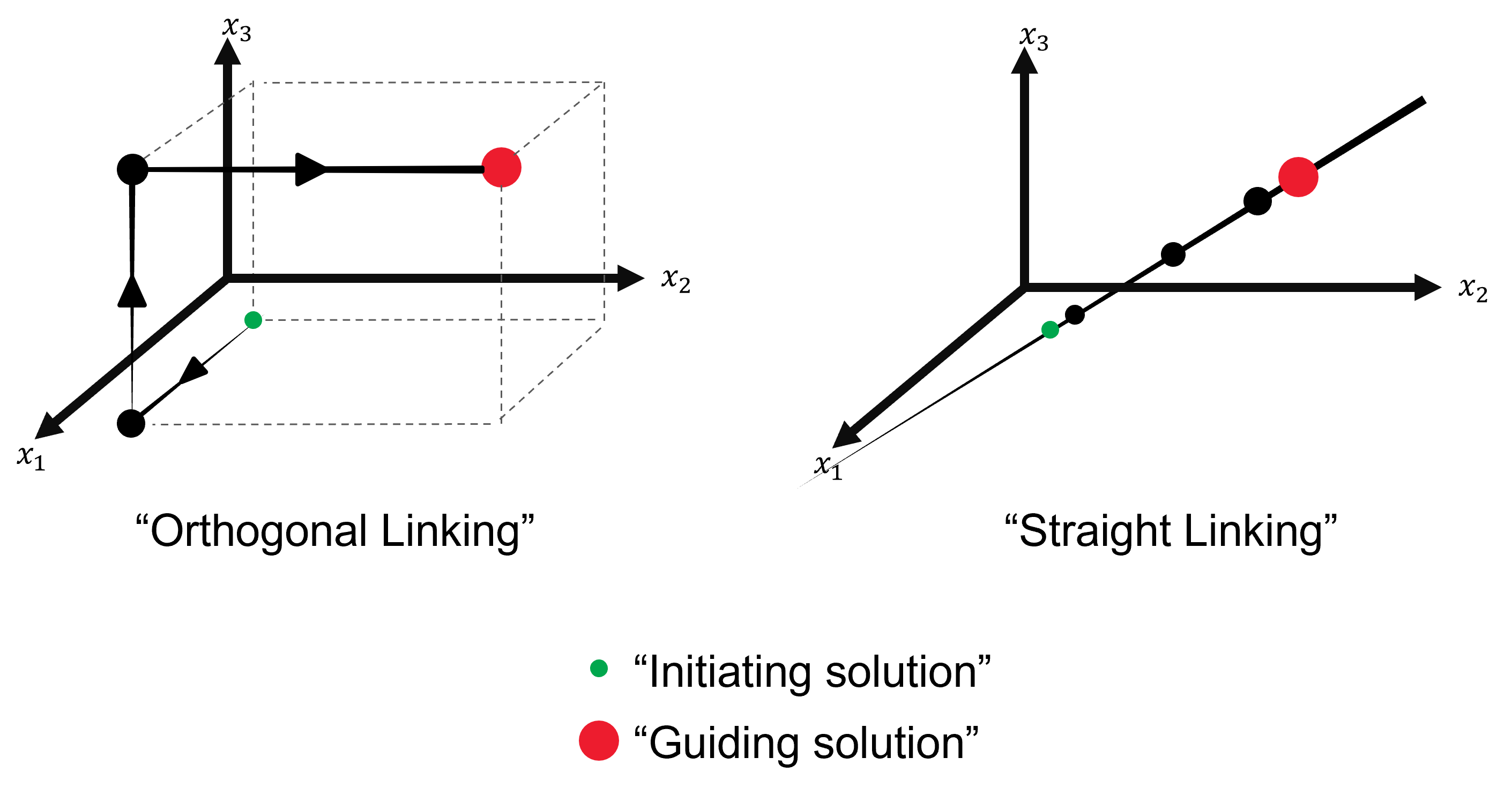}
\caption{Orthogonal vs. straight path relinking.}
\label{fig:Orthogonal_vs_straight_path_relinking}
\end{figure}

\subsection{Line Search}

There are a number of algorithms available to conduct a global line search. \citet{papageorgiou2024linewalker} recently showed that two of the most competitive methods on nonconvex functions are LineWalker and Bayesian optimization \citep{brochu2010tutorial,shahriari2015taking}. In this work, we take advantage of LineWalker, which we describe below.

Linewalker is a simple, but effective sampling method for optimizing and learning a discrete approximation of a deterministic, mostly smooth, multi-dimensional function along a one-dimensional line segment of interest. 
The method does not rely on derivative information.
The method constructs a smooth surrogate on a set of equally-spaced grid points by evaluating the true function at a sparse set of judiciously chosen grid points.  At each iteration, the surrogate's non-tabu local minima and maxima are identified as candidates for sampling. Tabu search constructs are also used to promote diversification. If no non-tabu extrema are identified, a simple exploration step is taken by sampling the midpoint of the largest unexplored interval. The algorithm continues until a user-defined function evaluation limit is reached. Numerous examples are shown to illustrate the algorithm's efficacy and superiority relative to state-of-the-art methods, including Bayesian optimization and NOMAD, on primarily nonconvex test functions. The numerical experiments clearly show that LineWalker is able to obtain good solutions, and accurate surrogates, using relatively few (typically only 20 -- 30) function evaluations.

Our goal is to take advantage of, and build upon, LineWalker's ability to efficiently find global minima in one dimension. Unlike traditional gradient/Hessian-based methods, which seek to find the nearest local optima and then stop, this method seeks to approximate the function along the entirety of one-dimensional subspaces of the $D$-dimensional search space using a small number of function evaluations. Since local information for a truly nonconvex function tells you nothing about the function's behavior far from the current point, the hope is that LineWalker, which searches beyond the nearest local optima, will efficiently uncover more information about the function than a traditional method.

\section{Methods} \label{sec:methods}
Here we present the different algorithmic approaches for solution polishing by optimizing over one-dimensional subspaces. Given a set of elite solutions, the simplest approach to utilize a line search is by considering the line segments obtained by connecting the different elite solutions by a straight line, i.e., by ``Straight Linking''. Such a straight linking approach is included in the numerical results and described in more detail in Section~\ref{sec:numerical_experiments}. However, our goal is to use a more elaborate linking approach that allows us to more efficiently utilize all function evaluations and focus the search to more promising neighborhoods. 

A revelation in this work was the realization that the one-dimensional search does not need to be limited to a straight line, but can in fact be any smooth curve of finite length. Considering a curve, instead of separate line segments, enables more efficient utilizing of function evaluations by learning a single one-dimensional surrogate function instead of several independent surrogate functions. For example, if we want to learn a function approximation of a smooth function along a curve, it is favorable to directly learn the approximation along the entire curve compared to dividing the curve into separate parts and learning these independently. The rationale behind this is straightforward: as we assume the function is smooth, it is likely that the function will follow a similar trajectory when moving from one part of the curve to another. As later described, it also allows us to reuse some function evaluations. 

Given a set of elite solutions, we want to construct a smooth curve that somehow links the elite solutions and that possesses promising characteristics which allow us to find an improved solution. A smooth curve is preferred to not introduce any additional non-smoothness into the function that we are trying to learn and optimize. This property is formally stated in the following lemma.

\begin{lemma}
Given a $D$-dimensional function $f$ that is smooth over $\mathcal{F}\subset\Re^D$, the function is also smooth along any smooth curve that lies within $\mathcal{F}$.
\end{lemma}
\begin{proof}
The proof follows trivially from the fact that a composition of two smooth functions is also smooth.
\end{proof}
Next, we present a simple approach for constructing a smooth curve that links a set of elite solutions.

\subsection{Optimizing curves}
The task of finding a smooth curve that goes through a given set of points $\{\v{p}_k\}_{k=0}^K$ in $\Re^D$ can be posed as an optimal control problem. For a more intuitive presentation, we envision a steel ball that moves frictionlessly without any external forces except a control force that we can apply to accelerate the ball in any desired direction. A smooth curve that goes through the points $\{\v{p}_k\}_{k=0}^N$ can then be obtained by minimizing the total force applied to the steel ball with the constraints that the ball must be at the points $p_k$ at the points in time $T_k$. Due to the momentum of the ball, a solution to the control problem is a smooth curve. Since we are minimizing the total acceleration, one could even claim that it is an optimally smooth curve. For the LineWalker algorithm, we only need a set of discrete samples along the curve, and we do not need the entire continuous curve. Therefore, we consider a discretization of the control problem that can be written as

\begin{equation}
\begin{aligned}
\label{eq:QP}
&\min &&\sum_{t=0}^T 	\lVert \v{a}_t \rVert^2_2   \\ 
&\text{s.t.} &&\v{x}_t - \v{x}_{t-1} = \v{v}_t, &&& t =1, \dots, T\\
& &&\v{v}_t - \v{v}_{t-1} = \v{a}_t, &&& t =1, \dots, T\\
& && \v{x}_{T_k} = \v{p}_k, &&& k = 1, \dots, N\\
& && \v{v}_0 = \v{0},\\
& && \v{x}^L \leq \v{x}_t \leq \v{x}^U, &&& t =0, \dots, T\\
& &&\v{x}_t \in\mathbb{R}^D, \v{v}_t \in\mathbb{R}^D, \v{a}_t \in\mathbb{R}^D,  &&& t =0, \dots, T
\end{aligned}
\end{equation}
where $\v{x}_t$ is the position coordinates of the ball at time $t$, $\v{v}_t$ the velocity, and $\v{a}_t$ the acceleration. ${T_k}$ is the time step at which the ball has to be at the given point $\v{p}_k$. Furthermore, $\v{x}^L$ and $\v{x}^U$ represent box constraints on the search domain as in \eqref{eq:generic_continuous_box_constrained_bbo_problem}. Note that~\eqref{eq:QP} is based on a simple forward difference to discretize the underlying control problem. For simplicity, we assume the time points $T_1,\dots,T_N$ are equally distributed. By adjusting the number of time steps $T$, we can obtain a grid of desired precision for the LineWalker algorithm. Linear and convex constraints in the black box problem can easily be integrated into \eqref{eq:QP}, forcing the entire curve to satisfy the constraints. Our framework can, therefore, easily deal with linear and convex constraints, although we only consider box constraints in this paper.

Problem~\eqref{eq:QP} is a convex quadratic programming (QP) problem that can be solved in polynomial time \cite{nesterov1994interior} and there are several efficient off-the-shelf software for solving such. We simply used the Matlab solver \texttt{quadprog} to solve our instances, and we were easily able to obtain 16D curves with 5000 points. In fact, solving the QP problems was not a limiting factor in the numerical study. From experiments, we found it favorable to include a small penalty on the curves total length, and we ended up using 
\begin{equation}
    \sum_{t=0}^T 	\lVert \v{a}_t \rVert^2_2  +0.001\sum_{t=0}^{T-1}  \lVert \v{x}_{t+1} -\v{x}_{t} \rVert^2_2,  
\end{equation}
as the objective function in \eqref{eq:QP}. The penalty on the curve's length prevents the curve from taking elongated turns away from the elite solutions.

\textbf{Example 1}: As an illustrative example, consider the task of finding a smooth two-dimensional curve that follows the path from (0,0) to (3,1), back to (0,0), then to (3,3), and returns to the starting point at (0,0). The curves obtained by solving problem \eqref{eq:QP} using 80 and 300 points are shown in Figure~\ref{fig:example_curves1}. The curves in Figure~\ref{fig:example_curves1} also show a welcomed side effect of generating discrete samples on a curve by solving \eqref{eq:QP}. Due to the  ``momentum of the ball'', we end up with more samples and a finer grid close to the given points, i.e., the elite solutions. As these are points of high interest, it is desirable to have a finer discretization close to these. 

\begin{figure} [h!]
\begin{subfigure}[b]{0.40\textwidth}
\includegraphics[width=\textwidth]{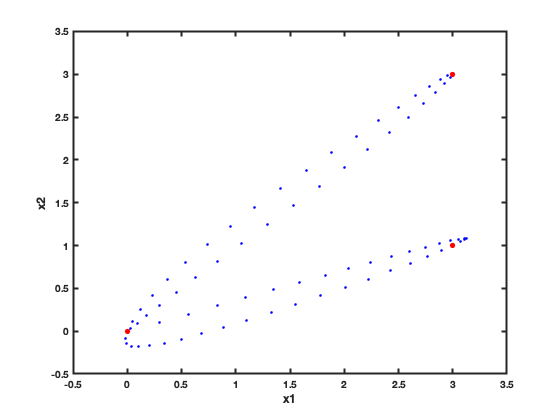}
\caption{80 point curve}
\label{fig:80_point_curve}
\end{subfigure}
\begin{subfigure}[b]{0.40\textwidth}
\includegraphics[width=\textwidth]{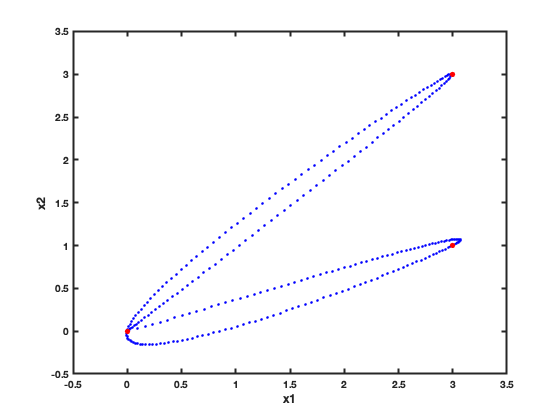}
\caption{300 point curve}
\label{fig:fig:80_point_curve}
\end{subfigure}
\caption{Illustration of the discrete points on smooth curves obtained by solving \eqref{eq:QP} using a different number of time steps. Observe that the coarse grid is only used for illustration purposes; for solution polishing, we use a significantly finer grid for the curve.}
\label{fig:example_curves1}
\end{figure}

By using the discrete points along the curve, obtained by solving \eqref{eq:QP}, we can apply the LineWalker algorithm to optimize along the curve. Next we propose two techniques for constructing the path, or sequence, through the elite solutions.  

\subsection{Curve search algorithm}
Given a set of elite solutions, we can construct different curves by solving \eqref{eq:QP} and changing the order in which the curve passes through the points. Following the path relinking philosophy, we want the curve to consider different ways of combining the points and simultaneously avoid creating overly complex curves that render minimization difficult. Below, we propose two approaches for constructing the curve.

Given a set of elite solutions, we first sort the solutions in increasing order according to their function value such that the most promising is first followed by the second most promising. If two solutions are equally good, we can arbitrarily pick one of them first. Starting from the most promising solution, we generate a curve that goes out to the second most promising point and back to the most promising point. The curve follows this pattern and visits all the elite solutions by going back and forth to the most promising solution.  We refer to the strategy as the ``\textbf{Multipoint curve}''. Our main motivations behind this approach are: I) the curve explores several directions from the most promising solution, and II) we are allowed to reuse some function evaluations as the most promising point appears at several discrete points on the curve. Figure~\ref{fig:example_curves} depicts an example of a four-dimensional multipoint curve projected down to three dimensions for illustrative purposes. The multipoint curve in Figure~\ref{fig:boha4D_multipoint_D234} passes through five elite solutions.  The procedure for generating the multipoint curve is presented as pseudocode in Algorithm~\ref{algo:generate_mp_curve}.

\begin{algorithm} 
\caption{\texttt{generateMultipointCurve($\mc{S}, N_{\text{between}}$)}: Generate a set $\mc{X}_{\text{grid}}$ of points on a smooth curve that goes though the elite solutions $\mc{S}$ with $N_{\text{between}}$ grid points in between the elite solutions.}
\label{algo:generate_mp_curve}
\begin{algorithmic}[1]
\State $\mc{S} \gets \texttt{Sort}(\mc{S}) $\Comment{Sort elite solutions in increasing order as per obj. function value}
\State $i = 1 $\Comment{Counter for points on curve}
\State $ \v{p}_0 = \mc{S}(1)$, $T_0 = 0$  \Comment{Start curve at most promising elite solution}
         \For{$k  = 1$ to $|\mc{S}| $}                    
        \State $\v{p}_i = \mc{S}(k)$, $T_i= T_{i-1} + N_{\text{between}}$  \Comment{Next elite solution to pass curve through}
         \State  $i = i + 1$
         \State $\v{p}_i = \mc{S}(0)$, $T_i= T_{i-1} + N_{\text{between}}$  \Comment{Pass curve back through the best point}
          \State $i = i + 1$
    \EndFor
   \State $\mc{X}_{\text{grid}} \gets$ minimizer $\{\v{x}_t^*\}_{t=0}^T$ of \eqref{eq:QP} using points $\v{p}_i$, time steps $T_i$, $T~=~2|\mc{S}|N_{\text{between}}$
\State \textbf{return} $\mc{X}_{\text{grid}}$ \Comment{Return the set of curve points}
\end{algorithmic}
\end{algorithm}

In the second approach, we base the curve on the most promising of the elite solutions. This curve is constructed by going back and forth from the most promising solution through all points that are obtained by taking a unit step in each direction from the most promising solution. We take a unit step both in a positive and negative direction in each dimension. If the unit step in some direction does not satisfy the box constraints we simply reduce the step length to stay within the box. We refer to this strategy as the ``\textbf{Propeller curve}'', due to its resemblance of a propeller in the two dimensional case. A three dimensional illustration of a propeller curve is also shown in Figure~\ref{fig:example_curves}. This approach creates a curve that explores several directions from the best elite solution, and focuses the search more towards the most promising point compared to the multipoint curve. In this approach we are also able to exploit that the most promising elite solution appears at multiple points along the curve and can all be obtained by a single function evaluation. By the propeller curve we are also able to consider the situation where we only have a single known solution that we wish to improve by polishing. The procedure for generating the propeller curve is presented as pseudocode in Algorithm~\ref{algo:generate_p_curve}.

\begin{algorithm} 
\caption{\texttt{generatePropellerCurve($\v{x}_{\text{elite}}, N_{\text{between}}$)}: Generate a set $\mc{X}_{\text{grid}}$ of points on a smooth curve that goes back and forth through an elite solution $\v{x}_{\text{elite}} \in \Re^D$ with $2N_{\text{between}}$ grid points in between the elite solution.}
\label{algo:generate_p_curve}
\begin{algorithmic}[1]
\State $i = 1 $\Comment{Counter for points on curve}
\State $ p_0 = \v{x}_{\text{elite}}$, $T_0 = 0$  \Comment{Start curve on the elite solution}
         \For{$k  = 1$ to $D$}                    
        \State $ p_i =  p_0 +\v{e}_k$ \Comment{Positive unit step. $\v{e}_k$ is a zero vector with 1 at position $k$}
        \State  $i = i + 1$
         \State $ p_i = p_0, T_i= T_{i-1} + N_{\text{between}}$  \Comment{Pass curve back through the best point}
                \State  $i = i + 1$
         \State $ p_i =  p_0 -\v{e}_k$ \Comment{Negative unit step}
        \State  $i = i + 1$
         \State $ p_i = p_0, T_i= T_{i-1} + N_{\text{between}}$  \Comment{Pass curve back through the best point}
          \State  $i = i + 1$
    \EndFor
   \State $\mc{X}_{\text{grid}} \gets$ minimizer $\{\v{x}_t^*\}_{t=0}^T$ of problem \eqref{eq:QP} with points $p_i$, time steps $T_i$, and $T~= ~4DN_{\text{between}}$.
\State \textbf{return} $\mc{X}_{\text{grid}}$ \Comment{Return the set of curve points}
\end{algorithmic}
\end{algorithm}

Once we have obtained the curve, we simply apply the LineWalker algorithm to minimize the black-box objective along the curve. Minimizing the objective along the curve is obviously a restrictive search, but keep in mind the goal is to improve upon some elite solutions using only a small number of function evaluations. There are numerous way of constructing curves to optimize, and we mainly motivate our proposed multipoint and propeller curves by I) an exploration of elite solutions in a path relinking fashion, II) a local search around the incumbent solution, III) a smooth curve to not introduce non-smoothness into the 1-dimensional search, and IV) an efficient utilization of function evaluations by ``reusing'' some points. In Section~\ref{sec:numerical_experiments} we show empirical evidence to support the claim that elite solutions can efficiently be improved by optimizing along these curves. 

For clarity we also present a pseudocode for the solution by optimizing along the curves in Algorithm~\ref{algo:curve_polish}. In the algorithm we allow LineWalker to use $N_{\text{given}}$ additional function evaluations, as we start the algorithm by  $N_{\text{given}}$ known points, i.e., the elite solutions along the curve.

\begin{algorithm} 
\caption{\texttt{CurvePolisher($\mc{S}, \text{fval}_{\text{max}}, \text{strategy}, N_{\text{between}}$)}: Perform solution polishing on a set $\mc{S}$ of elite solution(s) using LineWalker to conduct a 1D curve search allowing at most $\texttt{fval}_{\text{max}}$ function evaluations. The parameters \texttt{strategy} and $N_{\text{between}}$ govern the curve and the number of grid points between elite solutions.}
\label{algo:curve_polish}
\begin{algorithmic}[1]
\If{strategy = multipoint} 
    \State $\mc{X}_{\text{grid}} \gets$ \texttt{generateMultipointCurve($\mc{S}, N_{\text{between}}$)}:
    \State $N_{\text{given}} = 2|\mc{S}| - 1$ \Comment{Number of known points along the curve}
    \ElsIf{strategy = propeller}
      \State $\v{x}_{\text{elite}} \gets \arg\min_{\v{x} \in \mc{S}} f(\v{x})$  \Comment{Choose elite solution with best objective value}
      \State  $\mc{X}_{\text{grid}} \gets$ \texttt{generatePropellerCurve($\v{x}_{\text{elite}}, N_{\text{between}}$)}
       \State $N_{\text{given}} = 3+2(dim(\v{x}_{\text{elite}})-1)$ \Comment{Depends on dimensionality of $\v{x}_{\text{elite}}$ }
\EndIf 
\State $ \v{x}^* \gets$   \texttt{LineWalker}($\mc{X}_{\text{grid}}, \text{fval}_{\text{max}}+|\mc{S}|$)\Comment{Optimize along the curve with function evaluation budget $\text{fval}_{\text{max}}+N_{\text{given}} $}
 \State \textbf{return} $\v{x}^*$ \Comment{Return polished solution}
\end{algorithmic}
\end{algorithm}

\textbf{Example 2}: To illustrate how Algorithms~\ref{algo:generate_mp_curve}--\ref{algo:curve_polish} work in practice, we consider the solution polishing task for the instance 4D Boha. More details on this optimization instance is given in Section~\ref{sec:numerical_experiments}.  
Figure~\ref{fig:example_curves} depicts examples of the multipoint and propeller curves.  Algorithm~\ref{algo:generate_mp_curve} generates the multipoint curve in Figure~\ref{fig:boha4D_multipoint_D234} passing through five elite solutions, while introducing 400 grid indices between each solution. Algorithm~\ref{algo:generate_p_curve} generates the propeller-shaped curve in Figure~\ref{fig:boha4D_propeller_D234} around the best (i.e., the one with the lowest objective function value) of the five elite solutions used in Figure~\ref{fig:boha4D_multipoint_D234}. It introduces 200 grid indices between each point. Both methods generate a grid with 3201 indices.

\begin{figure} [h!]
\begin{subfigure}[b]{0.49\textwidth}
\includegraphics[width=\textwidth]{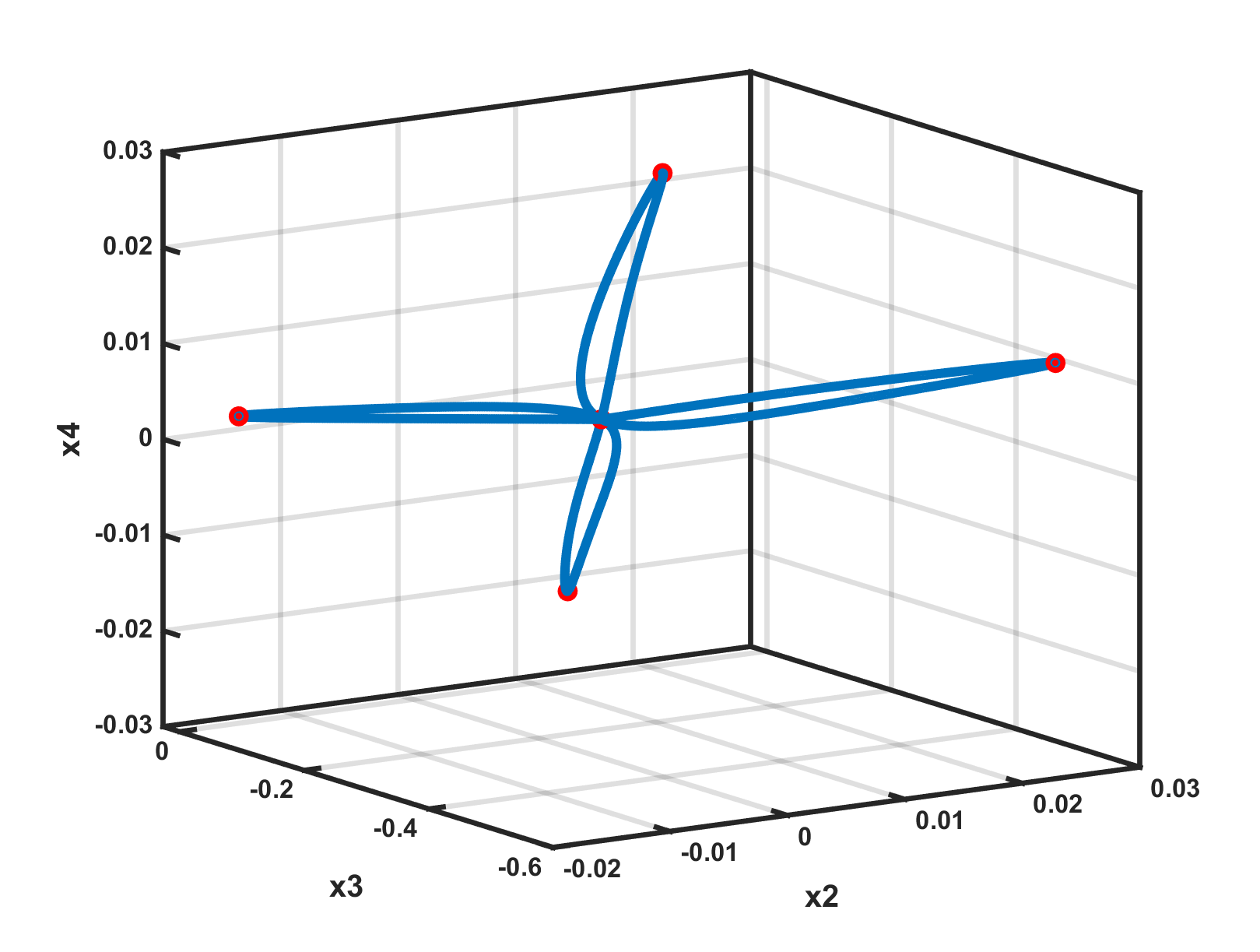}
\caption{Multipoint curve}
\label{fig:boha4D_multipoint_D234}
\end{subfigure}
\begin{subfigure}[b]{0.49\textwidth}
\includegraphics[width=\textwidth]{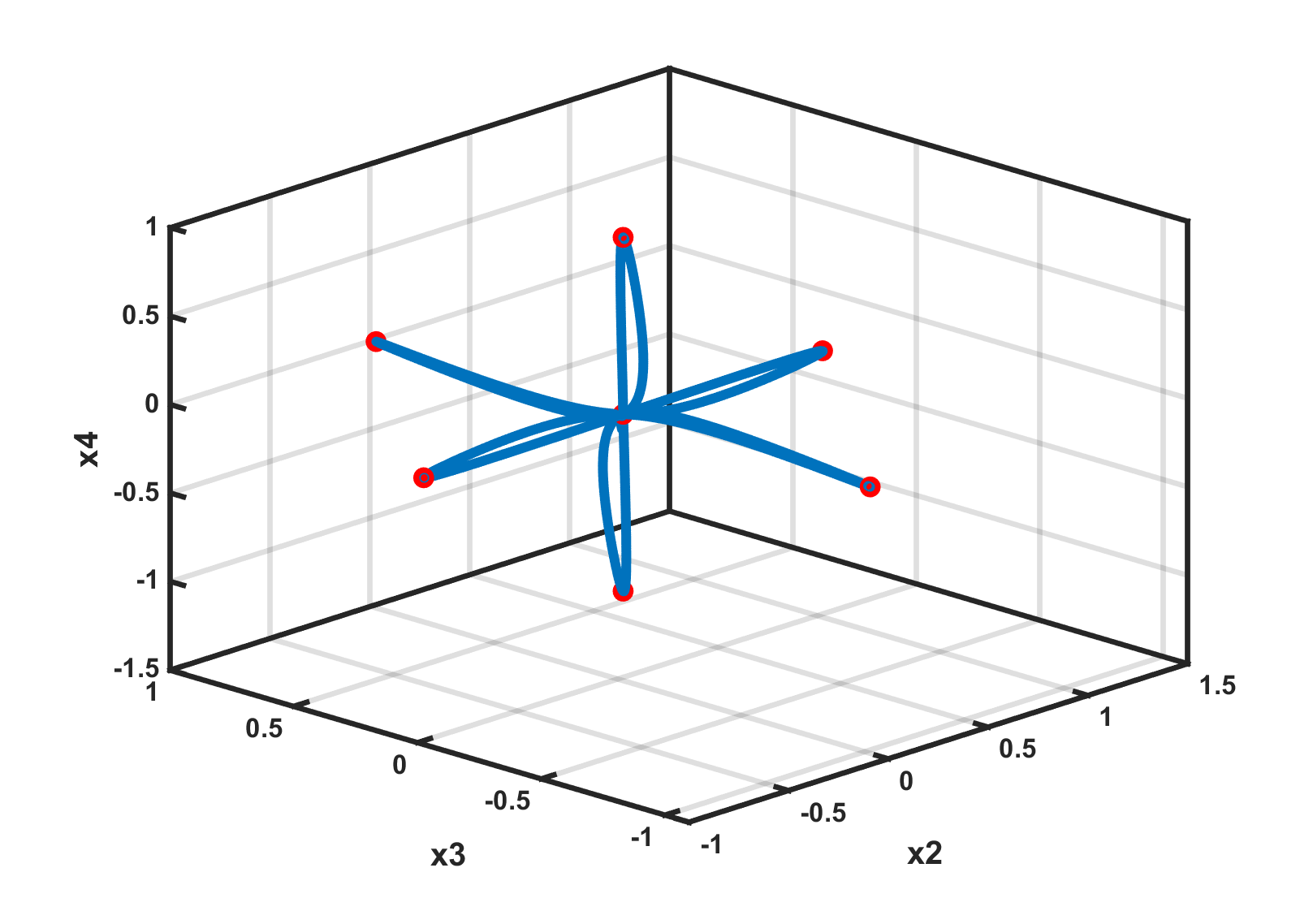}
\caption{Propeller curve}
\label{fig:boha4D_propeller_D234}
\end{subfigure}
\caption{Examples of curves generated for the 4D Boha instance.  Only dimensions 2, 3, and 4 are shown.}
\label{fig:example_curves}
\end{figure}

Figure~\ref{fig:linewalker_curves} shows LineWalker's ability to find an improving solution on the generated curves shown in Figure~\ref{fig:example_curves}. After creating a discrete grid on indices, LineWalker judiciously samples grid indices, which correspond to solutions, to simultaneously find a global optimum and construct a surrogate. In Figure~\ref{fig:linewalker_curves}, a grid of 3201 grid indices is used and rougly 200 indices are actually sampled.  Since optimality trumps surrogate quality when performing solution polishing, we could actually use far fewer than 200 samples and still find an improving solution. In Figure~\ref{fig:linewalker_curves}, LineWalker's ``Current Fit'' is nearly identical to the ``Ground Truth,'' which would not be known in most applications. The terms ``Initial Sample,'' ``Last Sample,'' and ``Previous Sample'' refer to the initial samples passed to LineWalker, the last grid index evaluated by LineWalker, and all previous samples where LineWalker made a function evaluation.    

Figure~\ref{fig:boha4D_multipoint_lw120_with_inset} depicts function evaluations and the surrogate generated by LineWalker while exploring the ``multipoint'' curve shown in Figure~\ref{fig:boha4D_multipoint_D234}.  LineWalker's surrogate is nearly identical to the ground truth, which of course would not be available to the user in practice.  The black dots correspond to the initial samples (the initial elite solutions) in the sequence.
The inset figure clearly shows that an improving solution was found at grid index 619 with an objective function of $\num{8.61e-5}$, more than two orders of magnitude lower than the best known objective function value of $\num{1.93e-2}$. 

\begin{figure} [h!]
\begin{subfigure}[b]{0.53\textwidth}
\includegraphics[width=\textwidth]{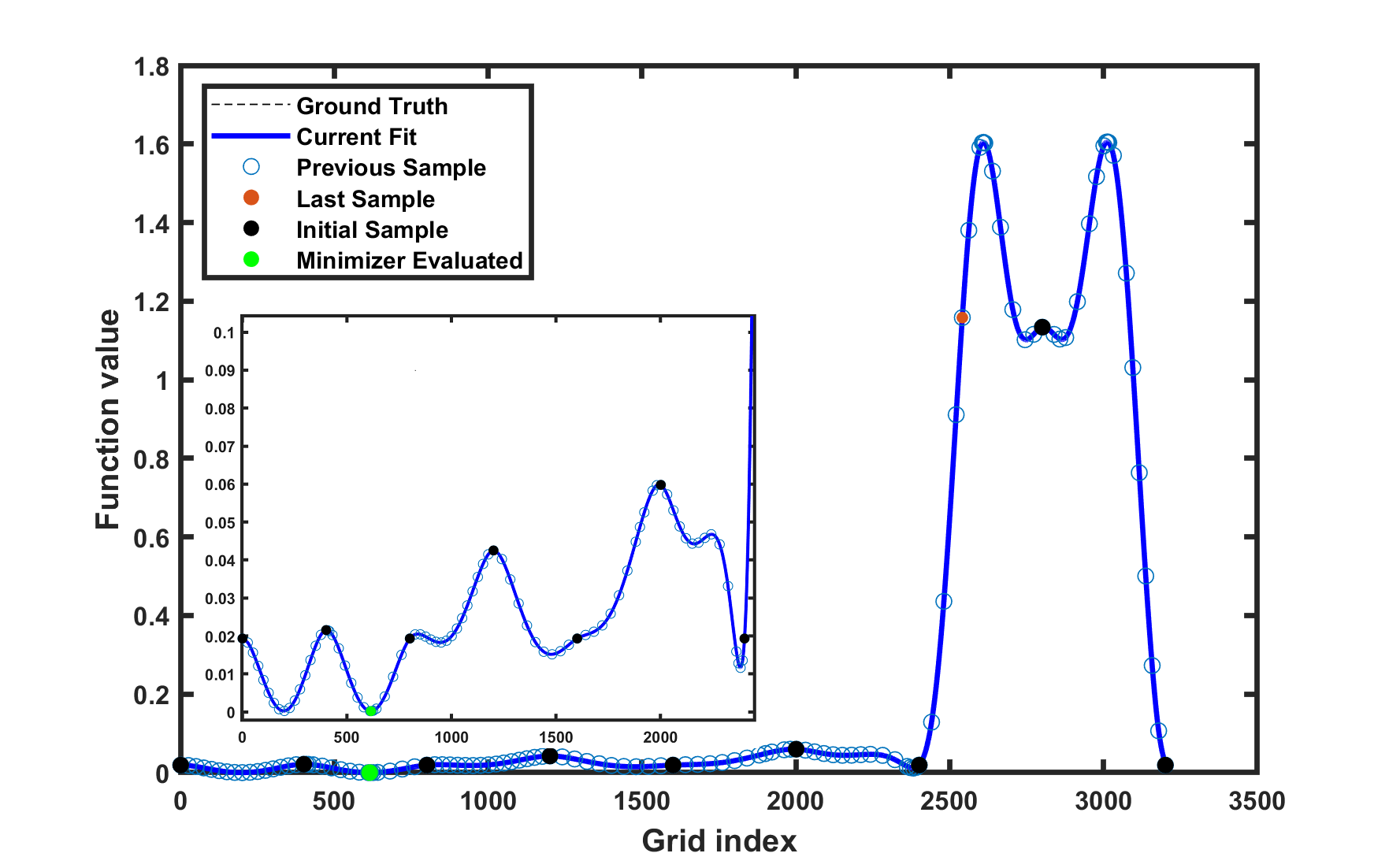}
\caption{Multipoint curve}
\label{fig:boha4D_multipoint_lw120_with_inset}
\end{subfigure}
\begin{subfigure}[b]{0.45\textwidth}
\includegraphics[width=\textwidth]{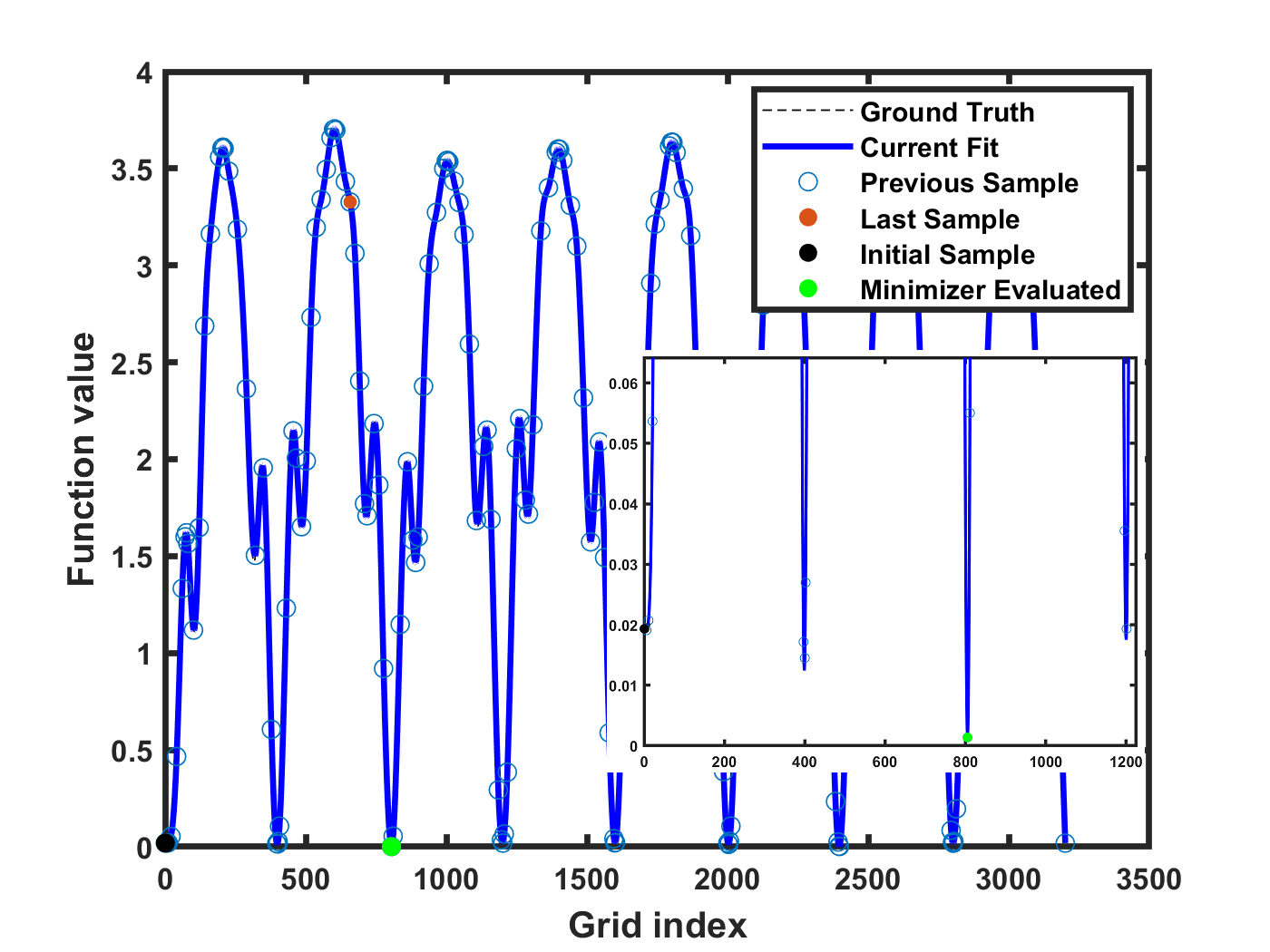}
\caption{Propeller curve}
\label{fig:boha4D_propeller_lw200_with_inset}
\end{subfigure}
\caption{LineWalker's line search performance on the 4D Boha instances in Figure~\ref{fig:example_curves}. The insets reveal that multiple improving solutions were found for both methods.}
\label{fig:linewalker_curves}
\end{figure}

Similarly, Figure~\ref{fig:boha4D_propeller_lw200_with_inset} depicts function evaluations and the surrogate generated by LineWalker while exploring the ``propeller'' curve shown in Figure~\ref{fig:boha4D_propeller_D234}. The inset zooms in on the interval where LineWalker evaluated multiple improving solutions. Specifically, the best known objective function value of $\num{1.93e-2}$ was improved an order of magnitude to $\num{1.14e-3}$. 
As a final important takeaway, the insets in both figures reveal that LineWalker is capable of finding more than one improving solution.

\section{Numerical Experiments} \label{sec:numerical_experiments}

\subsection{Experimental set up} \label{sec:numerical_setup}

\subsubsection{Test suite of functions}

We consider 19 standard test functions (available at \url{https://www.sfu.ca/~ssurjano/}, \url{https://github.com/luclaurent/optiGTest/tree/master/unConstrained}, and \cite[Figure 7]{duarte2011path}) to demonstrate our solution polishing methods. We use the terms ``test function'' and ``instance'' interchangeably throughout.
Table~\ref{table:test_fnc_table} lists the test functions along with their bounds and minimum objective function values.  While Figure~\ref{fig:test_fnc_gallery} plots these test functions in 2 dimensions, we perform solution polishing on each test function in 2, 4, 8, and 16 dimensions. 

\begin{table}[]
\caption{Benchmark test functions. The lower and upper bounds refer to the variable bounds $\v{x}^L$ and $\v{x}^U$, respectively. ``Global Minimum Objval'' refers to the known global minimum objective function value, some of which depend on the dimension $D$.}
\label{table:test_fnc_table}
\begin{tabular}{lrrc}
\toprule
\textbf{Function Name}    & \textbf{Lower Bound} & \textbf{Upper Bound} & \textbf{Global Minimum Objval} \\
\midrule
\href{http://www.sfu.ca/~ssurjano/ackley.html}{ackley}                    & -32.768     & 32.768      & 0              \\
\href{https://github.com/luclaurent/optiGTest/blob/master/unConstrained/funCosineMixture.m
}{cosineMixture}             & -1          & 1           & $-0.06301D$       \\
\href{https://github.com/luclaurent/optiGTest/blob/master/unConstrained/funDeflectedCorrugatedSpring.m}{deflectedCorrugatedSpring} & 0           & 10          & -1             \\
\href{https://www.sfu.ca/~ssurjano/dixonpr.html}{DixonPrice}                & -10         & 10          & 0              \\
\href{https://github.com/luclaurent/optiGTest/blob/master/unConstrained/funGiunta.m}{giunta}                    & -1          & 1           & $-0.26776D+0.6$       \\
\href{https://www.sfu.ca/~ssurjano/griewank.html}{griewank}                  & -600        & 600         & 0              \\
\href{http://www.sfu.ca/~ssurjano/levy.html}{levy}                      & -10         & 10          & 0              \\
\href{http://www.sfu.ca/~ssurjano/michal.html}{michal}                    & 0           & 3.14159     & See Table~\ref{tab:michal_global_minimum_objective_function_values}            \\
\href{https://github.com/luclaurent/optiGTest/blob/master/unConstrained/funPinter.m}{pinter}                    & -10         & 10          & 0              \\
\href{https://www.sfu.ca/~ssurjano/powell.html}{powell}                    & -4          & 5           & 0       \\      
\href{http://www.sfu.ca/~ssurjano/rastr.html}{rastrigin}                 & -5.12       & 5.12        & 0              \\
\href{https://www.sfu.ca/~ssurjano/rosen.html}{rosenbrock}                & -5          & 10          & 0              \\
\href{http://www.sfu.ca/~ssurjano/schwef.html}{schwefel}                    & -500        & 500         & 0              \\
\href{https://link.springer.com/article/10.1007/s00500-010-0650-7/figures/7}{boha}                      & -100        & 100         & 0              \\
\href{https://link.springer.com/article/10.1007/s00500-010-0650-7/figures/7}{shiftedSchaffer}           & -100        & 100         & 0              \\
\href{https://www.sfu.ca/~ssurjano/spheref.html}{spheref}                   & -5.12       & 5.12        & 0              \\
\href{http://www.sfu.ca/~ssurjano/stybtang.html}{stybtang}                  & -5          & 5           & $-39.1662D$        \\
\href{https://github.com/luclaurent/optiGTest/blob/master/unConstrained/funTrigonometric2.m}{trig2}                     & -500        & 500         & 1              \\
\href{http://www.sfu.ca/~ssurjano/zakharov.html}{zakharov}                  & -5          & 10          & 0              \\
\bottomrule
\end{tabular}
\end{table}

\begin{figure} [h!]
\begin{subfigure}[b]{0.24\textwidth}
\includegraphics[width=\textwidth]{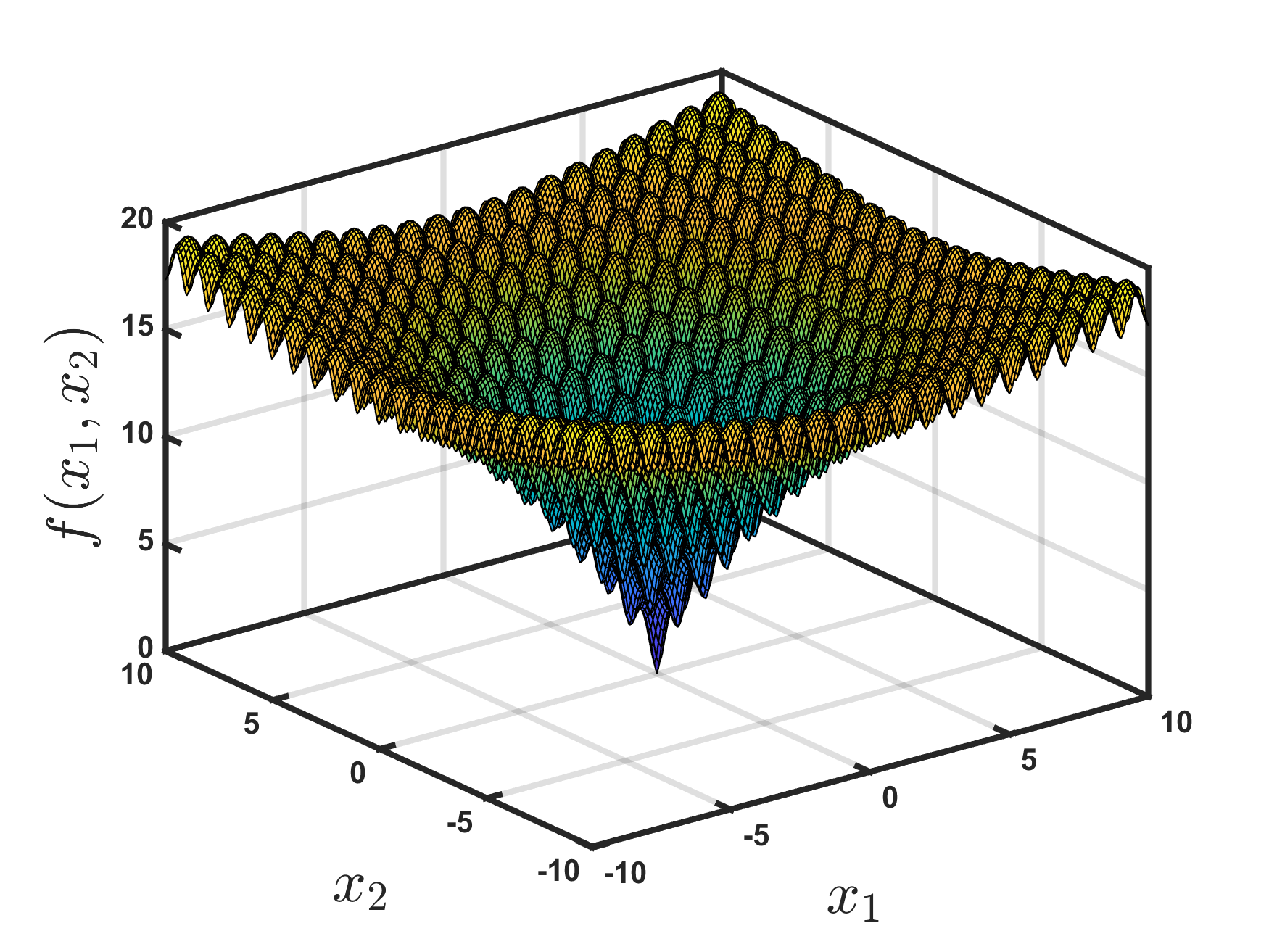}
\caption{\href{http://www.sfu.ca/~ssurjano/ackley.html}{ackley}}
\end{subfigure}
\begin{subfigure}[b]{0.24\textwidth}
\includegraphics[width=\textwidth]{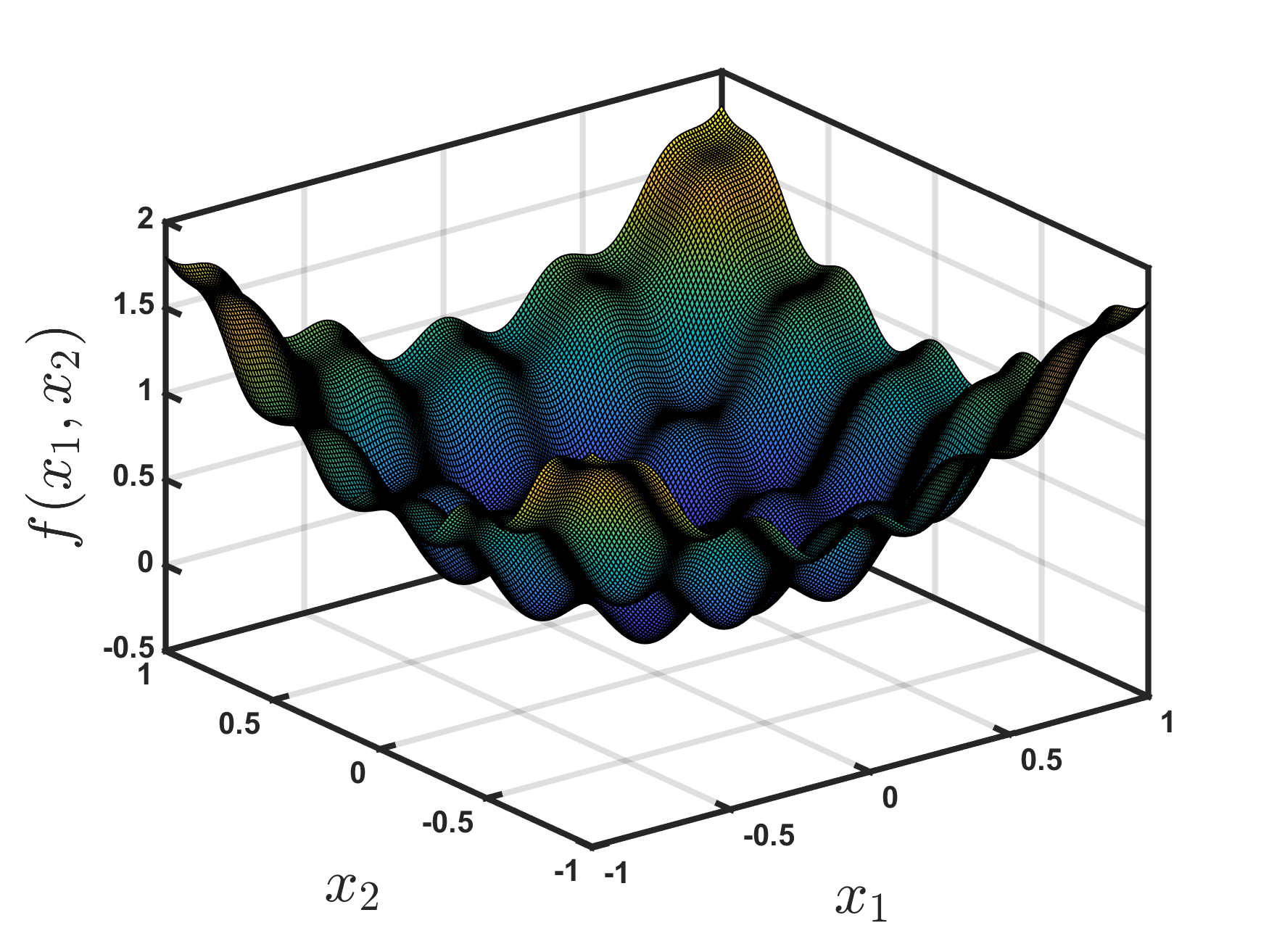}
\caption{\href{https://github.com/luclaurent/optiGTest/blob/master/unConstrained/funCosineMixture.m}{cosineMixture}}
\end{subfigure}
\begin{subfigure}[b]{0.24\textwidth}
\includegraphics[width=\textwidth]{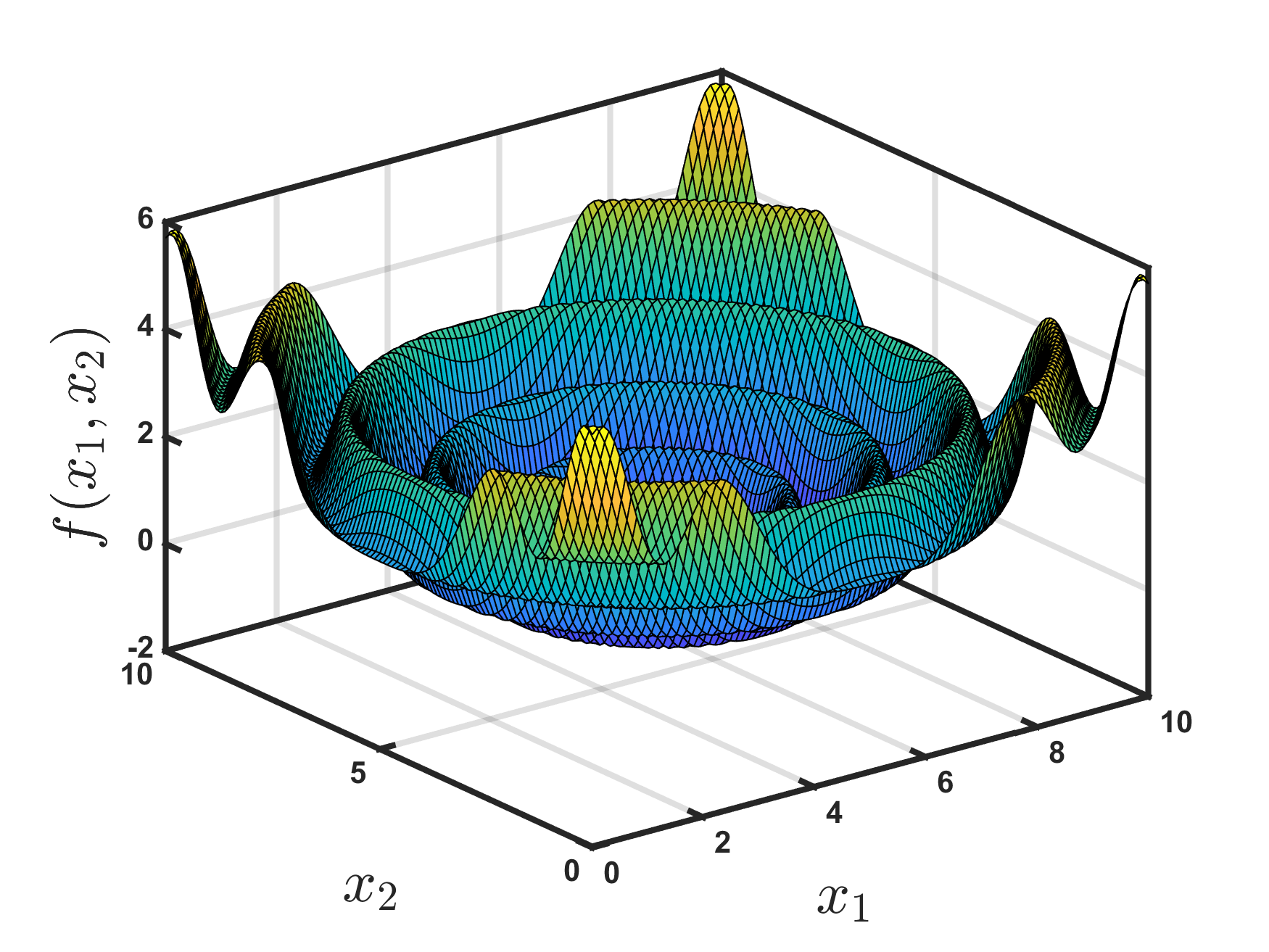}
\caption{\href{https://github.com/luclaurent/optiGTest/blob/master/unConstrained/funDeflectedCorrugatedSpring.m}{dc spring}}
\end{subfigure}
\begin{subfigure}[b]{0.24\textwidth}
\includegraphics[width=\textwidth]{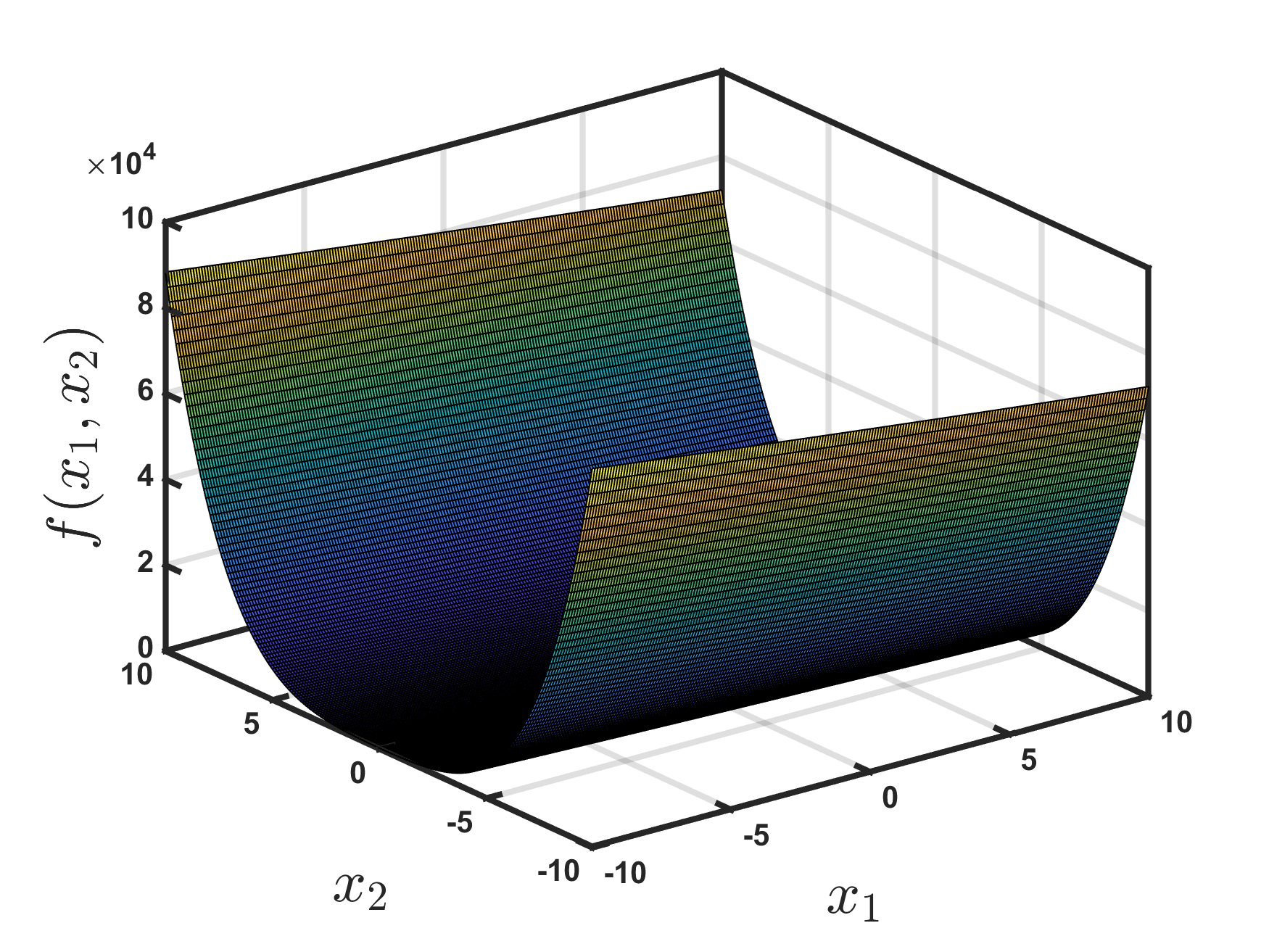}
\caption{\href{http://www.sfu.ca/~ssurjano/dixonpr.html}{dixonpr}}
\end{subfigure}
\begin{subfigure}[b]{0.24\textwidth}
\includegraphics[width=\textwidth]{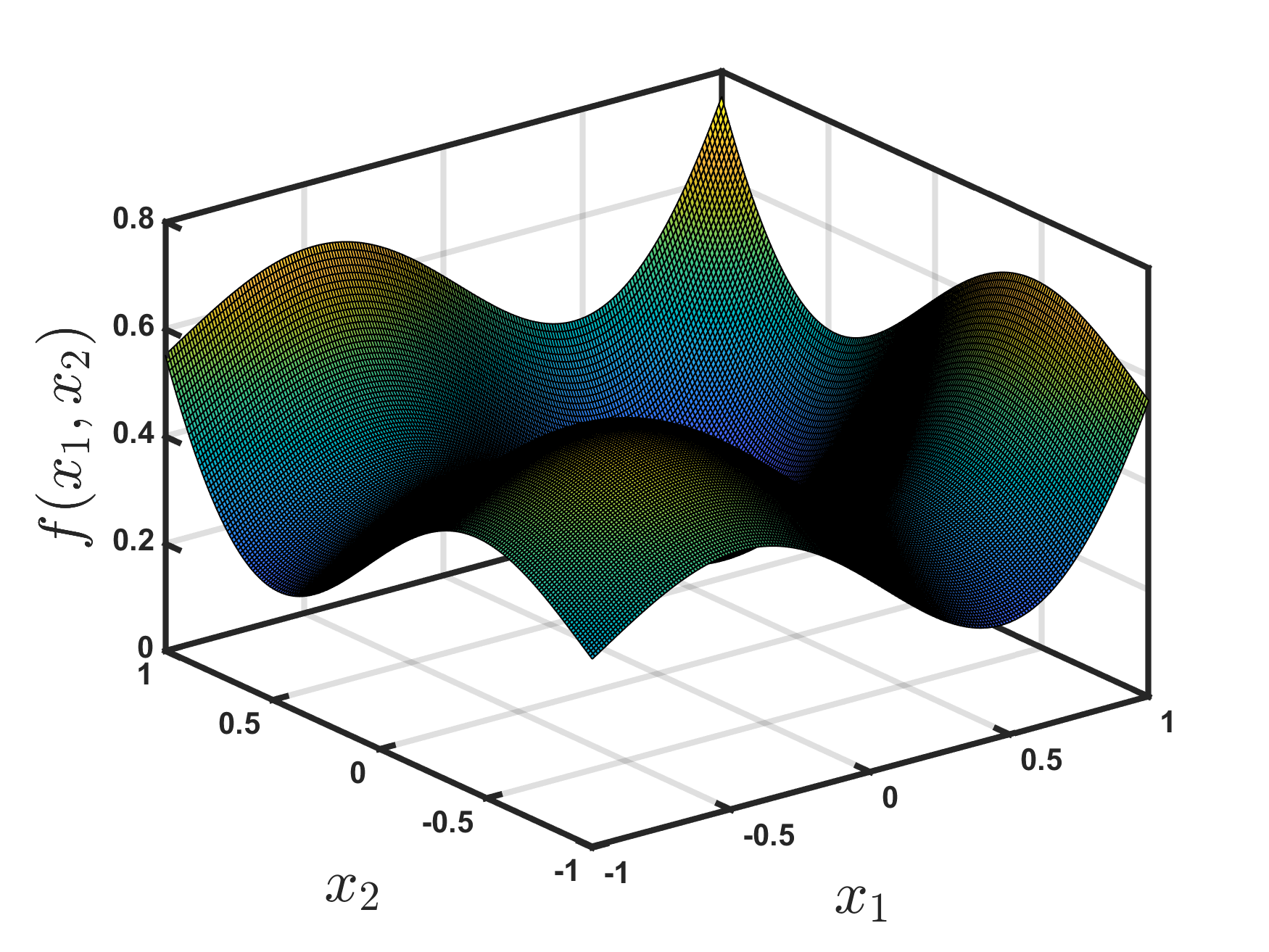}
\caption{\href{https://github.com/luclaurent/optiGTest/blob/master/unConstrained/funGiunta.m}{giunta}}
\end{subfigure}
\begin{subfigure}[b]{0.24\textwidth}
\includegraphics[width=\textwidth]{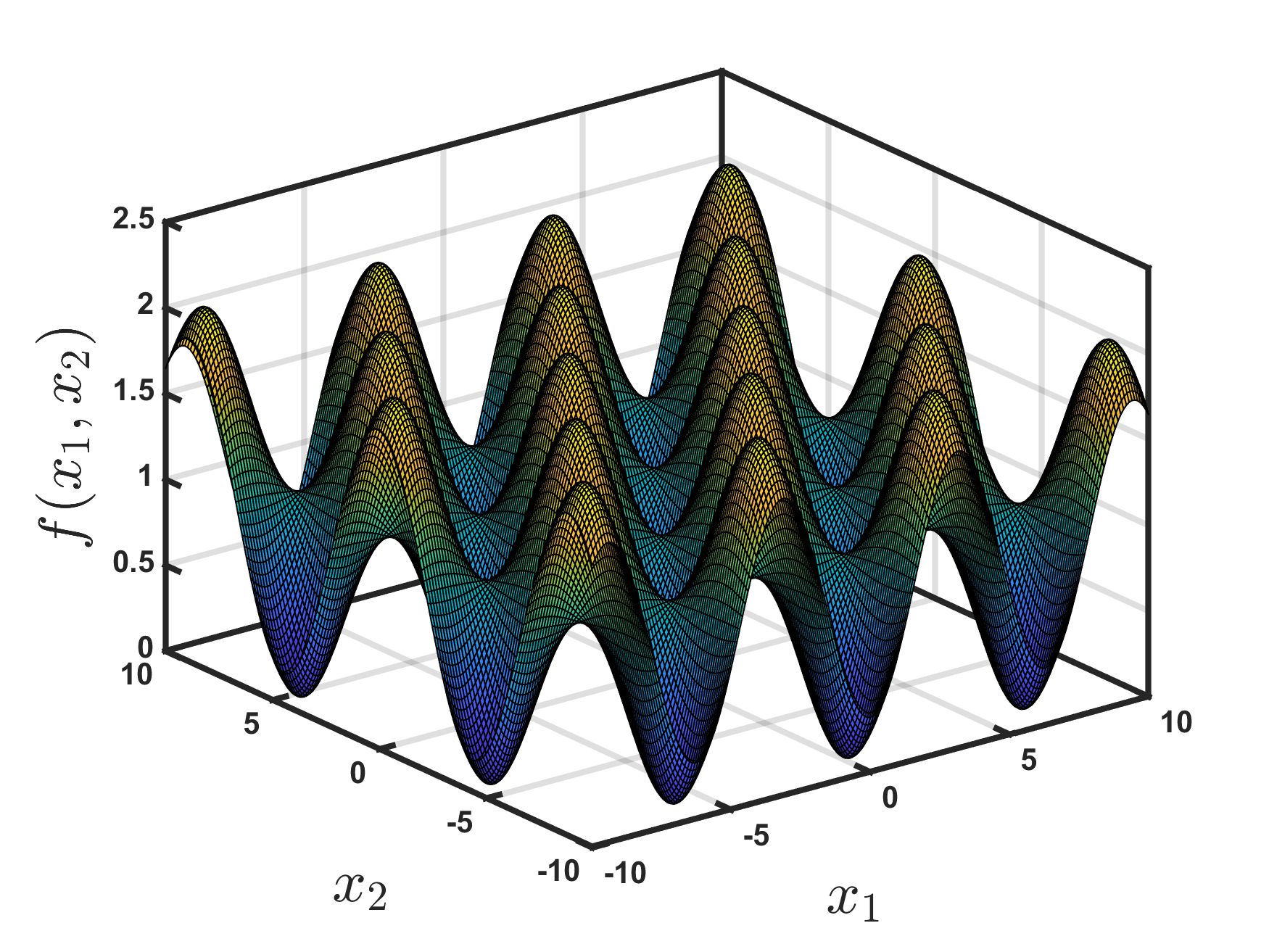}
\caption{\href{http://www.sfu.ca/~ssurjano/griewank.html}{griewank}}
\end{subfigure}
\begin{subfigure}[b]{0.24\textwidth}
\includegraphics[width=\textwidth]{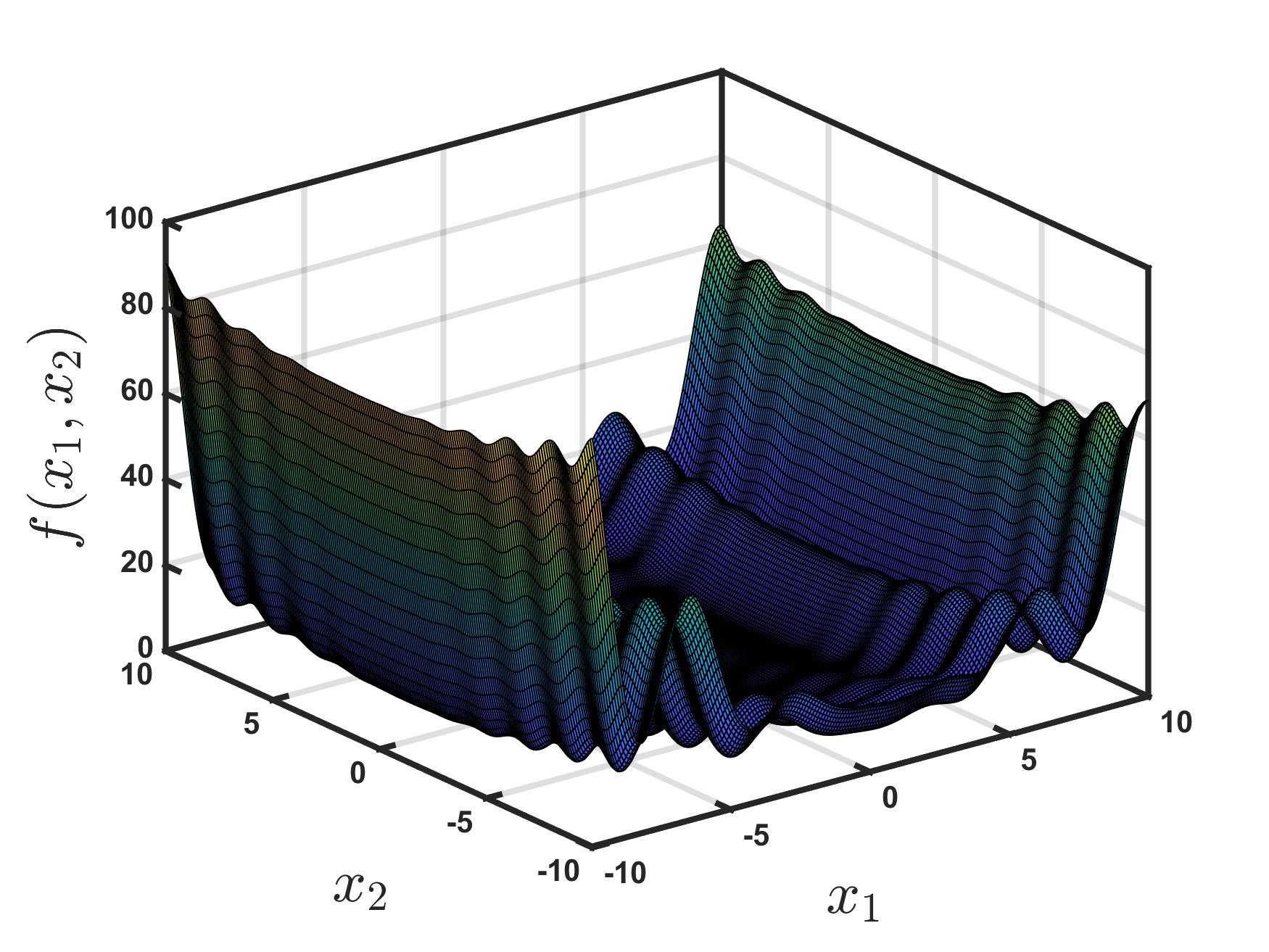}
\caption{\href{http://www.sfu.ca/~ssurjano/levy.html}{levy}}
\end{subfigure}
\begin{subfigure}[b]{0.24\textwidth}
\includegraphics[width=\textwidth]{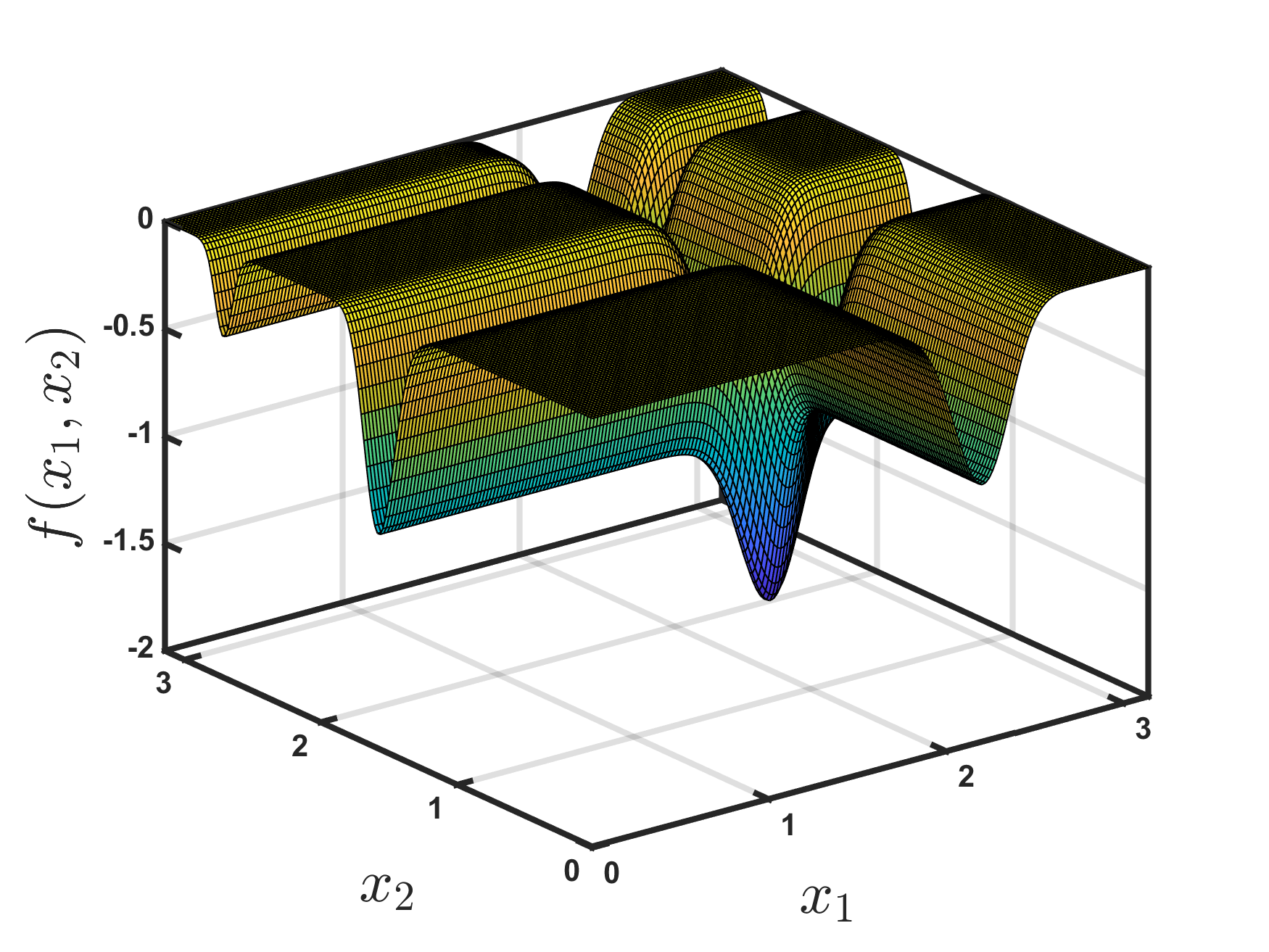}
\caption{\href{http://www.sfu.ca/~ssurjano/michal.html}{michal}}
\end{subfigure}
\begin{subfigure}[b]{0.24\textwidth}
\includegraphics[width=\textwidth]{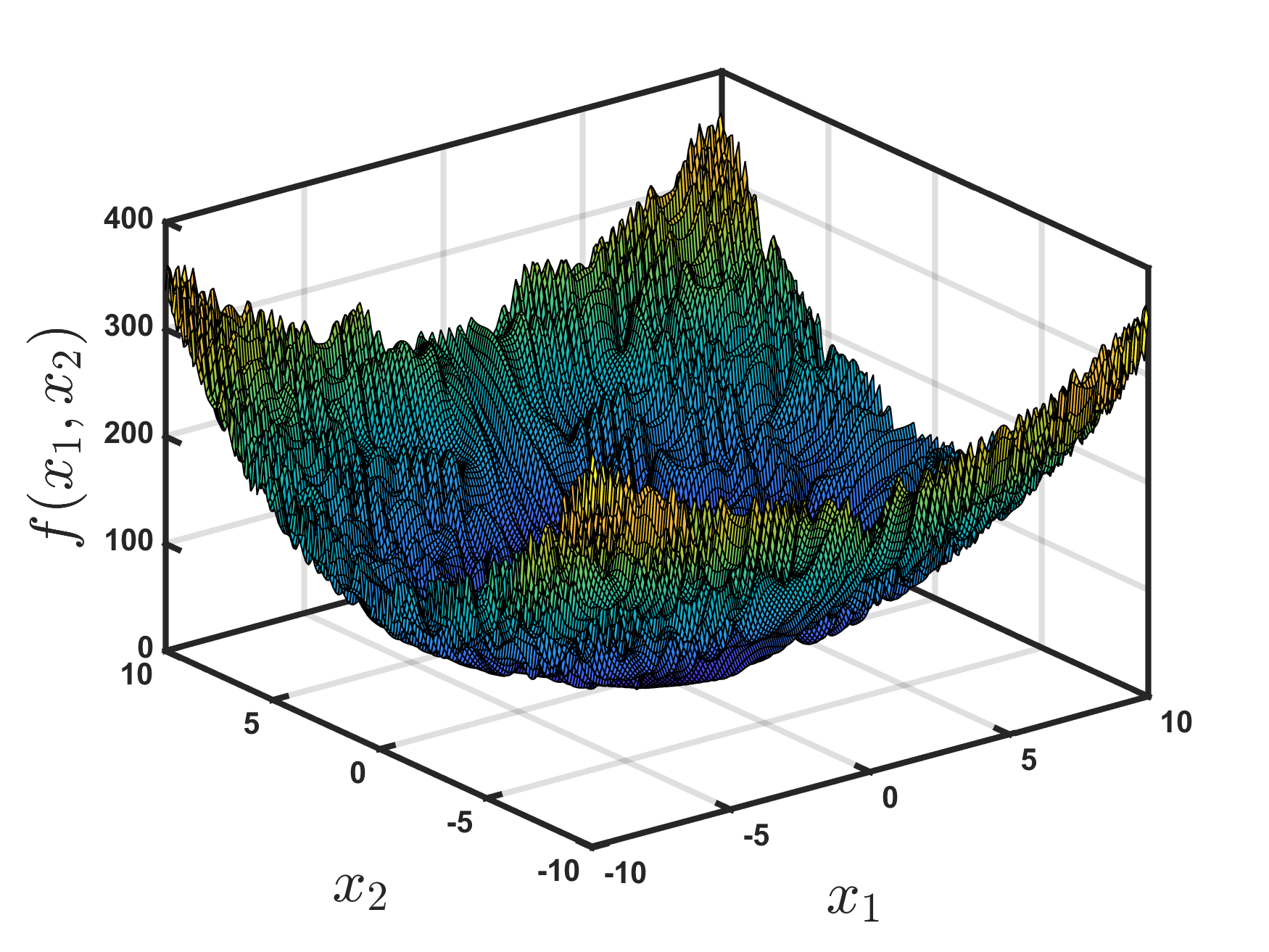}
\caption{\href{http://infinity77.net/global_optimization/test_functions_nd_P.html\#go_benchmark.Pinter}{pinter}}
\end{subfigure}
\begin{subfigure}[b]{0.24\textwidth}
\includegraphics[width=\textwidth]{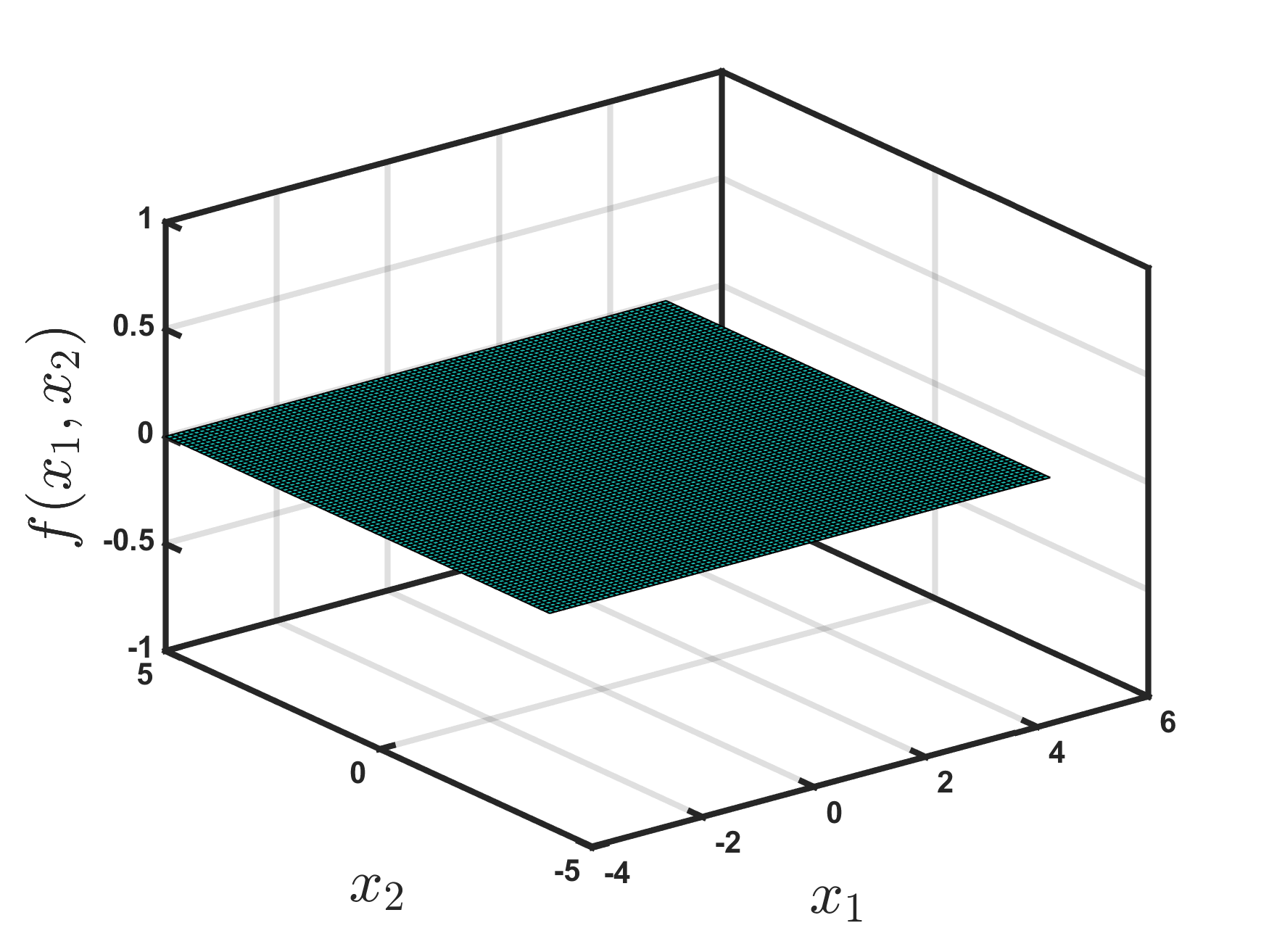}
\caption{\href{http://www.sfu.ca/~ssurjano/powell.html}{powell}}
\end{subfigure}
\begin{subfigure}[b]{0.24\textwidth}
\includegraphics[width=\textwidth]{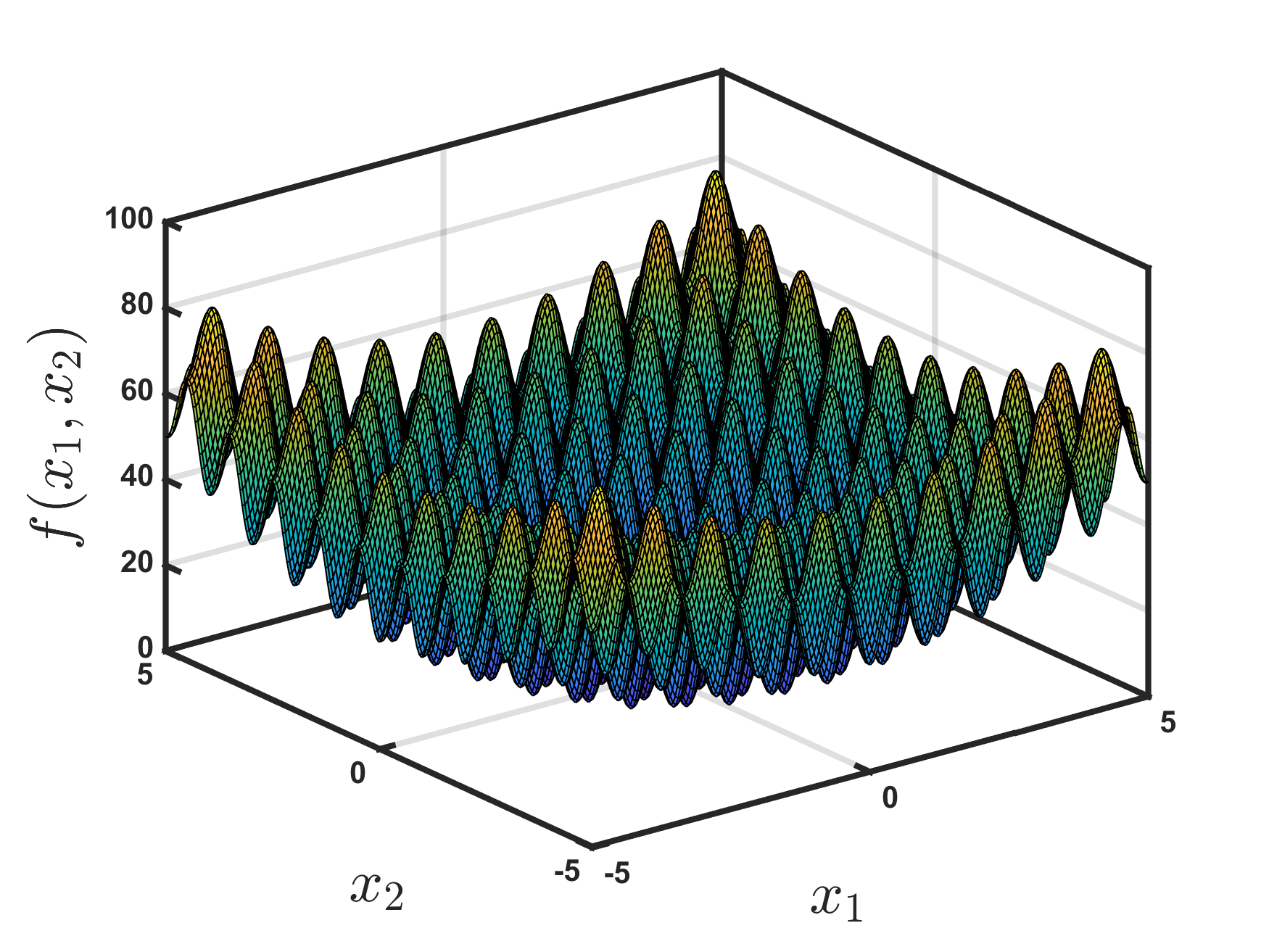}
\caption{\href{http://www.sfu.ca/~ssurjano/rastr.html}{rastr}}
\end{subfigure}
\begin{subfigure}[b]{0.24\textwidth}
\includegraphics[width=\textwidth]{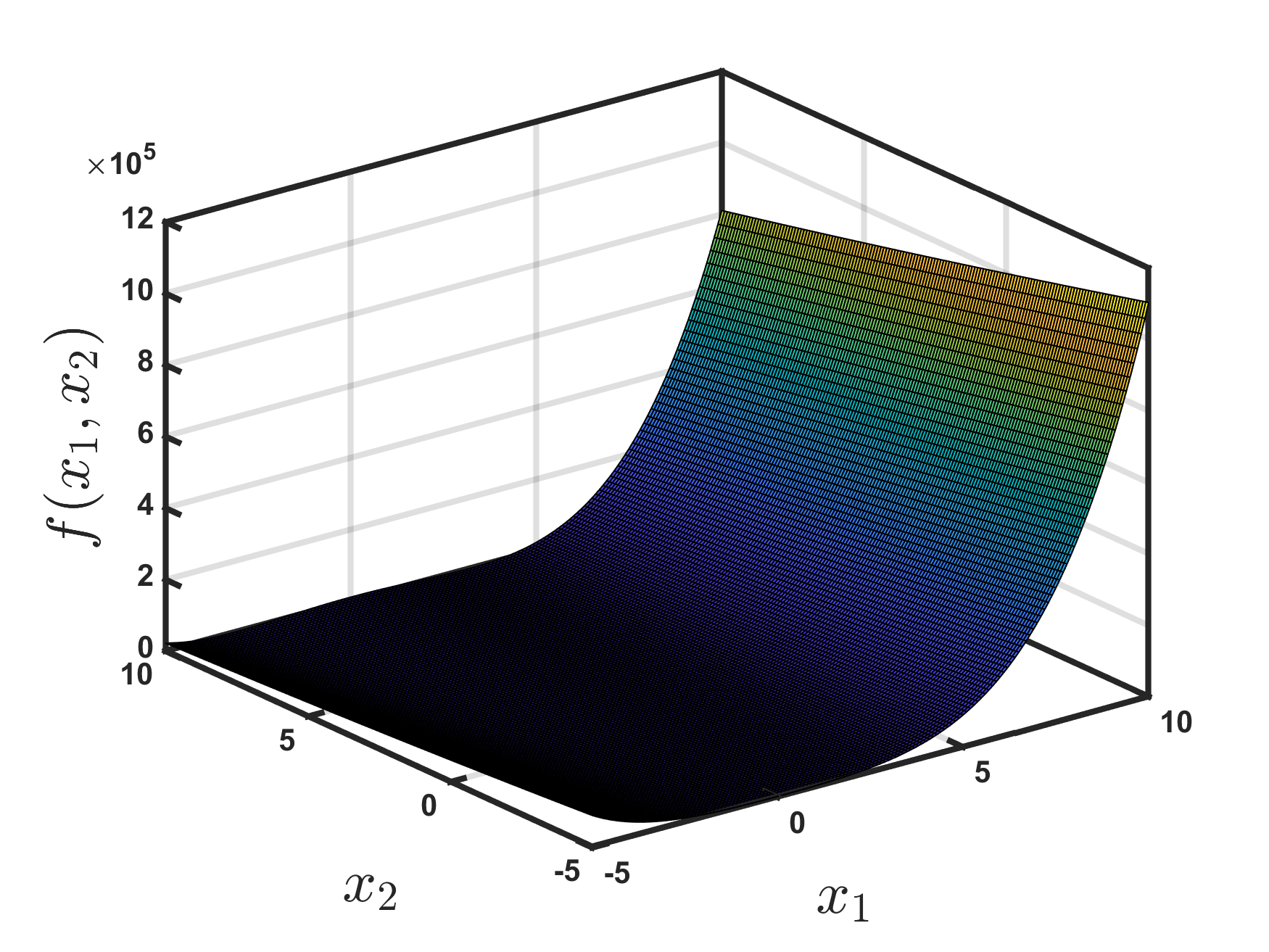}
\caption{\href{http://www.sfu.ca/~ssurjano/rosen.html}{rosen}}
\end{subfigure}
\begin{subfigure}[b]{0.24\textwidth}
\includegraphics[width=\textwidth]{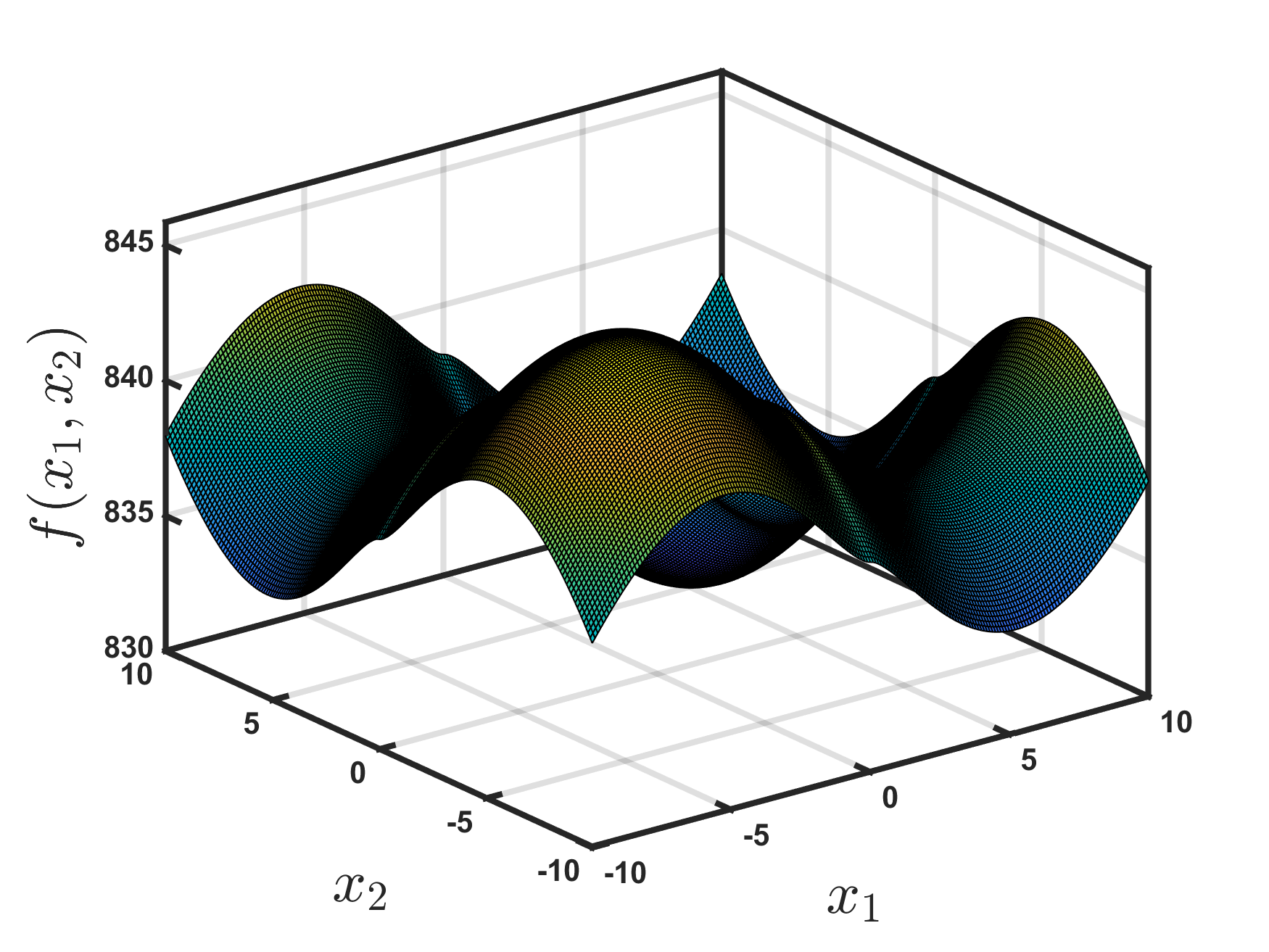}
\caption{\href{http://www.sfu.ca/~ssurjano/schwef.html}{schwef}}
\end{subfigure}
\begin{subfigure}[b]{0.24\textwidth}
\includegraphics[width=\textwidth]{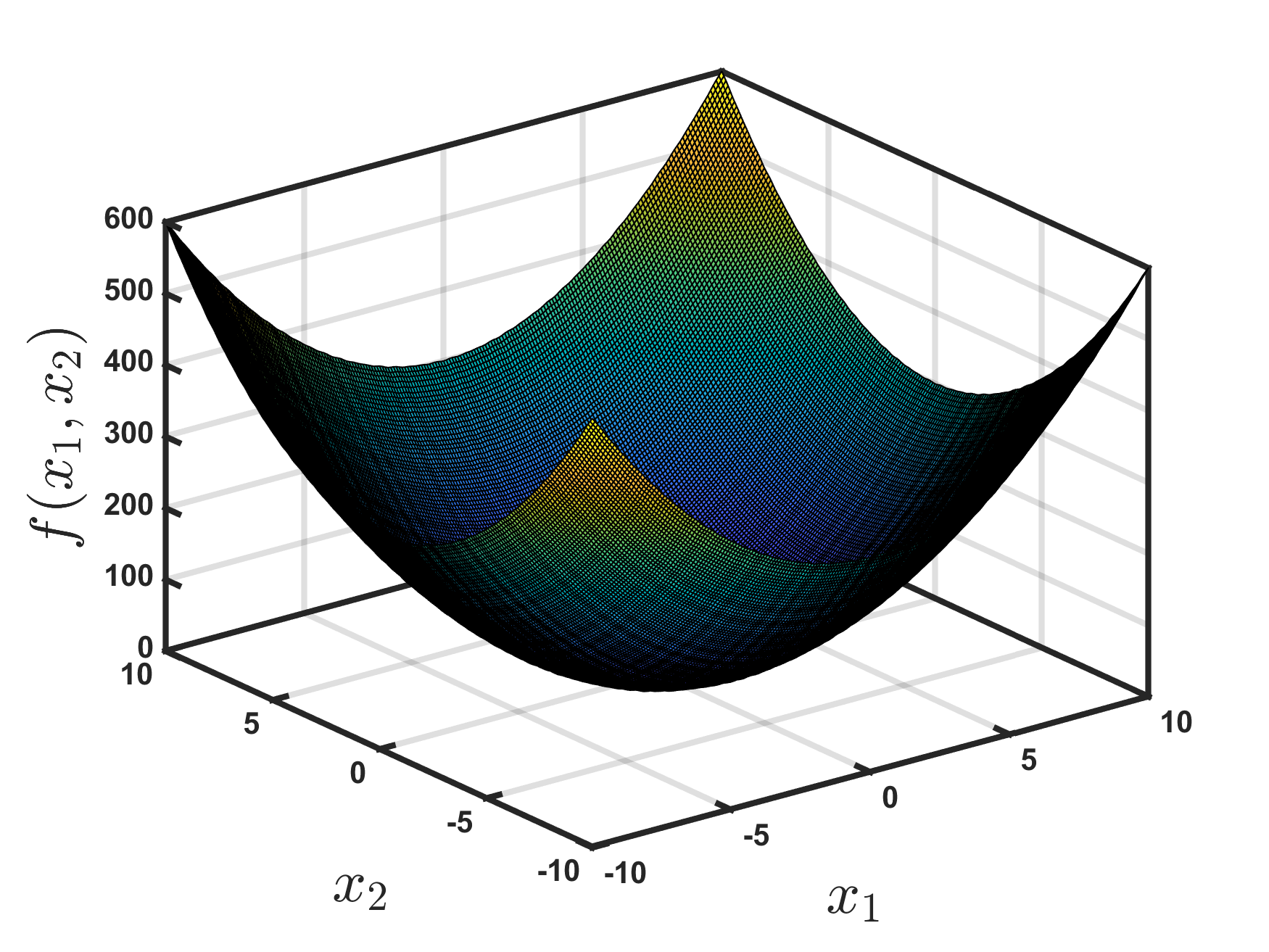}
\caption{\href{https://link.springer.com/article/10.1007/s00500-010-0650-7/figures/7}{shiftedBoha1}}
\end{subfigure}
\begin{subfigure}[b]{0.24\textwidth}
\includegraphics[width=\textwidth]{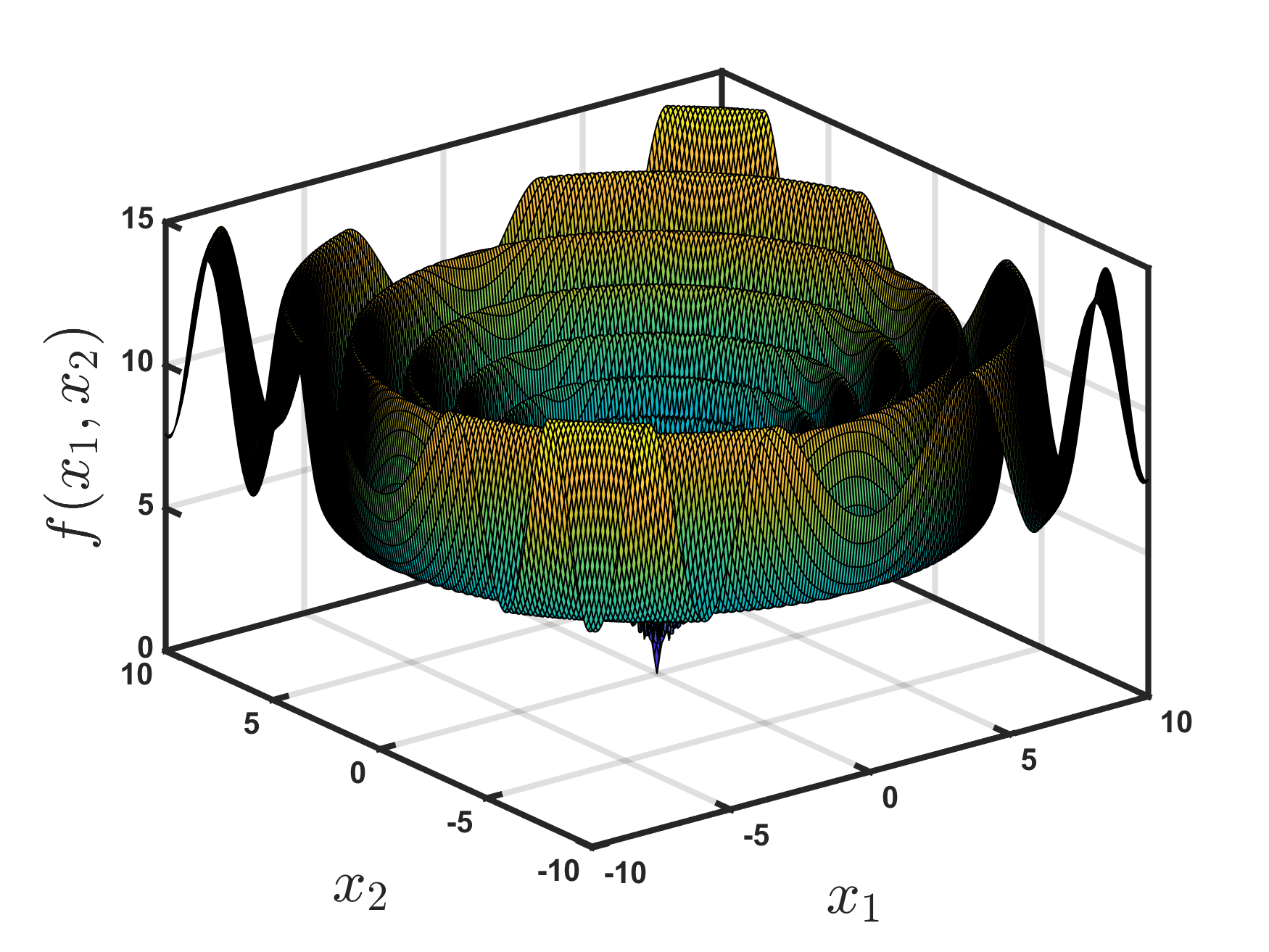}
\caption{\href{https://link.springer.com/article/10.1007/s00500-010-0650-7/figures/7}{shiftedSchaffer}}
\end{subfigure}
\begin{subfigure}[b]{0.24\textwidth}
\includegraphics[width=\textwidth]{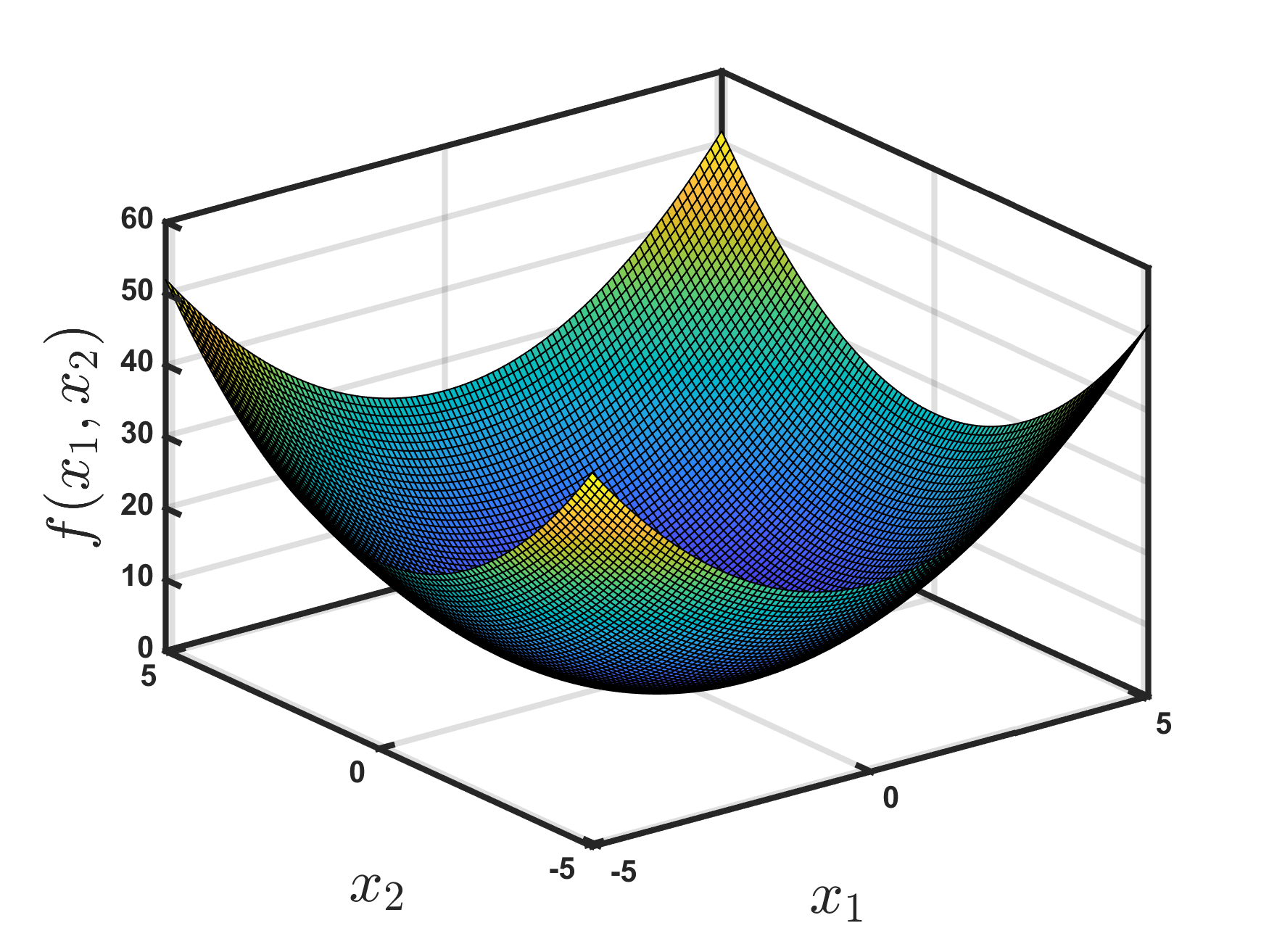}
\caption{\href{http://www.sfu.ca/~ssurjano/spheref.html}{spheref}}
\end{subfigure}
\begin{subfigure}[b]{0.24\textwidth}
\includegraphics[width=\textwidth]{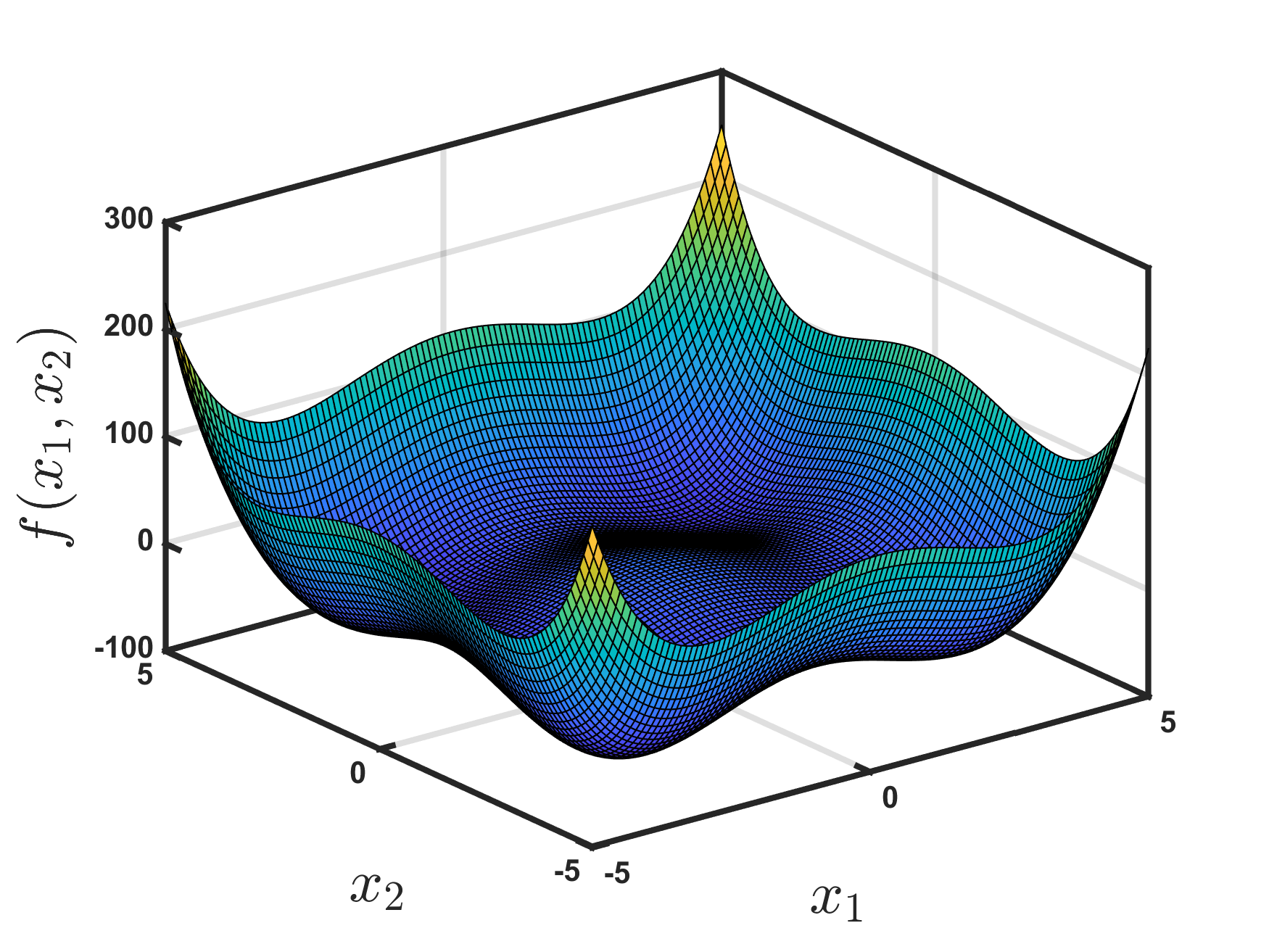}
\caption{\href{http://www.sfu.ca/~ssurjano/stybtang.html}{stybtang}}
\end{subfigure}
\begin{subfigure}[b]{0.24\textwidth}
\includegraphics[width=\textwidth]{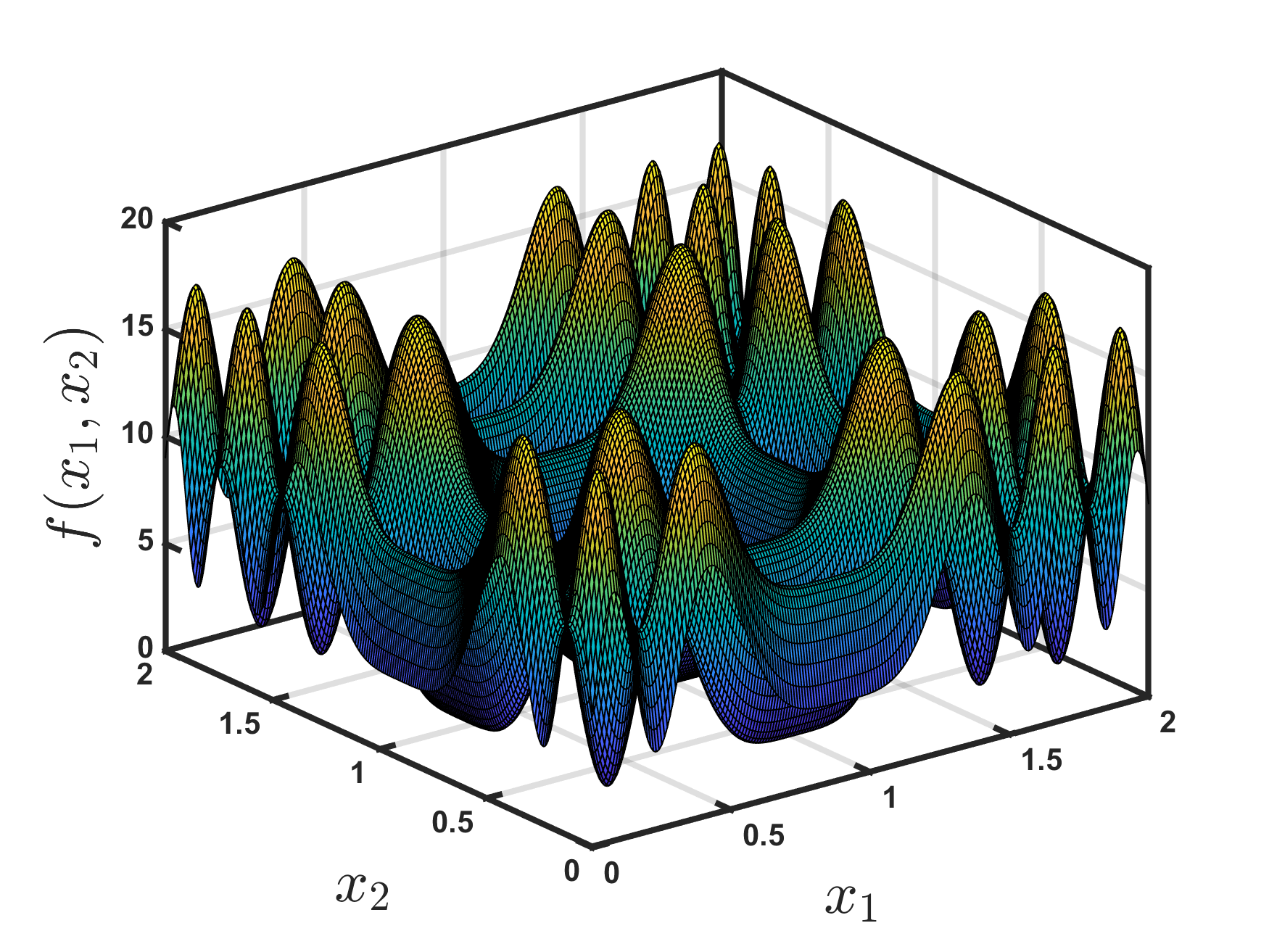}
\caption{\href{https://github.com/luclaurent/optiGTest/blob/master/unConstrained/funTrigonometric2.m}{trigonometric2}}
\end{subfigure}
\begin{subfigure}[b]{0.24\textwidth}
\includegraphics[width=\textwidth]{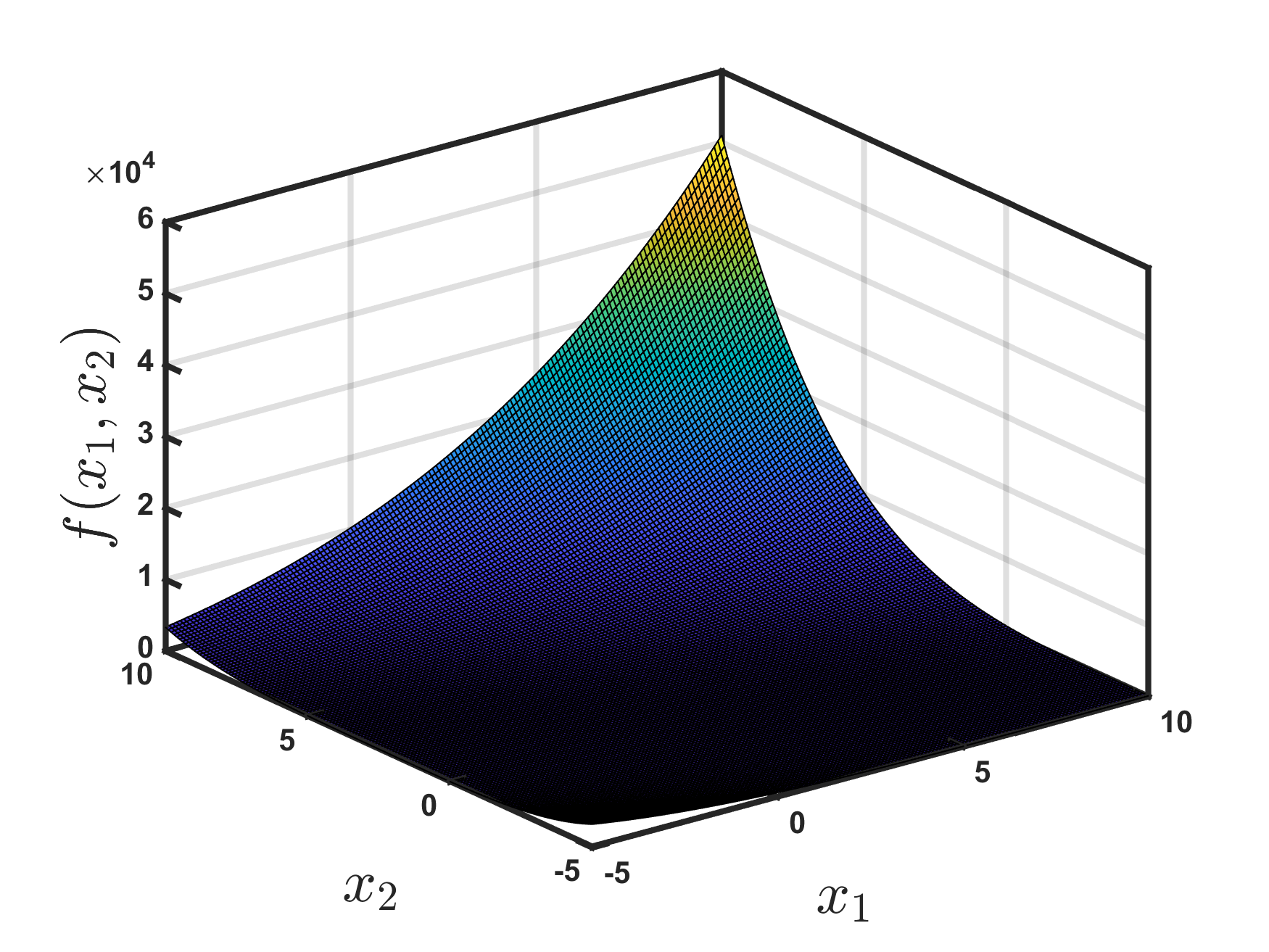}
\caption{\href{http://www.sfu.ca/~ssurjano/zakharov.html}{zakharov}}
\end{subfigure}
\caption{Gallery of test functions shown in two dimensions only. For some functions, the lower and upper bounds have been modified from those used in our computational experiments (see Table~\ref{table:test_fnc_table}) to improve the visualization.}
\label{fig:test_fnc_gallery}
\end{figure}

\subsubsection{Black-box optimization solvers compared}

While there are dozens of BBO solvers in the literature \cite{rios2013derivative}, we chose to compare Particle Swarm Optimization (PSO) and NOMAD \citep{Le2011a} due to their popularity and consistently strong performance in the BBO community.  Additionally, their contrasting search philosophies offer an interesting counterpoint in the context of solution polishing. We describe both solvers below and how we invoke them. 

PSO is a metaheuristic optimization algorithm, capable of handling both smooth and nonsmooth functions, inspired by the social behavior of birds and fish. We used Matlab's PSO implementation \url{www.mathworks.com/help/gads/particleswarm.html}. PSO treats individual solutions as particles in a multidimensional space and iteratively refines their positions based on their own experience and the collective information shared within the swarm. We set the number of particles to 20 for all the experiments. We set a seed to obtain the candidate solutions and allow the algorithm to choose its initial particles based on the seed. We sort the final 20 particles in ascending order according to objective function value. We place the first solution in our set $\mc{S}$ of elite solutions.  We then include the second best solution in $\mc{S}$, if it is not ``too close'' to the first solution. We continue including solutions, unless they are ``too close'' to a solution already in $\mc{S}$.  For solution polishing with PSO, we give the 5 elite solutions as initial particles and allow PSO to choose the other 15 initial particles. Since Matlab's PSO does not take initial function values as parameters in most of its versions, we give only the coordinates for each elite solution and allow the algorithm to compute their function values. The maximum number of function evaluations is $(N_I + 1)N_p$, where $N_I$ is the maximum number of iterations and $N_p$ the number of particles. We set the maximum number of iterations accordingly as that is the parameter within our control. 

NOMAD (Nonlinear Optimization with the MADS algorithm) \citep{Le2011a} is ``a C++ implementation of the Mesh Adaptive Direct Search algorithm (MADS), designed for difficult blackbox optimization problems'' \url{https://www.gerad.ca/en/software/nomad/}. With the ability to handle nonsmooth functions, constraints, and integer/categorical decision variables, it remains one of the most consistent and dominant DFO solvers in practice \cite{rios2013derivative}. 
Except for setting $\texttt{min\_mesh\_size} = \num{1e-4}$ and $\texttt{initial\_mesh\_size} = 10$, we use the default parameter settings. 

\begin{algorithm} 
\caption{\texttt{generateEliteSolutions($f,D,\v{x}^L,\v{x}^U,K,N)$}: Generate a set of $K \geq 1$ elite solutions given a function $f:\Re^D \mapsto \Re$ on the domain $[\v{x}^L,\v{x}^U] \subset \Re^D$ and a positive integer $N$ (e.g., the number of grid indices used in LineWalker)}
\label{algo:generate_elite_solutions}
\begin{algorithmic}[1]
\State $\mc{S} = \emptyset$ \Comment{Initialize the set $\mc{S}$ of elite solutions}
\State \texttt{options = nomadset('max\_bb\_eval',50*$D$)} \Comment{Define NOMAD options}
\While{$(|\mc{S}| < K)$} 
	\State $\v{x}^0 = \texttt{rand}(D,1)*(\v{x}^U - \v{x}^L) + \v{x}^L$ \Comment{Choose a random initial solution}
	\State $[\tilde{\v{x}},\tilde{f}] = \texttt{nomad}(f,\v{x}^0,\v{x}^L,\v{x}^U,\texttt{options})$ \Comment{Invoke NOMAD} \label{step:invoke_NOMAD}
	\State $V = \max\{ |\hat{x}_d| : \hat{\v{x}} \in \mc{S}, d=1,\dots,D\} $ \Comment{Compute the max coordinate value} 
	\State $\texttt{tol} = \frac{x_1^U - x_1^L}{N V}$ \Comment{Set the distance tolerance} \label{step:set_tolerance} 
	\If{$\texttt{dist}(\tilde{\v{x}},\hat{\v{x}})>\texttt{tol}~\forall \hat{\v{x}} \in \mc{S}$} \label{step:check_tolerance} \Comment{Ensure that $\tilde{\v{x}}$ is sufficiently far away}
		\State $\mc{S} \gets \mc{S} \cup \{\tilde{\v{x}}\}$ \Comment{Include $\tilde{\v{x}}$ in the set of elite solutions}
	\EndIf 	
\EndWhile
\State \textbf{return} $\mc{S}$ \Comment{Return the set of elite solutions}
\end{algorithmic}
\end{algorithm}

Algorithm~\ref{algo:generate_elite_solutions} provides pseudocode for how we generate $K$ elite solutions using NOMAD. The algorithm is straightforward except for one nuance - a requirement that all elite solutions are sufficiently far apart.  After generating a random initial solution in the box $[\v{x}^L,\v{x}^U]$, we invoke NOMAD given a maximum of $50D$ function evaluations.  If NOMAD returns a solution that is sufficiently far from all previously found elite solution, then we accept the solution.  Otherwise, we repeat the search until we have found $K$ elite solutions. 
Note that, in step~\ref{step:check_tolerance}, we only accept a new solution $\tilde{\v{x}}$ in the set $\mc{S}$ of elite solutions if it is sufficiently far from all previously found elite solutions, even if its objective function value is better. The criterion for ``sufficiently far'' is determined in step~\ref{step:set_tolerance} where we set a tolerance \texttt{tol} based on the number $N$ of grid indices that we use in \texttt{LineWalker}. The tolerance prevents the possibility of two elite solutions being mapped to the same \texttt{LineWalker} grid index. 

The appendix \ref{sec:NOMAD_profiles} shows detailed profiles of NOMAD's empirical convergence behavior for each of the 19 test functions in 2, 4, 8, and 16 dimensions.  Specifically, the profiles depict the number of black-box function evaluations to reach a ``local optimum,'' a point where no more improvement of its best objective value is observed. Because the results indicate that NOMAD is capable of converging within $50D$ function evaluations, we chose to set the maximum number of function evaluations to $50D$ for NOMAD and PSO when generating their elite solutions, i.e., before attempt to invoke solution polishing.

\subsubsection{Solution polishing methods compared}

As stated above, we partition our numerical experiments according to which BBO solver that we use to generate elite solutions and to polish solutions in $D$ dimensions.  Specifically, if we use PSO to generate elite solutions, then we also use PSO to polish these solutions and compare its performance against our one-dimensional solution polishing algorithms.  The same is true for NOMAD.  Although it certainly could be done, we do not attempt to use PSO to generate elite solutions and then use NOMAD to polish those solutions, or vice versa.

Under this backdrop, given a set $\mc{S}$ of elite solutions generated by PSO or NOMAD, we compare the following approaches, each of which is given at most 290 function evaluations, to perform solution polishing:
\begin{itemize}
\item \textbf{Straight Linking Line Search (LS)} (multiple 1-dimensional searches): For each of the ``$|\mc{S}|$ choose 2'' (e.g., 10) pairs of elite solutions, we allow LineWalker to search using at most 30 functions evaluations on the ``extended'' straight line segment. Specifically, given a pair of elite solutions $\hat{\v{x}}^1$ and $\hat{\v{x}}^2$, we construct a line segment through these two points whose endpoints touch the boundary of the box $[\v{x}^L,\v{x}^U]$. Since we pass $\hat{\v{x}}^1$ and $\hat{\v{x}}^2$, along with their corresponding evaluated function values $f(\hat{\v{x}}^1)$ and $f(\hat{\v{x}}^2)$, to LineWalker, we only allow LineWalker to make an additional $28 (=30-2)$ function evaluations for each line segment. Lastly, since LineWalker may not evaluate its predicted minimizer, we perform one final function evaluation.  In total, we perform 29 additional function evaluations per line segment.  Since there are 10 pairs of elite solutions, this approach uses 290 total function evaluations.  
\item \textbf{Multipoint curve (MP)} (1-dimensional search): We invoke Algorithm~\ref{algo:generate_mp_curve} to generate a curve through all $|\mc{S}|$ elite solutions.  Then, we use LineWalker to search along this curve using a maximum of 290 function evaluations. 
\item \textbf{Propeller curve (Prop)} (1-dimensional search): We invoke Algorithm~\ref{algo:generate_p_curve} to generate a propeller-shaped curve through the best solution (i.e., the one with the smallest objective function value) in $\mc{S}$. We then employ LineWalker, with a maximum of 290 function evaluations, to search this curve for an improving solution. 
\item \textbf{BBO solver (PSO or NOMAD)} ($D$-dimensional search): If PSO was used to generate $|\mc{S}|$ elite solutions, then we initialize PSO with 20 particles (solutions) by providing these $|\mc{S}|$ elite solutions, while allowing PSO to select $20-|\mc{S}|$ remaining solutions at random.  Otherwise, NOMAD was used to generate elite solutions and we initialize NOMAD's search from the best-known solution. The BBO solver then searches the entire $D$-dimensional space for an improving solution. 
\end{itemize}

As a final point of comparison, we note that the BBO solvers used $50DK$ function evaluations to generate their $K=5$ elite solutions, or a total of 500, 1000, 2000, and 4000 function evaluations in 2, 4, 8, and 16 dimensions, respectively. In contrast, solution polishing was given at most 290 function evaluations. In 2D, one could argue that we allot too many function evaluations to solution polishing relative to the budget used to generate elite solutions. However, we believe that the solution polishing budget allocation is sensible for larger dimensions. 

\subsection{Numerical Results} \label{sec:numerical_results}

This section highlights our main numerical results. Since it makes little sense to evaluate how well a solution polishing method can improve a solution that is already deemed optimal, we only consider instances that the DFO solvers NOMAD and PSO did not solve to global optimality. We follow standard requirements used in the DFO community in which ``A solver [is] considered to have successfully solved a problem [i.e., optimized a test function] if it returned a solution with an objective function value within 1\% or 0.01 of the global optimum, whichever was larger'' \cite{ploskas2021review,rios2013derivative}.
Mathematically, let $f^{*,\textrm{true}}$ denote the optimal objective function value of the true underlying function. Let $x^{*,\textrm{eval}}$ denote the sample with the smallest evaluated objective function value. A test function is declared ``solved'' if 
\begin{equation} \label{def:instance_is_solved}
|\fhat(x^{*,\textrm{eval}})-f^{*,\textrm{true}}| \leq 0.01 \max\{1,|f^{*,\textrm{true}}|\}. 
\end{equation}

We use the metric ``optimality gap closed'' (or simply ``gap closed'') to measure the improvement obtained by a solution polishing method.
The definition for ``optimality gap closed'' expressed as a percentage is given by: 
\begin{equation}
\gamma(\fBefore,\fAfter) = \left[ 1 - \frac{f^{*,\textrm{true}} - \fAfter}{f^{*,\textrm{true}}-\fBefore} \right] 100 \%
\end{equation}
where $f^{\textrm{before}}$ and $f^{\textrm{after}}$ denote the minimum objective function value evaluated before and after solution polishing, respectively.
Observe that the optimality gap closed can only be computed for instances where we have not found the true minimum before polishing. If we have found the minimum, we run into division by zero. 
The optimality gap closed tells us the percentage of the best possible improvement that we have obtained via solution polishing. 

\begin{figure} [h!]
\begin{subfigure}[b]{0.49\textwidth}
\includegraphics[width=\textwidth]{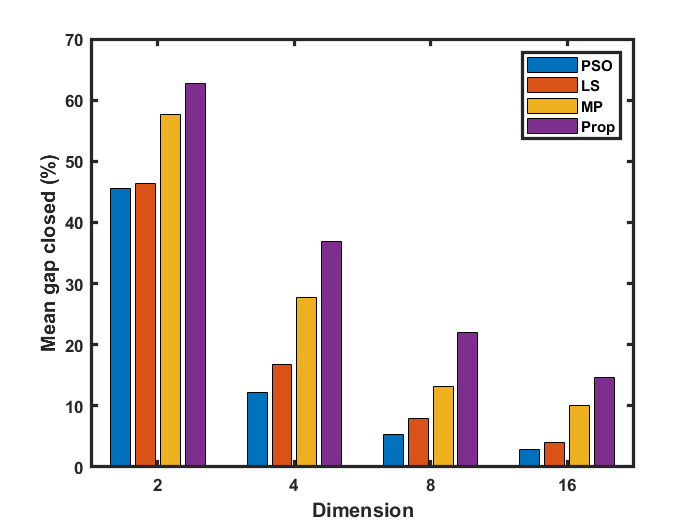}
\caption{PSO comparison}
\label{fig:mean_gap_closed_pso}
\end{subfigure}
\begin{subfigure}[b]{0.49\textwidth}
\includegraphics[width=\textwidth]{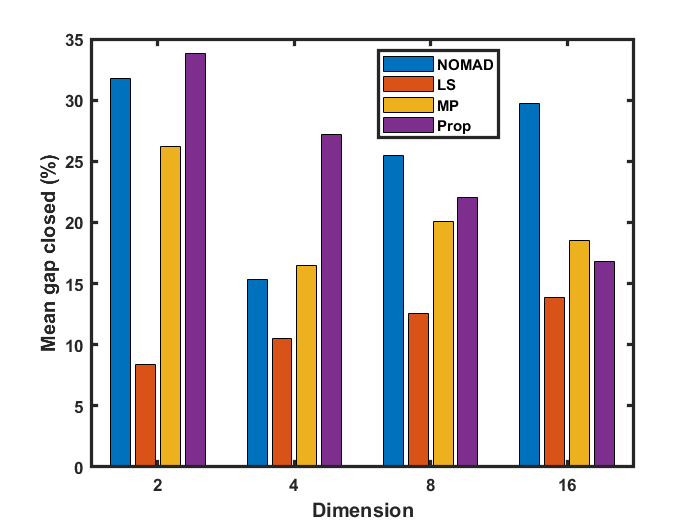}
\caption{NOMAD comparison}
\label{fig:mean_gap_closed_nomad}
\end{subfigure}
\caption{Mean optimality gap closed using elite solutions generated by (a) PSO and (b) NOMAD.}
\label{fig:mean_gap_closed}
\end{figure}

\begin{figure} [h!]
\begin{subfigure}[b]{0.24\textwidth}
\includegraphics[width=\textwidth]{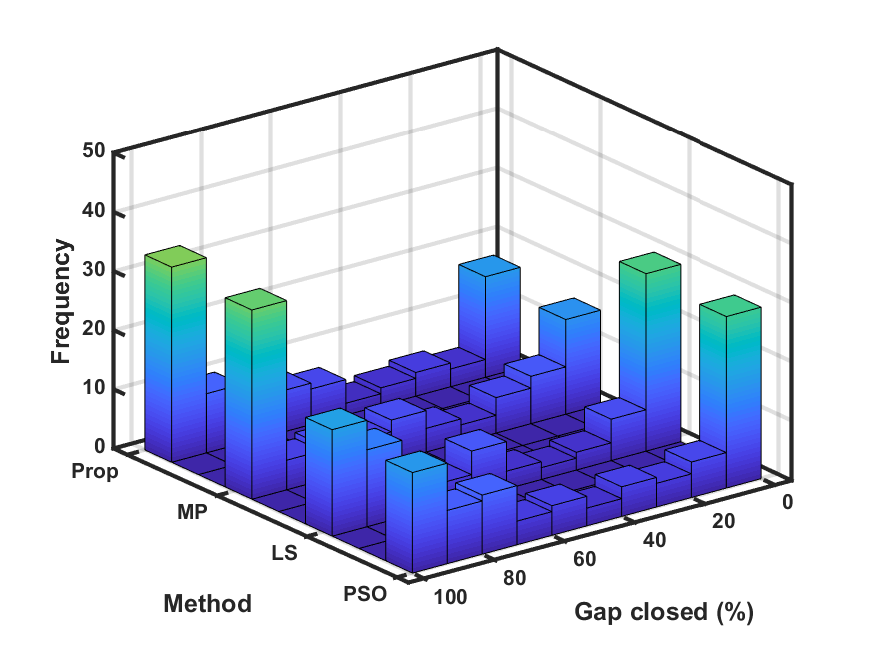}
\caption{PSO 2D}
\label{fig:histogram_gap_closed_pso_2D}
\end{subfigure}
\begin{subfigure}[b]{0.24\textwidth}
\includegraphics[width=\textwidth]{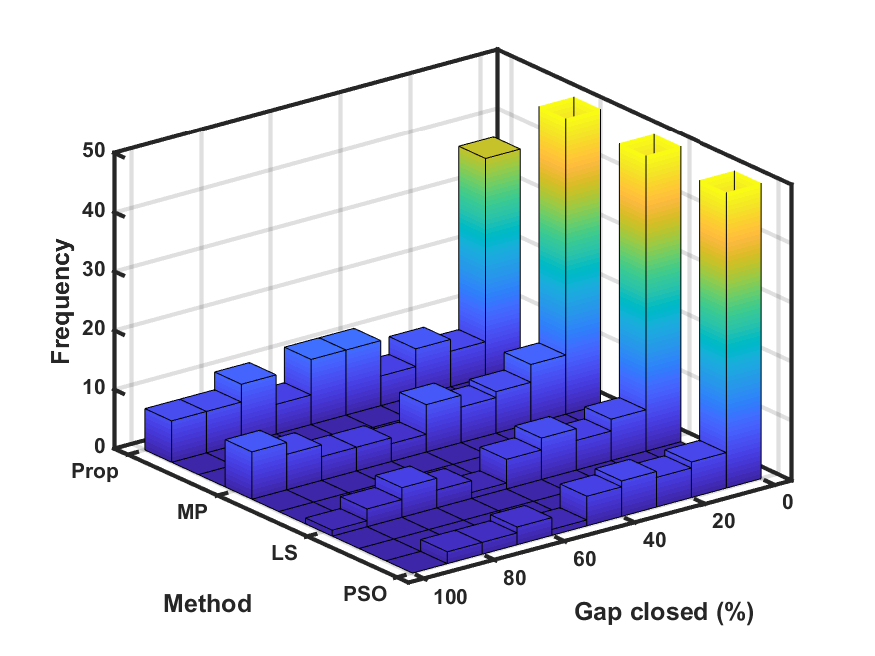}
\caption{PSO 4D}
\label{fig:histogram_gap_closed_pso_4D}
\end{subfigure}
\begin{subfigure}[b]{0.24\textwidth}
\includegraphics[width=\textwidth]{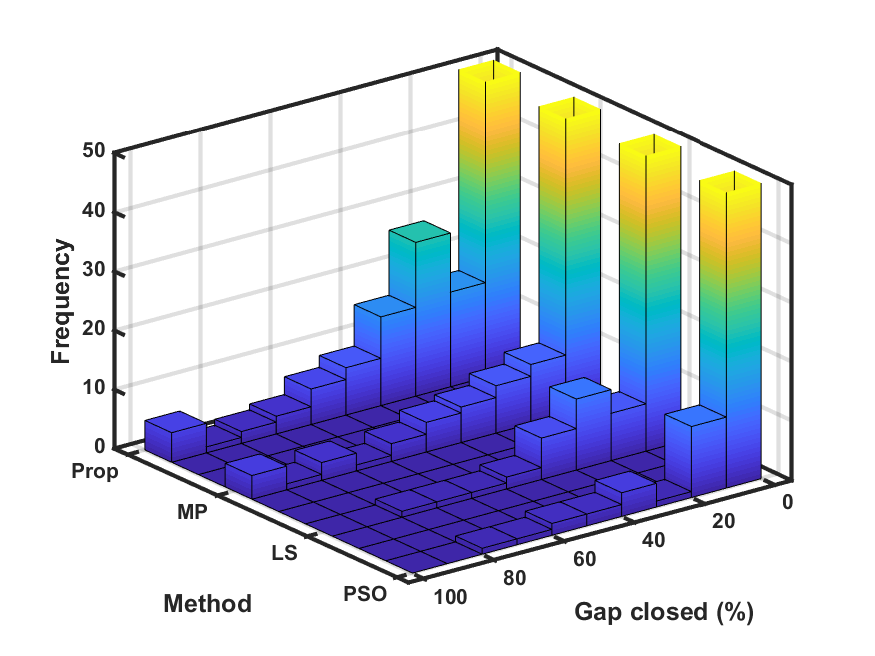}
\caption{PSO 8D}
\label{fig:histogram_gap_closed_pso_8D}
\end{subfigure}
\begin{subfigure}[b]{0.24\textwidth}
\includegraphics[width=\textwidth]{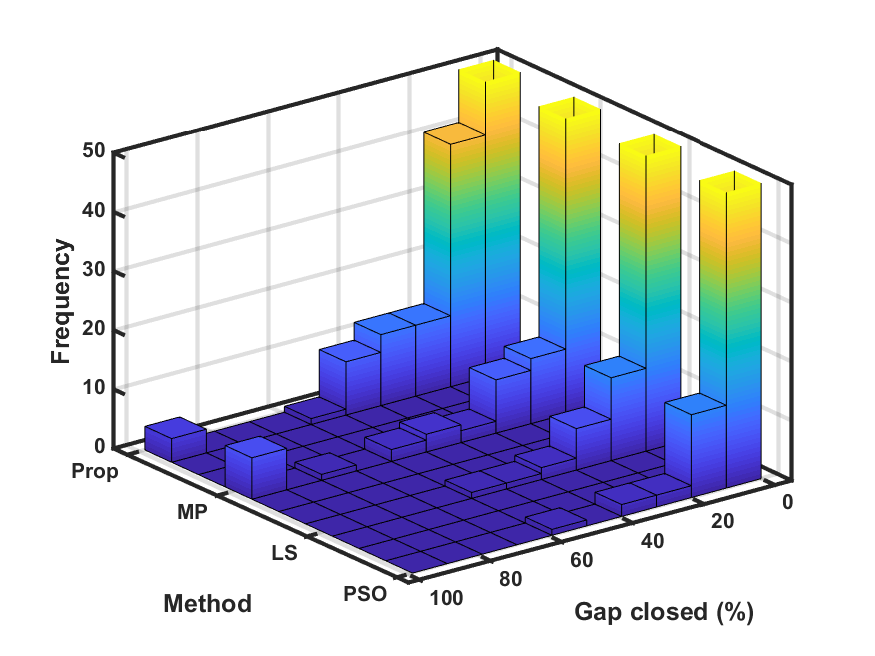}
\caption{PSO 16D}
\label{fig:histogram_gap_closed_pso_16D}
\end{subfigure}
\begin{subfigure}[b]{0.24\textwidth}
\includegraphics[width=\textwidth]{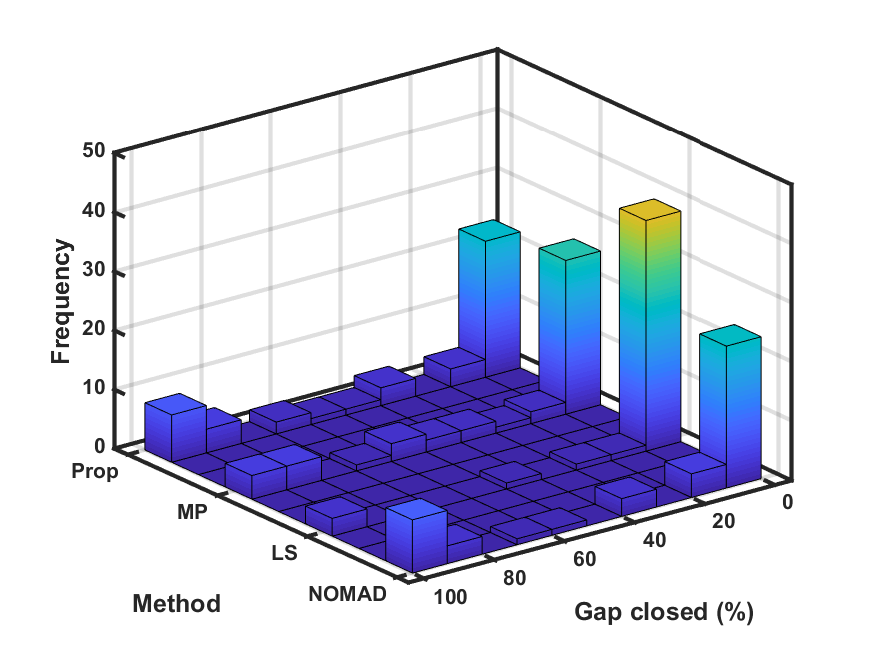}
\caption{NOMAD 2D}
\label{fig:histogram_gap_closed_nomad_2D}
\end{subfigure}
\begin{subfigure}[b]{0.24\textwidth}
\includegraphics[width=\textwidth]{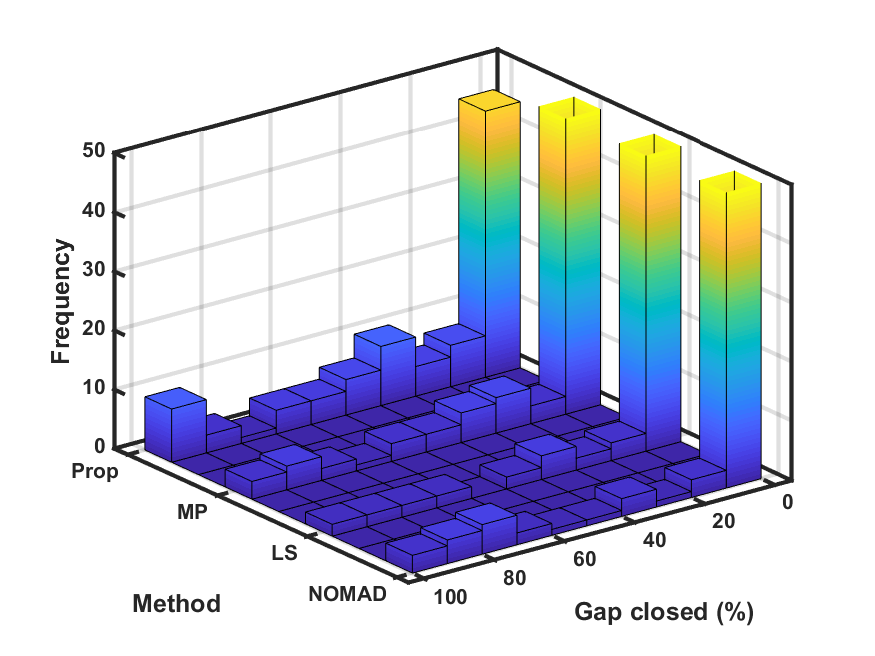}
\caption{NOMAD 4D}
\label{fig:histogram_gap_closed_nomad_4D}
\end{subfigure}
\begin{subfigure}[b]{0.24\textwidth}
\includegraphics[width=\textwidth]{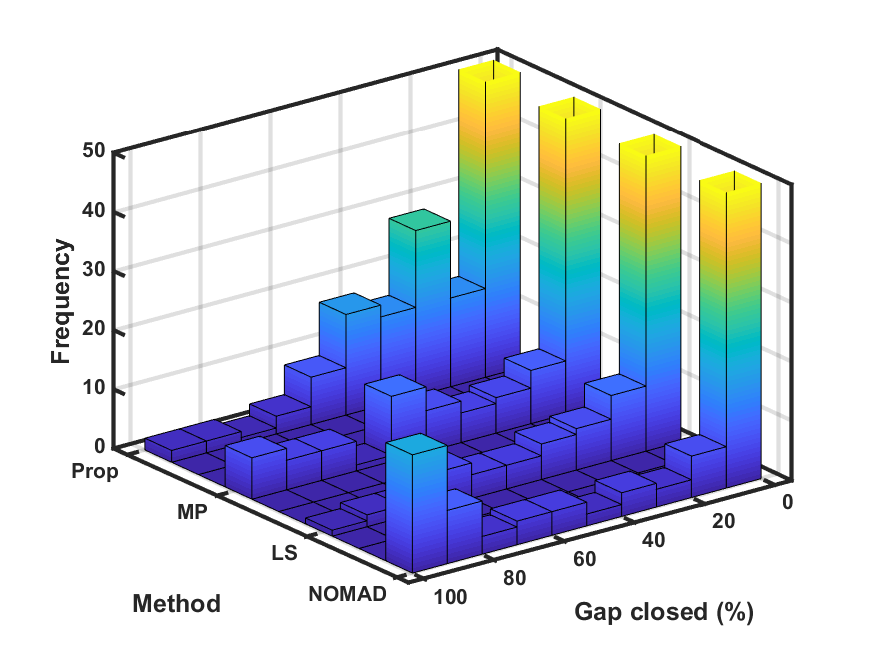}
\caption{NOMAD 8D}
\label{fig:histogram_gap_closed_nomad_8D}
\end{subfigure}
\begin{subfigure}[b]{0.24\textwidth}
\includegraphics[width=\textwidth]{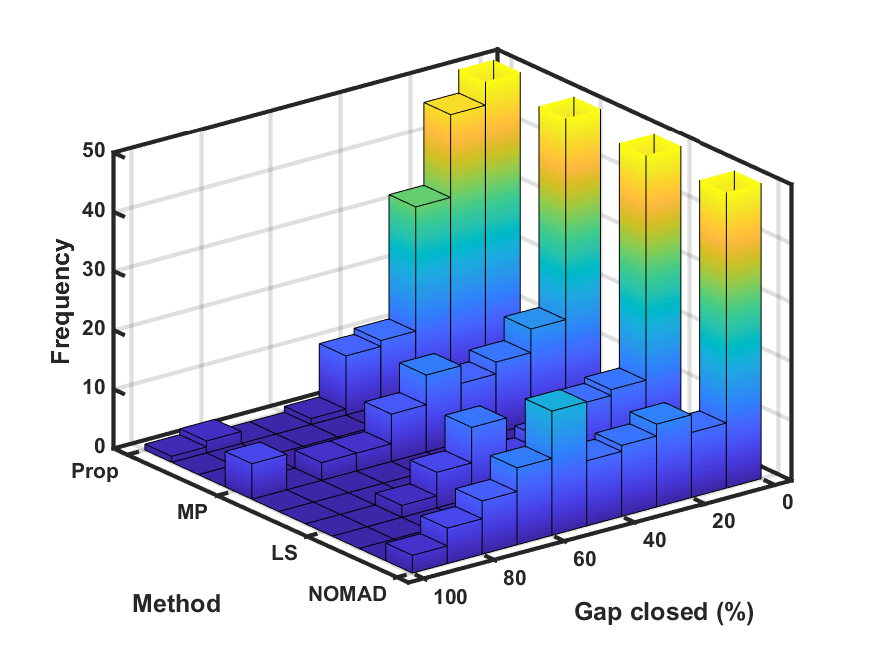}
\caption{NOMAD 16D}
\label{fig:histogram_gap_closed_nomad_16D}
\end{subfigure}
\caption{Histograms of optimality gap closed for instances in 2, 4, 8, and 16 dimensions. Top row: PSO comparison.  Bottow row: NOMAD comparison.}
\label{fig:histogram_gap_closed}
\end{figure}

Figure~\ref{fig:mean_gap_closed} shows the mean optimality gap closed for instances that were not already deemed solved. For Figure~\ref{fig:mean_gap_closed_pso}, we used PSO as the BBO solver to generate elite solutions and to perform solution polishing in $D$ dimensions.  For Figure~\ref{fig:mean_gap_closed_nomad}, we used NOMAD for these purposes. When PSO is used as the BBO solver, Figure~\ref{fig:mean_gap_closed_pso} shows a clear ordering of the solution polishing methods with the propeller method yielding the most improvement on average, followed by the multipoint method, line segment search, and PSO.  Equally notable, we see that solution polishing appears to be less effective as the dimension increases since the mean gap closed decreases as we double the number of dimensions.

Figure~\ref{fig:mean_gap_closed_nomad} shows how the solution polishing methods compare when NOMAD is the BBO solver. Whereas the maximum mean gap closed (i.e., the maximum value on the Figure~\ref{fig:mean_gap_closed_pso} y-axis) was over 60\% when PSO generated the elite solutions, the maximum mean gap closed in Figure~\ref{fig:mean_gap_closed_nomad} was just below 35\% when NOMAD generated the elite solutions, suggesting that NOMAD found better solutions than PSO as less improvement was possible.  Meanwhile, while $D$-dimensional solution polishing with PSO performed worse than all one-dimensional solution polishing methods, the same cannot be said for NOMAD. Figure~\ref{fig:mean_gap_closed_nomad} shows that $D$-dimensional solution polishing with NOMAD outperforms, on average, all one-dimensional solution polishing methods in 8 and 16 dimensions.  In 2 and 4 dimensions, however, Algorithm~\ref{algo:generate_p_curve} using the propeller curve performs the best overall unsolved instances. 

Figure~\ref{fig:histogram_gap_closed} depicts histograms of the optimality gap closed for each BBO solver and thus furnishes more granular information about the results.  The maximum frequency shown is capped at 50 to improve readability and comparison.
Consistent with the average performance shown in Figure~\ref{fig:mean_gap_closed}, we see the general trend that as the number of dimensions increases, the histogram bin heights decrease for ``Gap closed (\%)'' above 50\%, revealing that the solution polishing methods cannot close as much of the optimality gap. 
Figure~\ref{fig:histogram_gap_closed_pso_2D} shows that the propeller and multipoint methods are able to close nearly 100\% of the optimality gap in over 30\% of the unsolved instances when PSO generates the elite solutions. Interesting, in 8D, NOMAD is nearly able to perfectly polish several elite solutions as shown in Figure~\ref{fig:histogram_gap_closed_nomad_8D}.

\section{Conclusions and future research directions} \label{sec:conclusions}

We introduced two methods for performing solution polishing along one-dimensional curves through one or more elite solutions. To generate these smooth curves, we presented a convex quadratic program to solve a discretized optimal control problem that conceptually minimizes the total acceleration of a steel ball, which moves frictionlessly without any external forces through the elite solutions. 
Our computational experiments suggested that solution polishing on a well-chosen curve is competitive with using a state-of-the-art DFO solver for solution polishing.
Meanwhile, traditional ``straight linking'' applied to all combinations of elite solutions was inferior to both solution polishing methods on a curve.
While NOMAD's high-dimensional solution polishing searches were most effect in 8 and 16 dimensions, Matlab's PSO solver exhibited inferior performance to all other methods. 
LineWalker was capable of finding improving solutions with a reasonable number of function evaluations.

As for future research directions, there are many ways to generate a curve through multiple elite solutions. It would be interesting to explore the impact of the step size and other parameters.  Since there are dozens of BBO solvers, one could perform similar numerical experiments with a broader set of BBO solvers. One could also mix and match BBO solvers whereby one BBO solver generates a set of elite solutions and another BBO solver attempts to polish them.
As another alternative, one could vary (as a function of the dimension $D$) the number of elite solutions and the maximum number function evaluations used for polishing to analyze how the solution quality improves. These can serve as hyperparameters that we can tune for solution polishing.
Another future direction could be to utilize the propeller search as a direction finding subproblem in an iterative descent algorithm. At a given point, one could for example optimize along the propeller curve to search for an improved solution. If an improved solution is found, then this could be considered as a descent direction from the current point, along which we could apply LineWalker to find an appropriate step length. However, this approach would only make sense in a setting with a more generous function evaluation budget, and this is not within the scope of this paper.


\section*{Data Availability Statement} \label{sec:Data_Availability}

Test instances are available at \url{https://www.sfu.ca/~ssurjano/},  \url{https://github.com/luclaurent/optiGTest/tree/master/unConstrained}, and \cite[Figure 7]{duarte2011path}.

\small
\bibliographystyle{plainnat}
\bibliography{linewalker_refs,path_relinking_refs}

\newpage
\section*{Appendix} \label{sec:appendix} 

\subsection{Michalewicz function: Optimal objective function values} \label{sec:michal_objvals}

Because the global minimum objective function value for the \href{https://www.sfu.ca/~ssurjano/michal.html}{Michalewicz} (``Michal'') function 
$
-\sum_{i=1}^D \sin(x_i)\sin^{20}\Big(\tfrac{i x_i^2}{\pi}\Big)
$
changes with the dimension $D$ and because this global minimum value is not reported in the literature for all $D$, we performed a computational experiment to produce provably global optimal objective function values up to a small degree of numerical error.  The final caveat stems from the potential for numerical issues in global solvers as we explain below.

\begin{table} [h!]
\centering
\caption{Putative (see \cite[Table 2]{vanaret2020certified}) and certified global minimum objective function values for the Michalewicz function.}
\label{tab:michal_global_minimum_objective_function_values}	
\begin{tabular}{crr}														
\toprule					
\textbf{Dimension}	&	\textbf{Putative} $f^*$	&	\textbf{Certified} $f^*$	\\
\midrule
1	&	-0.69593	&	-0.80130341	\\
2	&	-1.69457	&	-1.80130341	\\
3	&	-2.69321	&	-2.76039468	\\
4	&	-3.69185	&	-3.69885710	\\
5	&	-4.69049	&	-4.68765818	\\
6	&	-5.68913	&	-5.68765818	\\
7	&	-6.68777	&	-6.68088531	\\
8	&	-7.68641	&	-7.66375735	\\
9	&	-8.68505	&	-8.66015172	\\
10	&	-9.68369	&	-9.66015172	\\
11	&	-10.68233	&	-10.65748226	\\
12	&	-11.68097	&	-11.64957500	\\
13	&	-12.67961	&	-12.64781799	\\
14	&	-13.67825	&	-13.64781799	\\
15	&	-14.67689	&	-14.64640019	\\
16	&	-15.67553	&	-15.64186482	\\
\bottomrule
\end{tabular}													
\end{table}
 
\href{www.gurobi.com}{Gurobi 11.0.0} provides global optimization capabilities, but admittedly with numerical error as it constructs piecewise linear approximations of nonconvex nonlinear functions (e.g., the sine function for Michal).  One can set the tolerance for these approximations to be quite small, but some degree of error is inevitable.  Thus, we took the following approach. For $D=1,\dots,16$, we used Gurobi 11.0.0 to generate what it deems a globally optimal solution.  We then passed this solution to CONOPT 4.1, IPOPT 3.11, and SNOPT 7.2 (we experienced errors with Knitro 13.1), which do not rely on piecewise linear approximations, to perform rigorous local optimization to ensure that optimality conditions were satisfied.  Fortunately, using Gurobi's optimal solution as an initial solution, the three local solvers were unanimous in their answer for every dimension $D$.  As an extra safeguard, we passed the resulting optimal solution $\v{x}^*$ into Matlab and evaluated $f(\v{x}^*)$ to ensure that it reached the same function value. Table~\ref{tab:michal_global_minimum_objective_function_values} shows the results and reveals that the formula $- 0.99864D + 0.30271$ presented in \cite[Table 2]{vanaret2020certified} to compute the Michal function's ``putative'' minimum objective function value is incorrect. 

\subsection{NOMAD profiles} \label{sec:NOMAD_profiles}

In this section we present the empirical convergence behavior of NOMAD over the 19 test functions in various dimensions, showing the number of black-box function evaluations to reach a point where no more improvement of its best objective value is observed. The results serve two purposes: i) they provide insight into our decision for setting $50D$ as the maximum number of function evaluations of existing methods before we perform solution polishing; and ii) the potential or challenge for our approach to be effective.

Our results were obtained by running NOMAD v4.3.1 100 times for each test function with a randomly selected starting point. Except for the maximum number of function evaluations being set to 10,000, we used the default termination criteria. Each figure contains 100 plots corresponding to each starting point showing the best objective value found so far as the number of black-box function evaluations increases.

\begin{figure}[h]
\includegraphics[width=\textwidth]{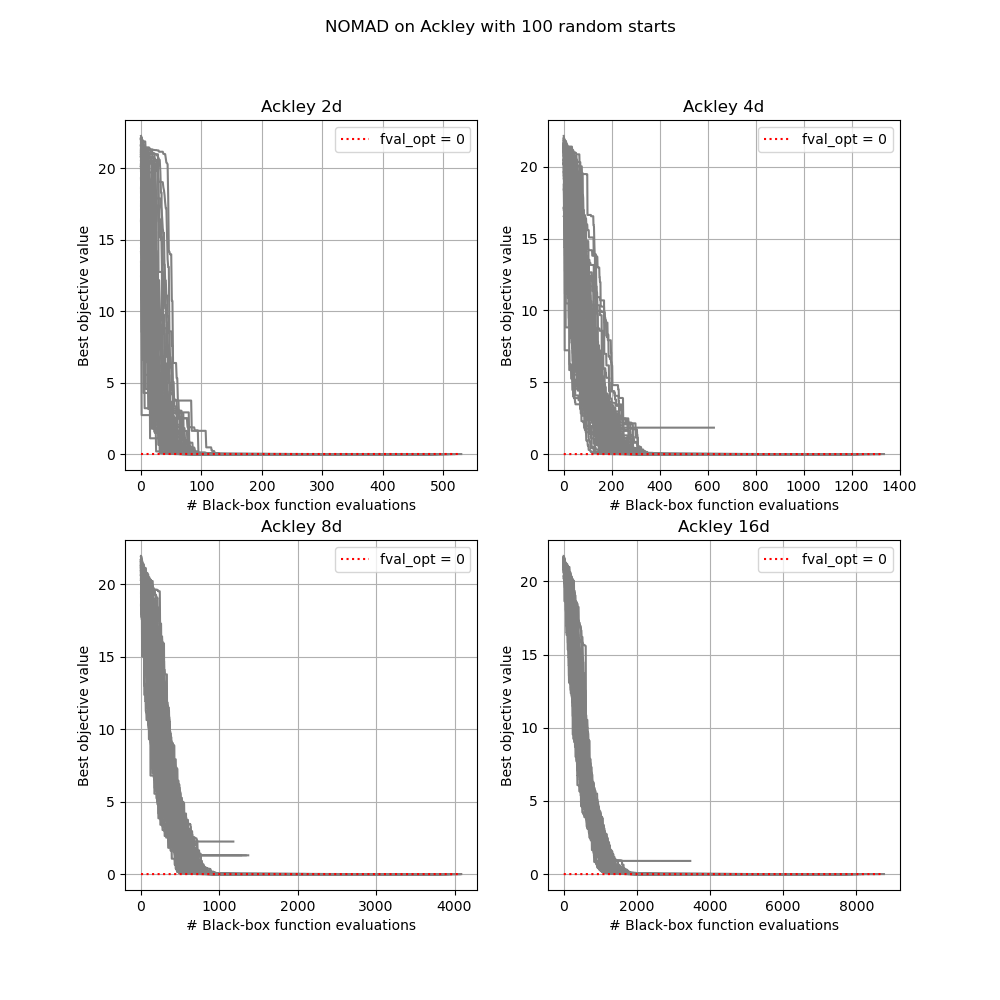}
\caption{\href{http://www.sfu.ca/~ssurjano/ackley.html}{ackley}}
\label{fig:NOMAD_profiles_ackley}
\end{figure}
\begin{figure}[h]
\includegraphics[width=\textwidth]{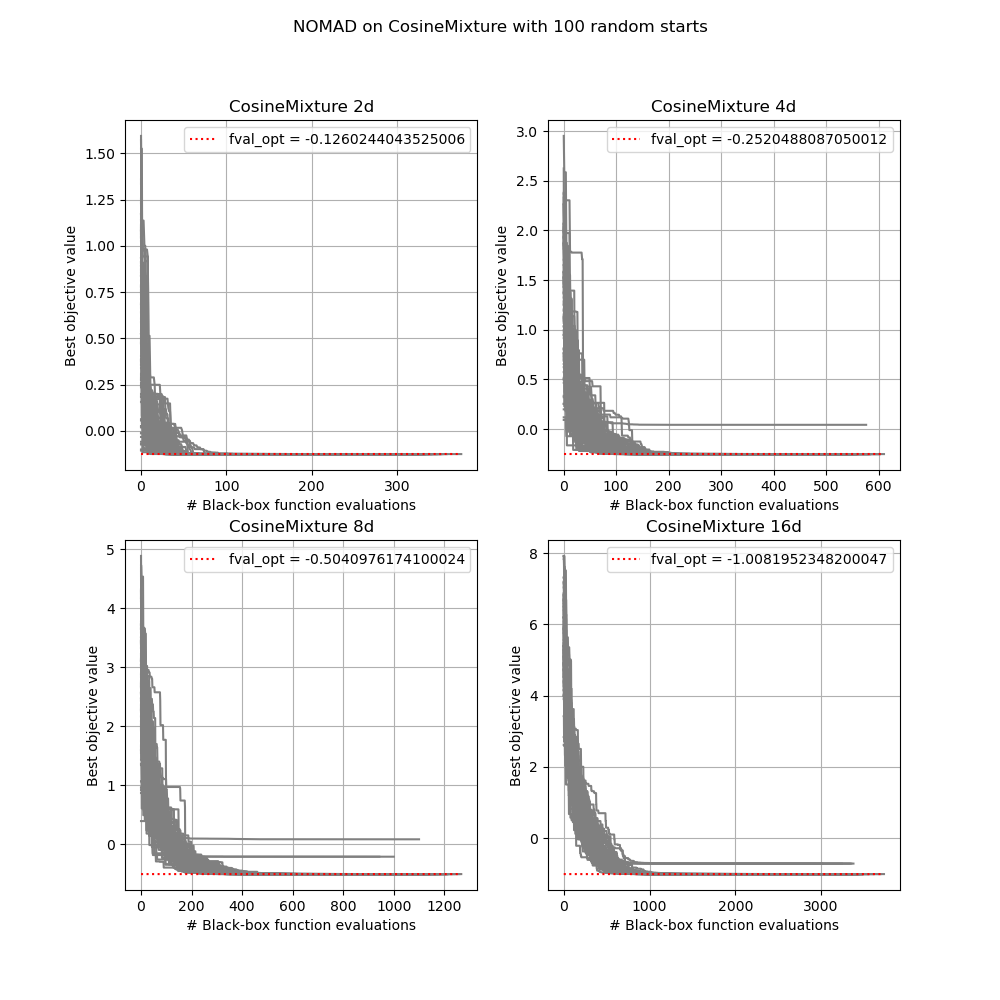}
\caption{\href{https://github.com/luclaurent/optiGTest/blob/master/unConstrained/funCosineMixture.m}{cosineMixture}}
\label{fig:NOMAD_profiles_cosineMixture}
\end{figure}
\begin{figure}[h]
\includegraphics[width=\textwidth]{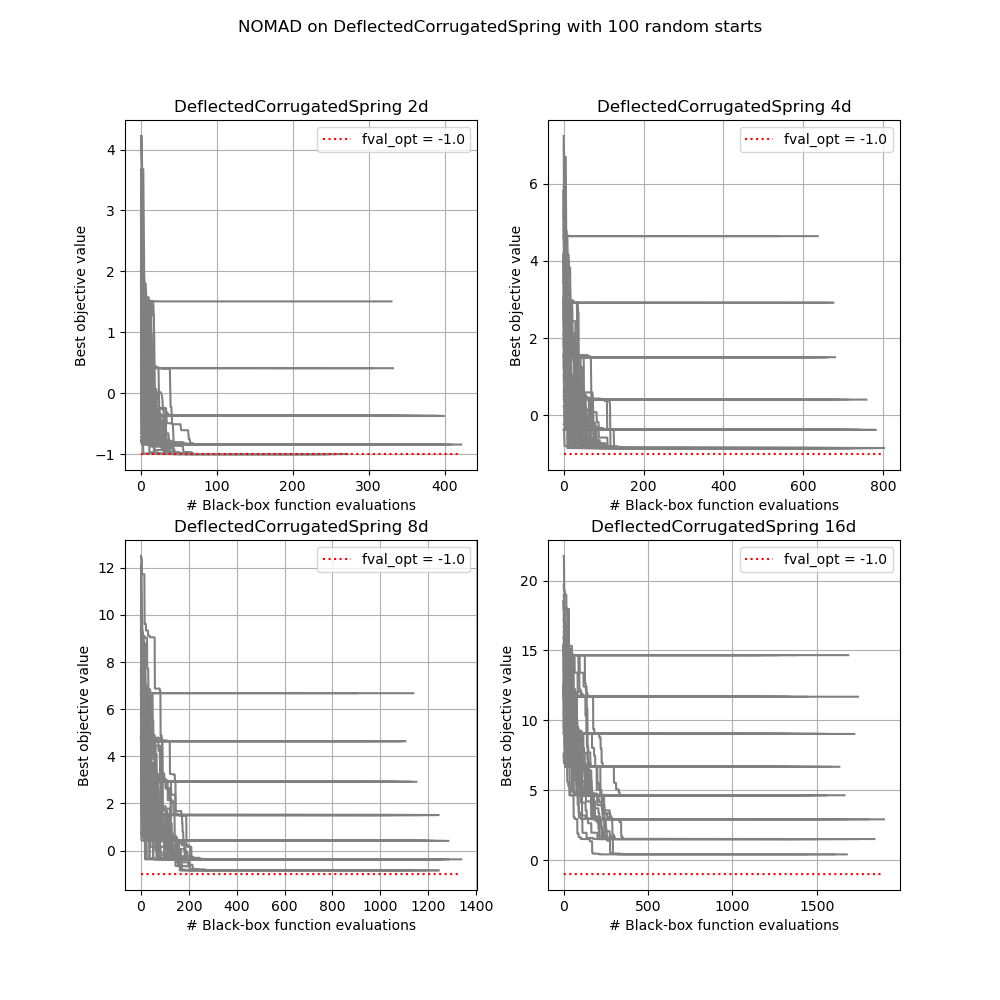}
\caption{\href{https://github.com/luclaurent/optiGTest/blob/master/unConstrained/funDeflectedCorrugatedSpring.m}{deflectedCorrugatedSpring}}
\label{fig:NOMAD_profiles_deflectedCorrugatedSpring}
\end{figure}
\begin{figure}[h]
\includegraphics[width=\textwidth]{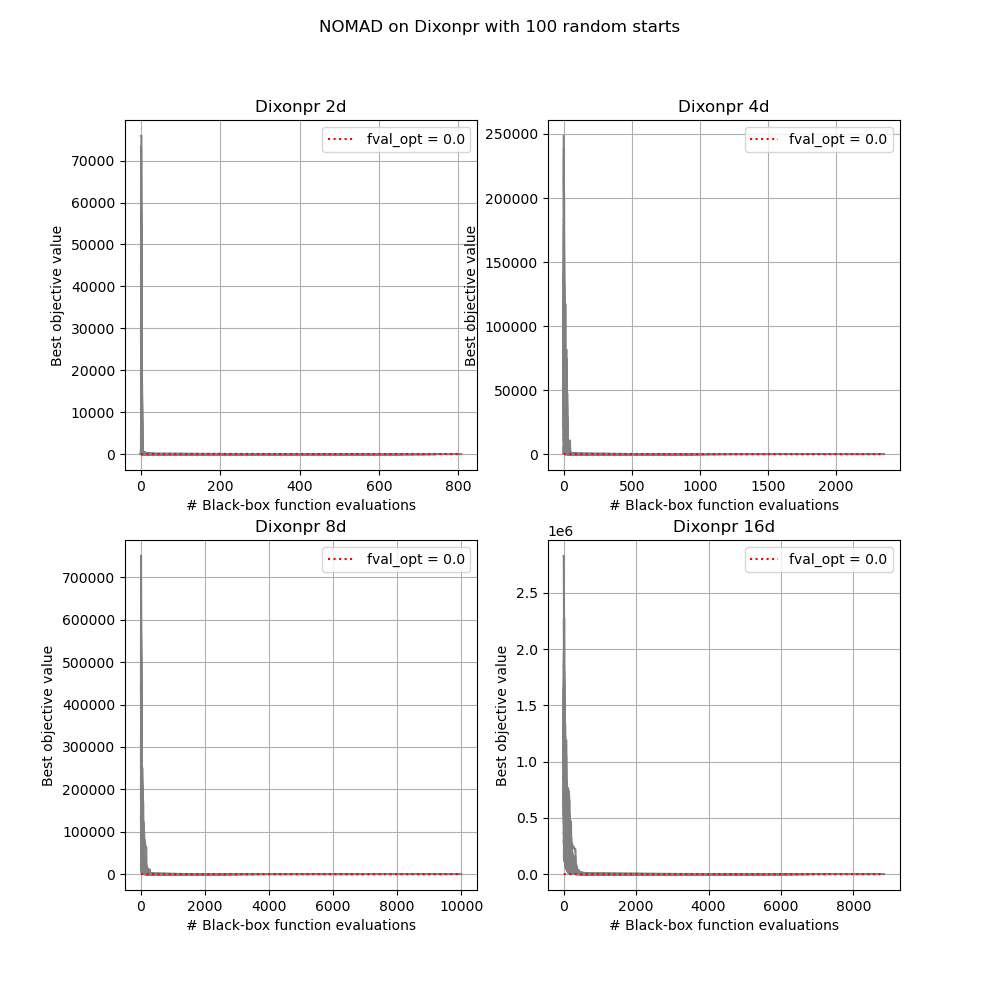}
\caption{\href{http://www.sfu.ca/~ssurjano/dixonpr.html}{dixonpr}}
\label{fig:NOMAD_profiles_dixonpr}
\end{figure}
\begin{figure}[h]
\includegraphics[width=\textwidth]{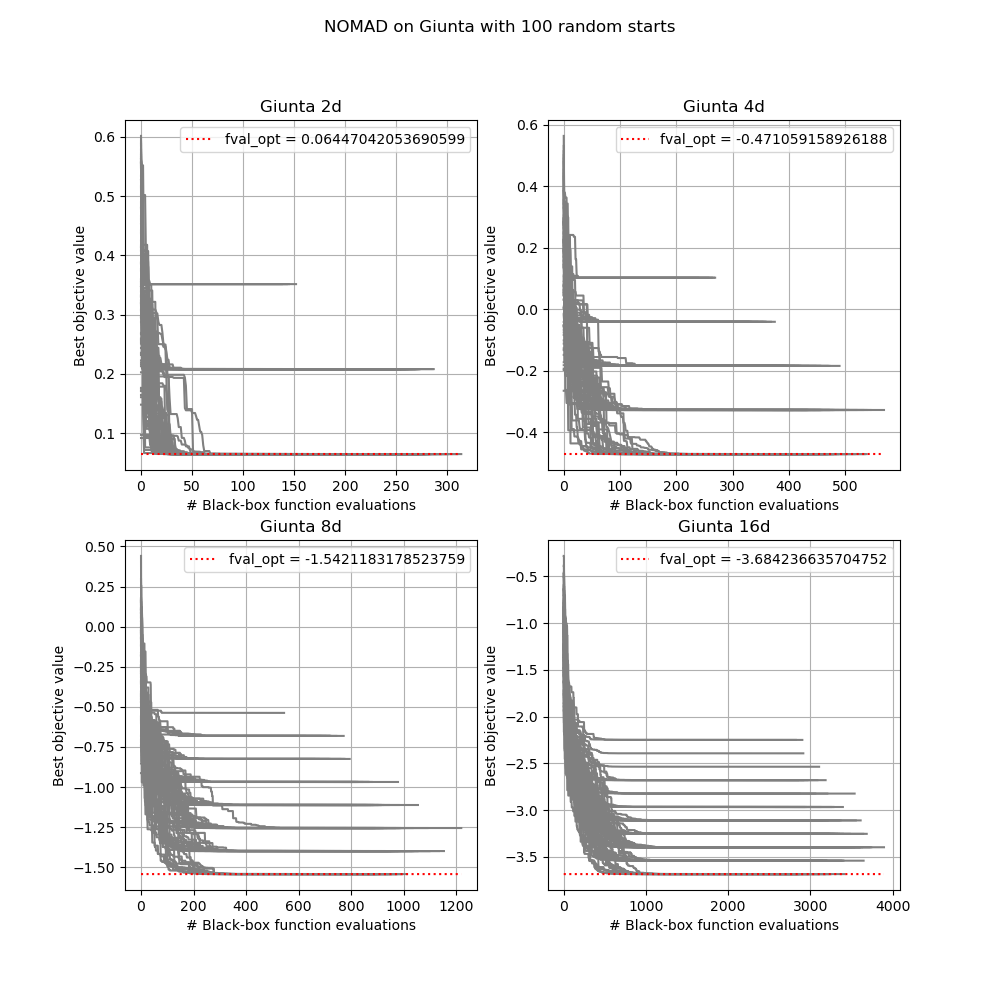}
\caption{\href{https://github.com/luclaurent/optiGTest/blob/master/unConstrained/funGiunta.m}{giunta}}
\label{fig:NOMAD_profiles_giunta}
\end{figure}
\begin{figure}[h]
\includegraphics[width=\textwidth]{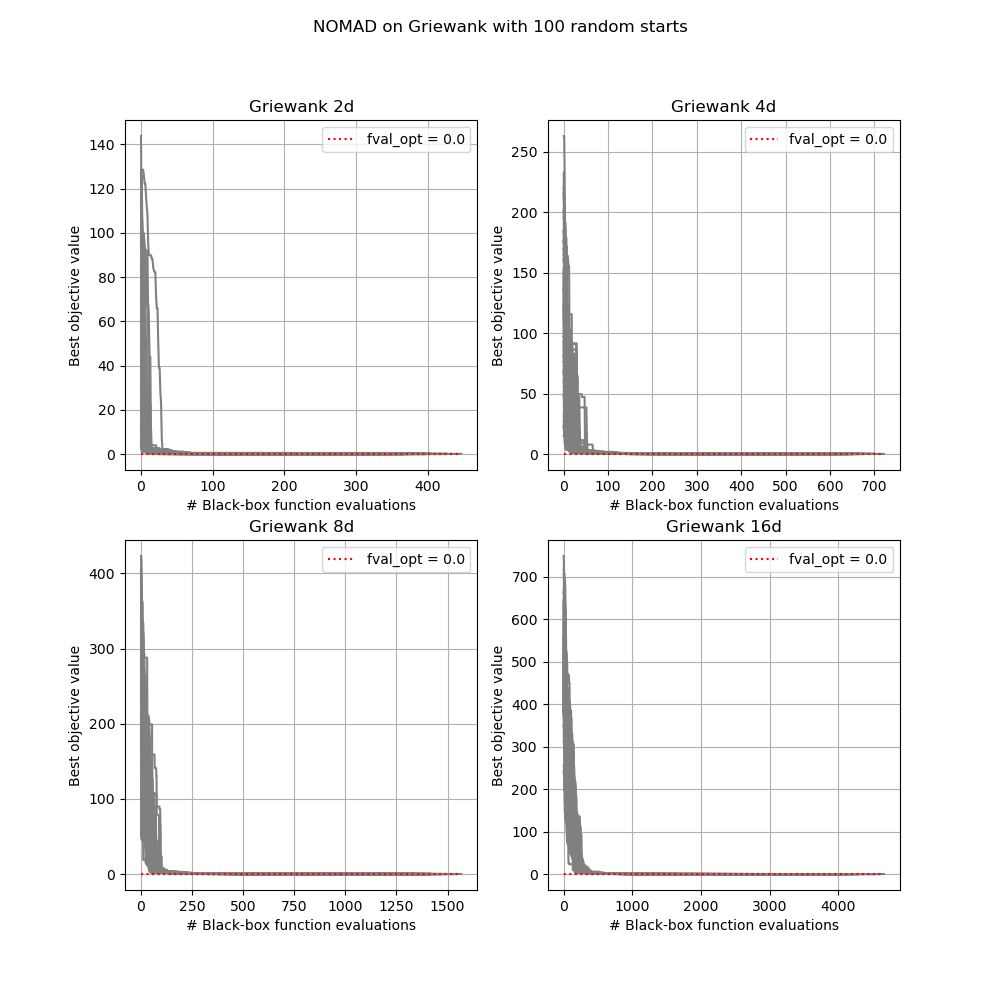}
\caption{\href{http://www.sfu.ca/~ssurjano/griewank.html}{griewank}}
\label{fig:NOMAD_profiles_griewank}
\end{figure}
\begin{figure}[h]
\includegraphics[width=\textwidth]{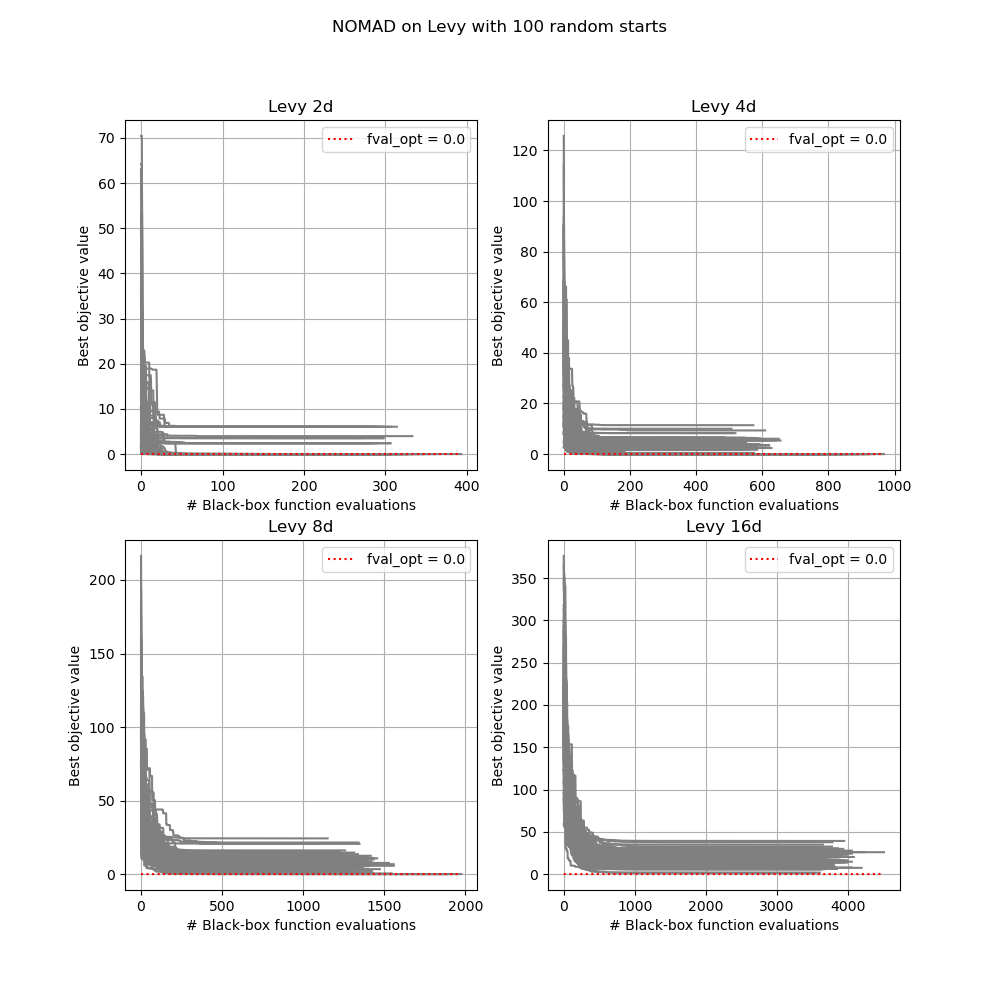}
\caption{\href{http://www.sfu.ca/~ssurjano/levy.html}{levy}}
\label{fig:NOMAD_profiles_levy}
\end{figure}
\begin{figure}[h]
\includegraphics[width=\textwidth]{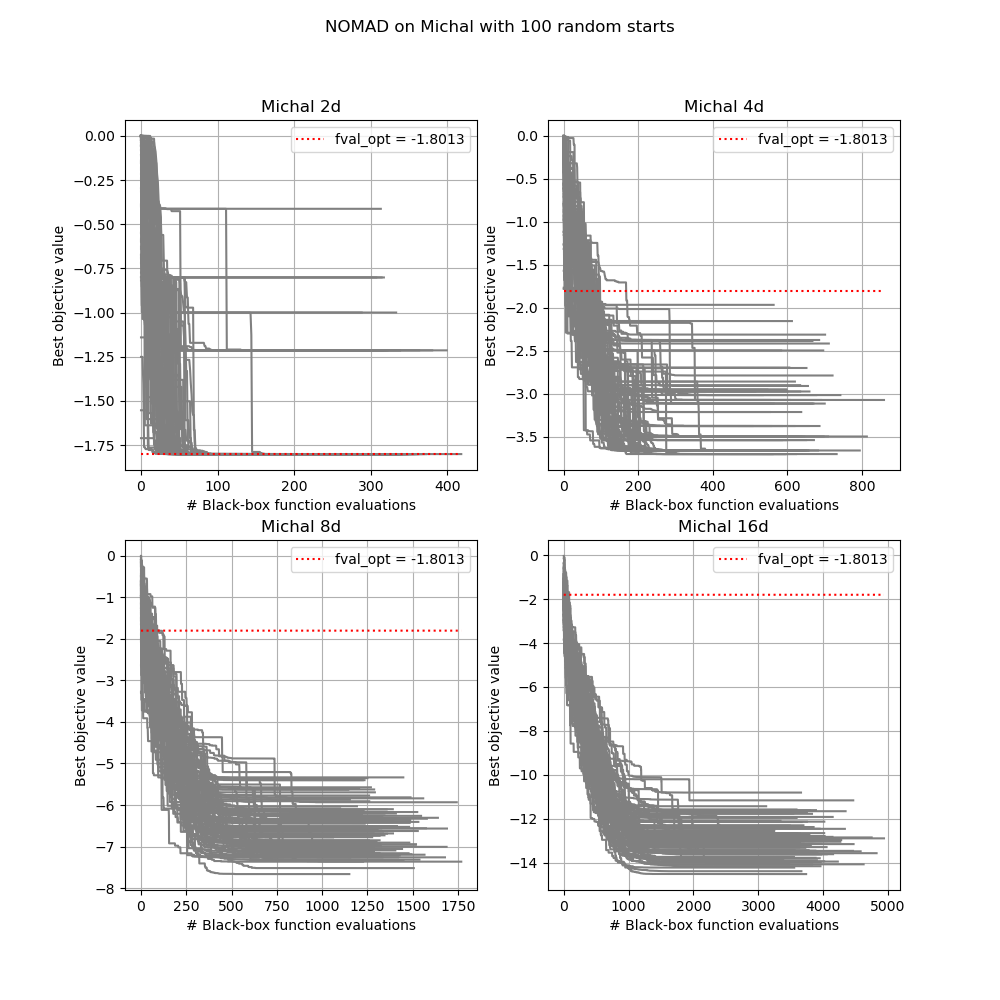}
\caption{\href{http://www.sfu.ca/~ssurjano/michal.html}{michal}}
\label{fig:NOMAD_profiles_michal}
\end{figure}
\begin{figure}[h]
\includegraphics[width=\textwidth]{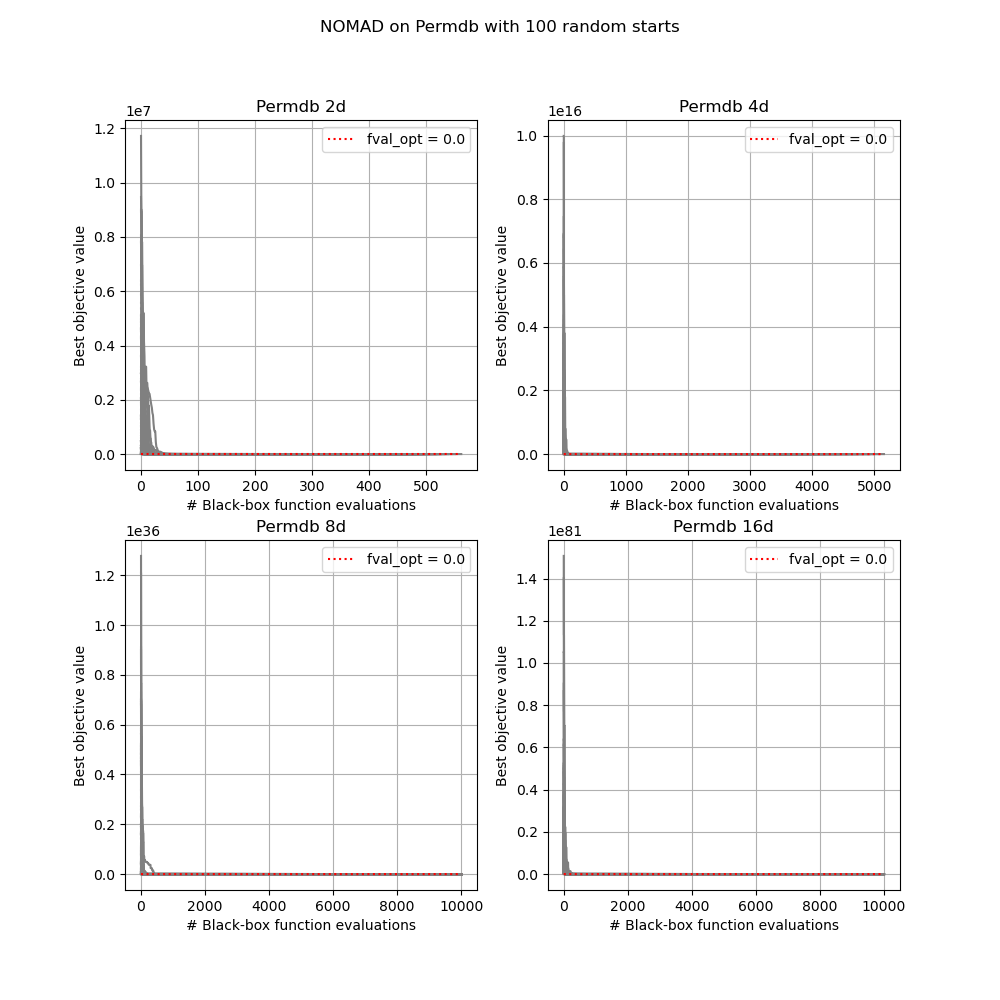}
\caption{\href{http://www.sfu.ca/~ssurjano/permdb.html}{permdb}}
\label{fig:NOMAD_profiles_permdb}
\end{figure}
\begin{figure}[h]
\includegraphics[width=\textwidth]{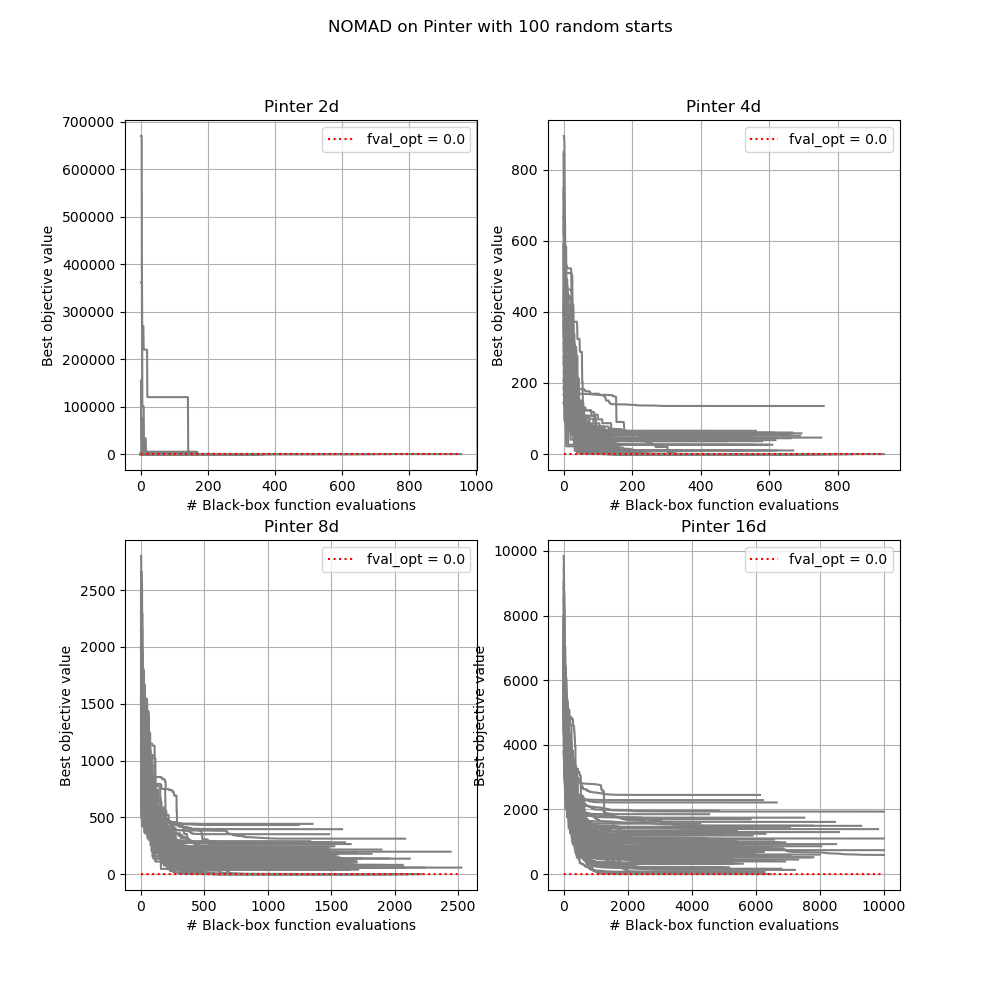}
\caption{\href{http://infinity77.net/global_optimization/test_functions_nd_P.html\#go_benchmark.Pinter}{pinter}}
\label{fig:NOMAD_profiles_pinter}
\end{figure}
\begin{figure}[h]
\includegraphics[width=\textwidth]{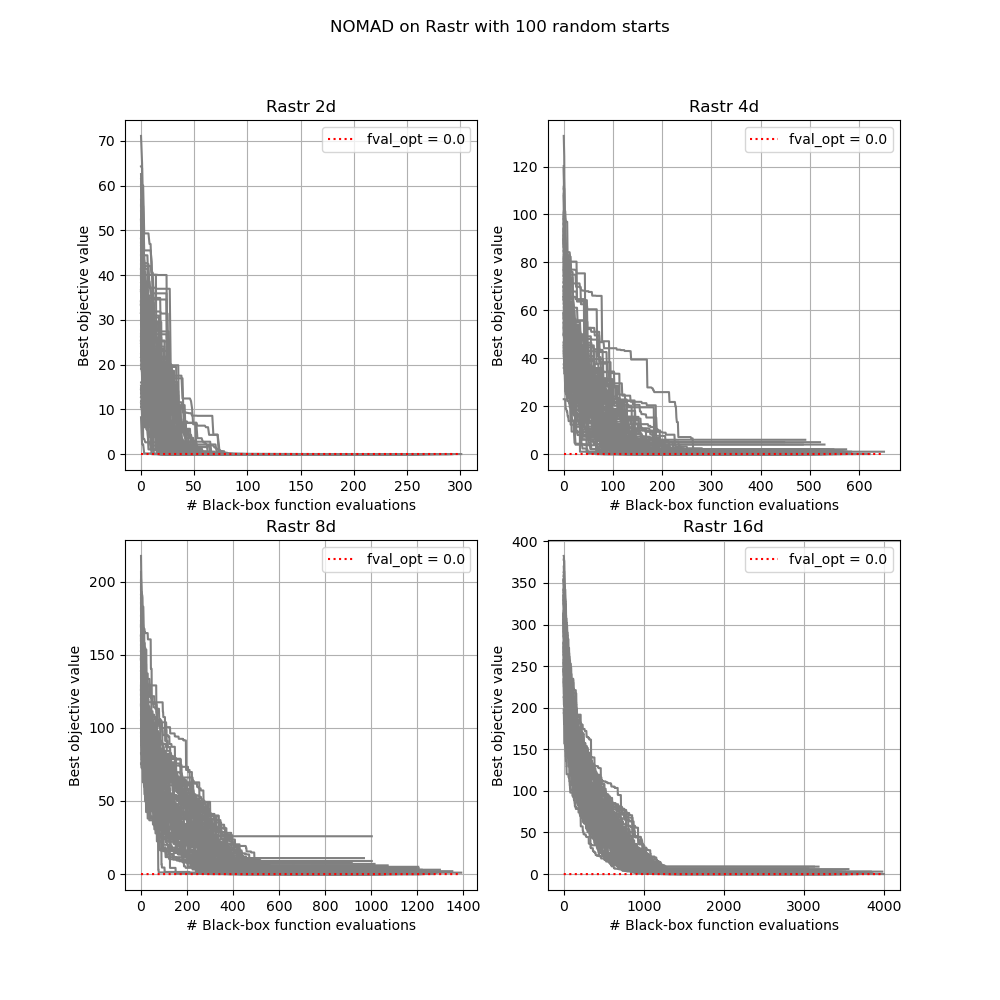}
\caption{\href{http://www.sfu.ca/~ssurjano/rastr.html}{rastr}}
\label{fig:NOMAD_profiles_rastr}
\end{figure}
\begin{figure}[h]
\includegraphics[width=\textwidth]{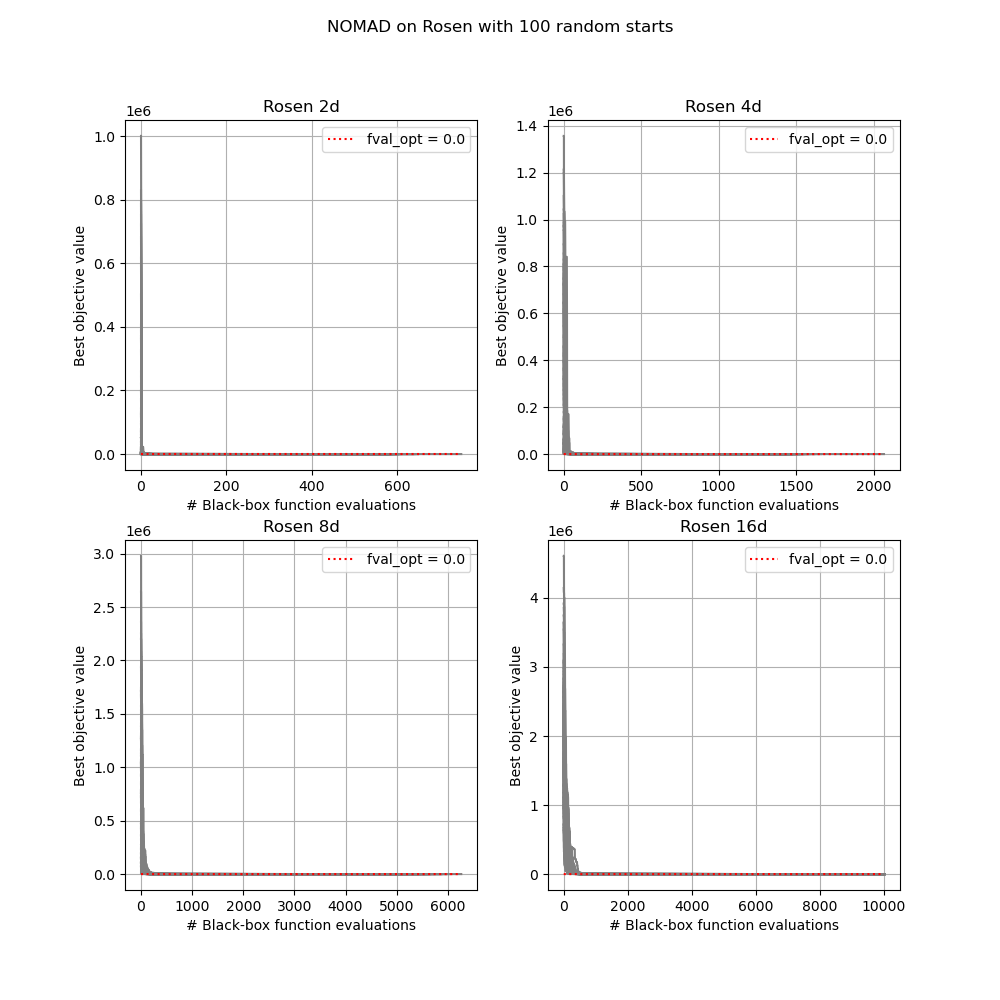}
\caption{\href{http://www.sfu.ca/~ssurjano/rosen.html}{rosen}}
\label{fig:NOMAD_profiles_rosen}
\end{figure}
\begin{figure}[h]
\includegraphics[width=\textwidth]{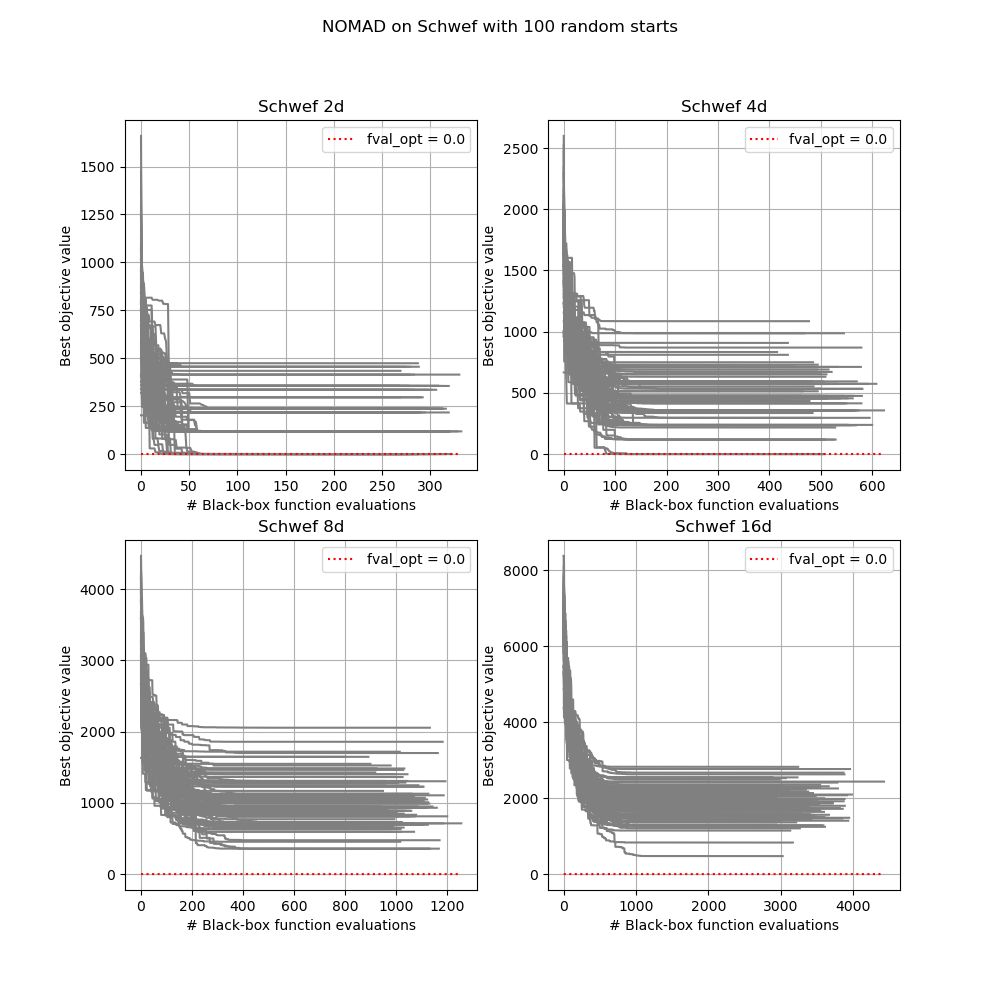}
\caption{\href{http://www.sfu.ca/~ssurjano/schwef.html}{schwef}}
\label{fig:NOMAD_profiles_schwef}
\end{figure}
\begin{figure}[h]
\includegraphics[width=\textwidth]{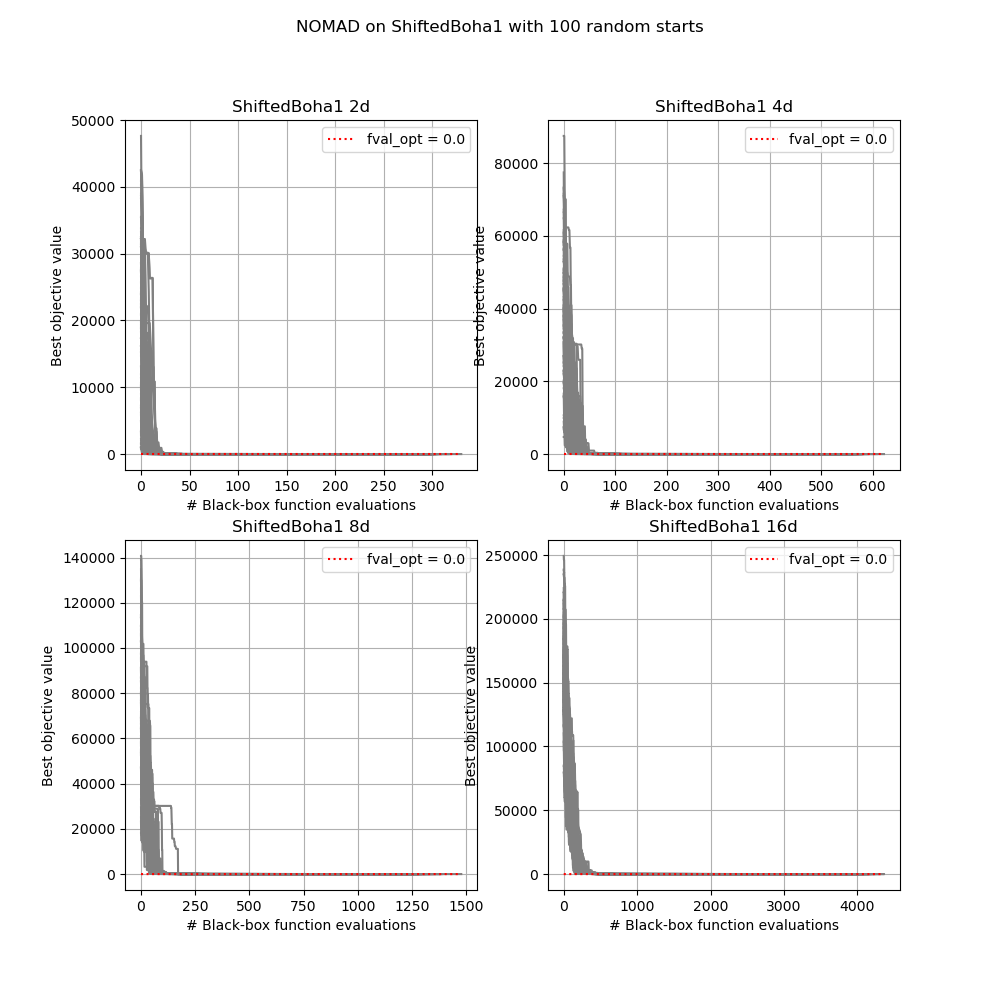}
\caption{\href{https://link.springer.com/article/10.1007/s00500-010-0650-7/figures/7}{shiftedBoha1}}
\label{fig:NOMAD_profiles_shiftedBoha1}
\end{figure}
\begin{figure}[h]
\includegraphics[width=\textwidth]{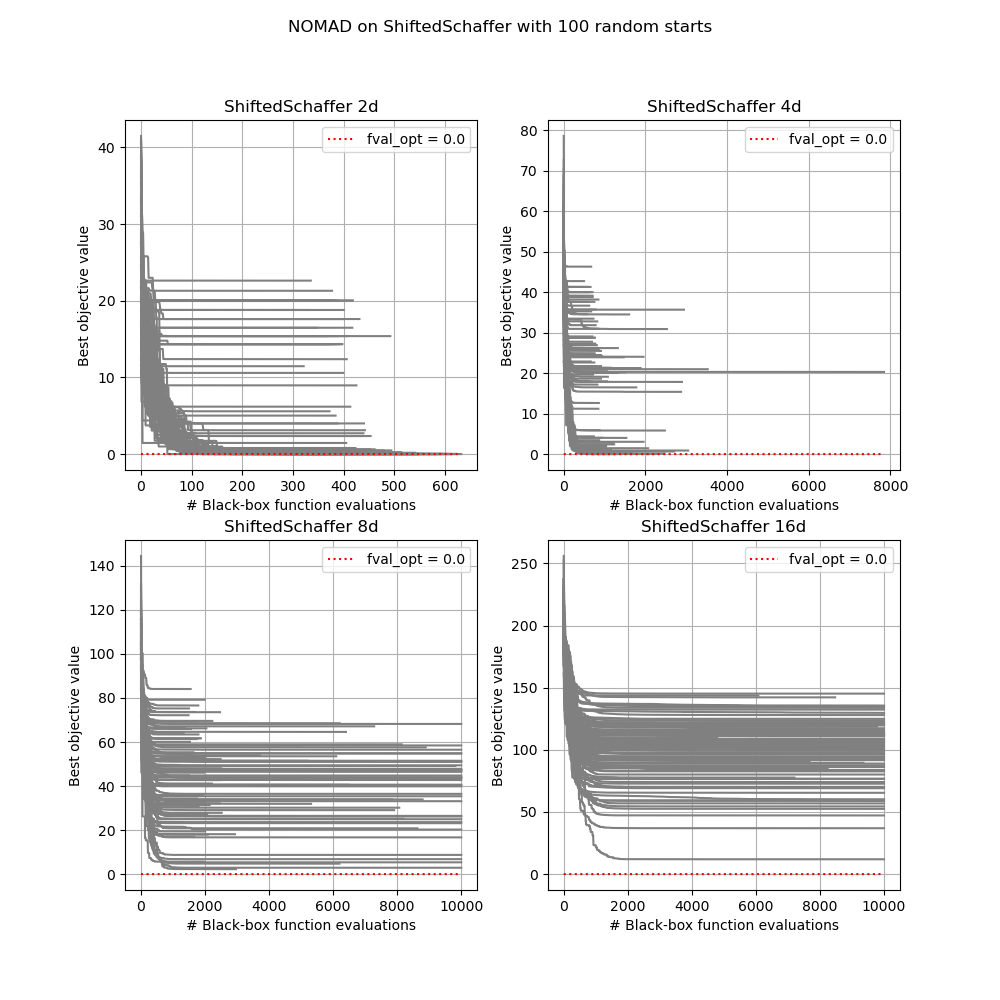}
\caption{\href{https://link.springer.com/article/10.1007/s00500-010-0650-7/figures/7}{shiftedSchaffer}}
\label{fig:NOMAD_profiles_shiftedSchaffer}
\end{figure}
\begin{figure}[h]
\includegraphics[width=\textwidth]{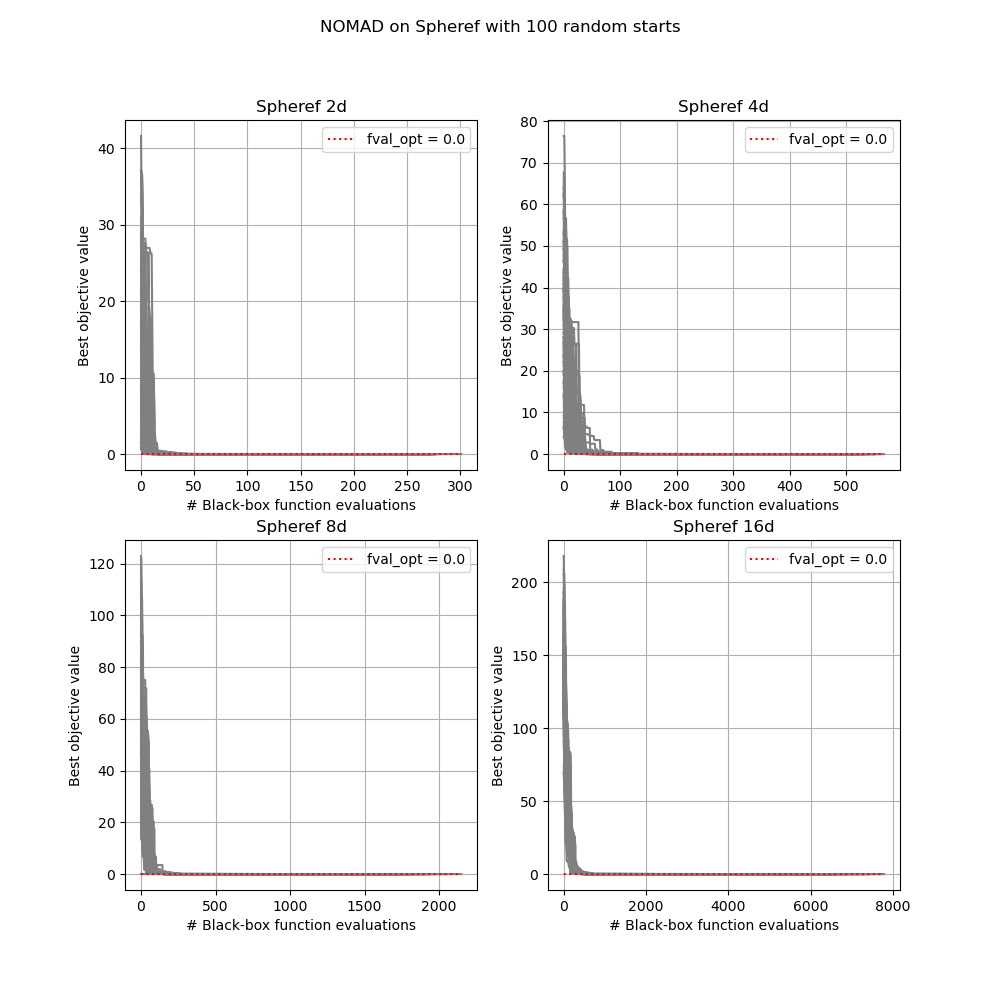}
\caption{\href{http://www.sfu.ca/~ssurjano/spheref.html}{spheref}}
\label{fig:NOMAD_profiles_spheref}
\end{figure}
\begin{figure}[h]
\includegraphics[width=\textwidth]{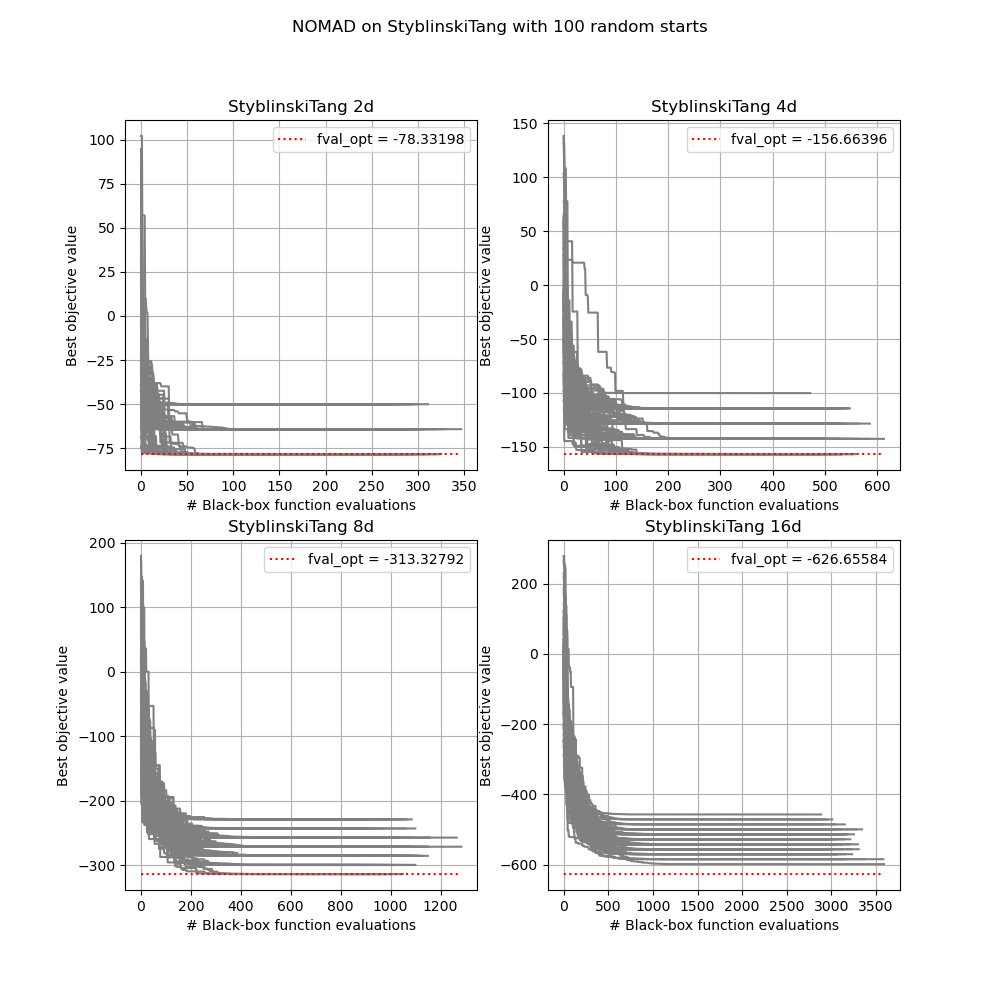}
\caption{\href{http://www.sfu.ca/~ssurjano/stybtang.html}{stybtang}}
\label{fig:NOMAD_profiles_stybtang}
\end{figure}
\begin{figure}[h]
\includegraphics[width=\textwidth]{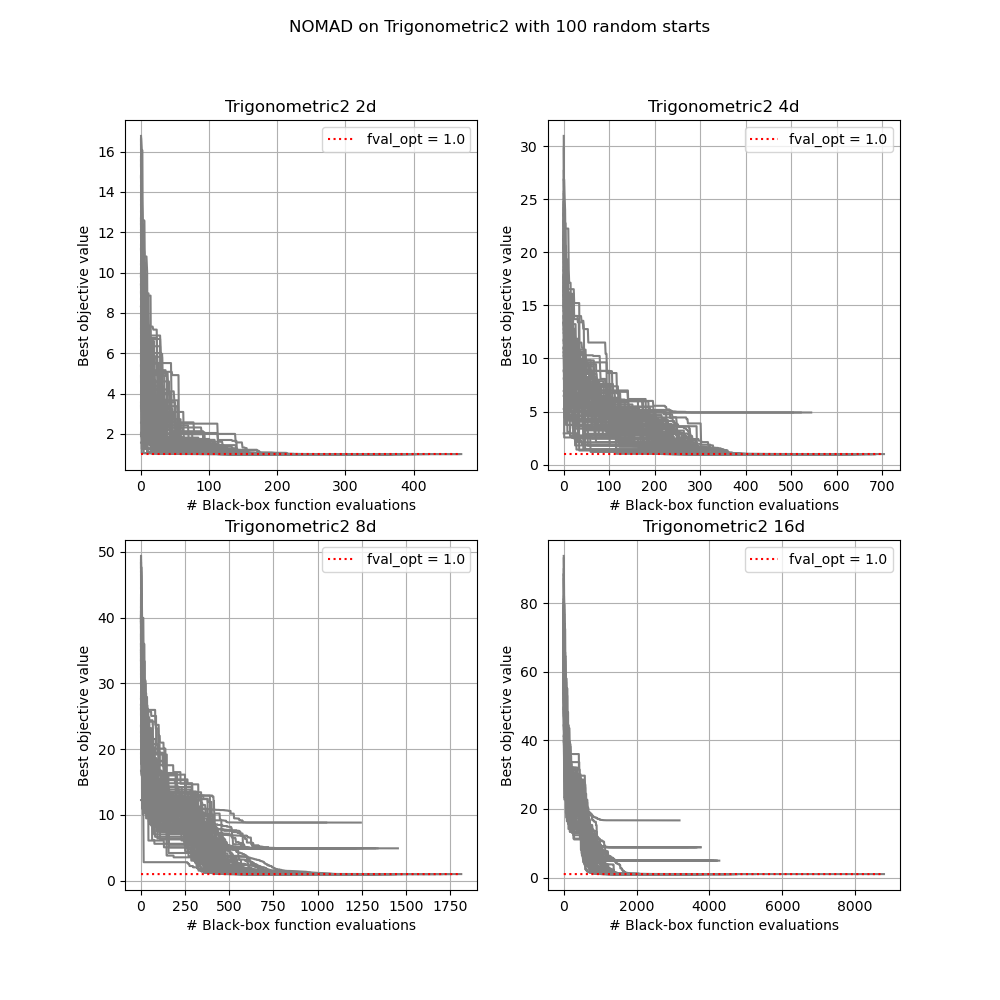}
\caption{\href{https://github.com/luclaurent/optiGTest/blob/master/unConstrained/funTrigonometric2.m}{trigonometric2}}
\label{fig:NOMAD_profiles_trigonometric2}
\end{figure}
\begin{figure}[h]
\includegraphics[width=\textwidth]{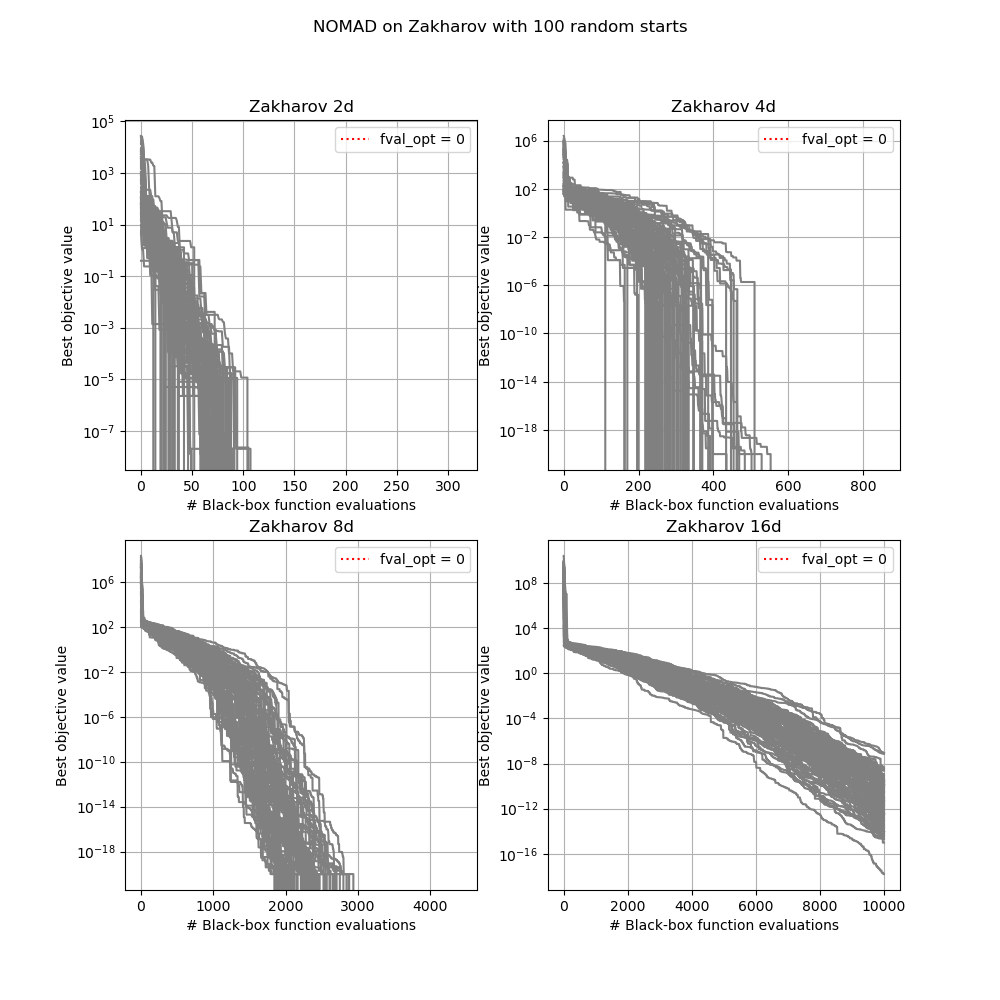}
\caption{\href{http://www.sfu.ca/~ssurjano/zakharov.html}{zakharov}}
\label{fig:NOMAD_profiles_zakharov}
\end{figure}

\end{document}